\documentclass{amsart}
\usepackage{amssymb}
\usepackage{amscd}
\usepackage{verbatim}
\usepackage{epsfig}

\newcommand{\pa}{\partial}
\newcommand{\CI}{C^\infty}
\newcommand{\dCI}{\dot C^\infty}
\newcommand{\Hom}{\operatorname{Hom}}
\newcommand{\supp}{\operatorname{supp}}
\newcommand{\Op}{\operatorname{Op}}

\newcommand{\ep}{\epsilon}

\newcommand{\WFsc}{\operatorname{WF}_{\scl}}

\newcommand{\Ran}{\operatorname{Ran}}
\newcommand{\Span}{\operatorname{Span}}
\newcommand{\sX}{\mathsf{X}}

\newcommand{\Id}{\operatorname{Id}}
\newcommand{\id}{\operatorname{id}}

\newcommand{\bl}{{\mathrm b}}
\newcommand{\scl}{{\mathrm{sc}}}

\newcommand{\Psisc}{\Psi_\scl}
\newcommand{\Diff}{\mathrm{Diff}}
\newcommand{\Diffsc}{\Diff_\scl}

\newcommand{\RR}{\mathbb{R}}
\newcommand{\Cx}{\mathbb{C}}
\newcommand{\NN}{\mathbb{N}}
\newcommand{\R}{\mathbb{R}}
\newcommand{\sphere}{\mathbb{S}}

\newcommand{\cS}{\mathcal S}

\newcommand{\cF}{\mathcal F}
\newcommand{\cL}{\mathcal L}

\newcommand{\cU}{\mathcal U}

\newcommand{\cX}{\mathcal X}
\newcommand{\cY}{\mathcal Y}
\newcommand{\cR}{\mathcal R}

\newcommand{\Tb}{{}^{\bl}T}

\newcommand{\Tsc}{{}^{\scl}T}

\newcommand{\Vb}{{\mathcal V}_{\bl}}
\newcommand{\Vsc}{{\mathcal V}_{\scl}}

\newcommand{\scH}{{}^{\scl}H}
\newcommand{\Hsc}{H_{\scl}}

\newcommand{\Hscd}{\dot H_{\scl}}
\newcommand{\Hscb}{\bar H_{\scl}}

\newcommand{\inter}{{\mathrm{int}}}
\newcommand{\Sym}{\mathrm{Sym}}
\newcommand{\be}[1]{\begin{equation}\label{#1}}
\newcommand{\ee}{\end{equation}}
\newcommand{\x}{p}
\newcommand{\X}{P}

\newcommand{\level}{\mathsf{c}}
\newcommand{\foliation}{\mathsf{x}}

\newcommand{\Ajd}{A_{j,\digamma}}
\newcommand{\Njd}{N_{j,\digamma}}

\newcommand{\bo}{\partial M} 
\newcommand{\zero}{^{(0)}}
\renewcommand{\r}[1]{(\ref{#1})} 
\newcommand{\mat}[4]{\left(\begin{array}{cc} #1 &#2\\#3 & #4 
\end{array}\right)}
\renewcommand{\d}{d} 
\newcommand{\dsymm}{\mathrm{d}^{\mathrm{s}}}
\newcommand{\dsymmw}{\mathrm{d}^{\mathrm{s}}_\digamma}
\newcommand{\dsymmY}{\mathrm{d}^{\mathrm{s}}_Y}
\newcommand{\dsymmsc}{\mathrm{d}^{\mathrm{s}}_\scl}
\newcommand{\dsymmscw}{\mathrm{d}^{\mathrm{s}}_{\scl,\digamma}}

\newtheorem{lemma}{Lemma}[section]
\newtheorem{prop}{Proposition}[section]
\newtheorem{proposition}{Proposition}[section]
\newtheorem{thm}{Theorem}[section]
\newtheorem{cor} {Corollary}[section]

\newtheorem*{thm*}{Theorem}
\newtheorem*{prop*}{Proposition}
\newtheorem*{cor*}{Corollary}
\newtheorem*{conj*}{Conjecture}
\numberwithin{equation}{section}
\theoremstyle{remark}
\newtheorem{rem}{Remark}[section]

\newtheorem*{rem*}{Remark}
\theoremstyle{definition}

\newtheorem*{Def*}{Definition}

\title[Local and global boundary rigidity]{Local and global boundary rigidity and the geodesic X-ray
  transform in the normal gauge}
\author[Plamen Stefanov, Gunther Uhlmann and Andras Vasy]{Plamen
  Stefanov, Gunther Uhlmann and Andr\'as Vasy}
\date{Revision: \today}
\address{Department of Mathematics, Purdue University, West Lafayette,
IN 47907-1395, U.S.A.}
\email{stefanov@math.purdue.edu}
\address{Department of Mathematics, University of Washington, 
Seattle, WA 98195-4350, U.S.A., and Institute for Advanced Study,
HKUST, Clear Water Bay, Hong Kong, China}
\email{gunther@math.washington.edu}
\address{Department of Mathematics, Stanford University, Stanford, CA
94305-2125, U.S.A.}
\email{andras@math.stanford.edu}
\thanks{The authors gratefully acknowledge partial support by the National Science Foundation.}
\subjclass{53C24,  53C65, 35R30, 35S05, 53C21}

\begin{document}
\begin{abstract}
In this paper we analyze the local and global boundary rigidity problem for general Riemannian
manifolds with boundary $(M,g)$. 
We show that the boundary distance function, i.e.\
$d_g|_{\pa M\times\pa M}$, known near a point $p\in \pa M$ at which $\pa M$ is strictly convex, determines $g$ in a suitable neighborhood of $p$ in  $M$, up to
the natural diffeomorphism invariance of the problem.

We also consider the closely related lens rigidity problem which is a more natural formulation if the boundary distance is not realized by unique minimizing geodesics. The lens relation measures the point and the direction of exit from $M$  of  geodesics issued from the boundary and the length of the geodesic. 
The lens rigidity problem is whether we can determine the metric up to isometry from the lens relation. We solve the lens rigidity problem under the assumption that there is a function on $M$ with suitable convexity properties
relative to $g$. This can be considered as a complete solution of a problem formulated first by Herglotz in 1905. 
We also prove a semi-global results given semi-global data. 
This shows, for instance, that  simply connected manifolds with  strictly convex boundaries  are lens rigid if the sectional curvature is non-positive or non-negative  or if there are no focal points.

The key tool is the analysis of the geodesic X-ray transform on 2-tensors,
corresponding to a metric $g$, in the normal gauge, such as normal
coordinates relative to a hypersurface, where one also needs to allow 
weights. This is handled by refining and extending our earlier results in the
solenoidal gauge.
\end{abstract} 

\maketitle
\section{Introduction and the main result}\label{sec:intro}  
\textit{Boundary rigidity} is the question whether the knowledge of the
boundary restriction (to $\pa M\times\pa M$) of the
distance function $d_g$ of a
Riemannian metric $g$ on a manifold with boundary $M$ determines $g$,
i.e.\ whether the map $g\mapsto d_g|_{\pa M\times \pa M}$ is
injective. Apart from its intrinsic geometric interest, this question has major real-life implications, especially if also a stability result and a reconstruction procedure  are 
given. Riemannian metrics in such
practical applications represent anisotropic media, for example  a sound  speed which, relative 
to the background Euclidean metric, depends 
on the point, \textit{and the direction} of propagation.  Riemannian metrics in the
conformal class of a fixed background metric represent isotropic wave
speeds. While many objects of interest are isotropic to a good
approximation, this is not always the case: for instance, the inner
core of the Earth exhibits anisotropic behavior, see, e.g., \cite{Creager}, as does muscle
tissue. The restriction of the distance function to the boundary is
then the travel time: the time it takes for waves to travel from one
of the points on the boundary to the other. Recall that most of the
knowledge of the interior of Earth comes from the study of seismic
waves, and in particular travel times of seismic waves; the precise
understanding of the boundary rigidity problem is thus very
interesting from this perspective as well.

There is a natural diffeomorphism invariance of the boundary rigidity problem: if
$\psi$ is a diffeomorphism fixing the boundary pointwise, the boundary distance
functions of $g$ and $\psi^*g$ are the same. Thus, the precise
question is whether  $d_g|_{\pa
  M\times\pa M}$ determines $g$ up to this diffeomorphism invariance, i.e.\ whether there is an isometry
$\psi$ (fixing $\pa M$) between $\hat g$ and $g$ if the distance functions of $\hat 
g$ and $g$ have the same boundary restriction.

There are counterexamples to this problem, and thus one needs some
geometric restrictions. The most common restriction is the {\em
  simplicity} of $(M,g)$: this is the requirement that the boundary is
strictly convex and any two points in $M$ can be joined by a unique
minimizing geodesic. (Everywhere in this paper, strict convexity means a positive second fundamental form.) Michel \cite{Michel} conjectured that compact simple
manifolds with boundary are boundary rigid. In this paper we prove
boundary rigidity or the closely related lens rigidity introduced below in dimensions $n\geq 3$ under a different assumption of the 
existence of a function with strictly  convex level sets. 
 Our assumptions hold for simply connected compact
manifolds with strictly convex boundaries such that the geodesic flow
has no focal points, or if the sectional curvature is negative (or just non-positive) or if the sectional curvature is non-negative, see Corollary~\ref{cor_B}. In particular, we prove boundary rigidity for simple manifolds in those cases, see Corollary~\ref{cor_C}. 
This result extends our earlier analogous
result which was in a fixed conformal class \cite{SUV_localrigidity}; recall
that the fixed conformal class problem has no diffeomorphism invariance issues to deal with. We prove local (near a boundary point), semiglobal and global rigidity results. The manifolds we study can have conjugate points. Contrary to previous results (except  for our conformal result in \cite{SUV_localrigidity}), we do not assume the metrics to be a priori close before we prove that they are isometric. In that sense, our results are global in the metrics; and also local in the data.

The conformal case has a long history. In 1905 and 1907,   Herglotz \cite{Herglotz} and Wiechert and Zoeppritz \cite{WZ} showed that one can recover a radial sound speed $c(r)$ (the metric is $c^{-2}d x^2$) in a ball under the condition 
\be{Her}
(r/c(r))'>0
\ee
by reducing the problem to solving an Abel type of equation. For simple manifolds, recovery of the conformal factor was proven in \cite{Mu2} and \cite{MuRo}, with a stability estimate. We showed in \cite{SUV_localrigidity} that for $n\ge3$, one has local and stable recovery near a strictly convex boundary point and semiglobal and global one under the foliation condition we use here, as well. We also showed there that the Herglotz   and Wiechert and Zoeppritz   condition \r{Her}    is equivalent to requiring  the Euclidean spheres $|x|=\text{const.}$ to be strictly convex in the metric $c^{-2}dx^2$. 

The first two-dimensional results are for non-positively curved surfaces by  Croke \cite{Croke90} and Otal \cite{Otal90}. Boundary rigidity of simple surfaces was  proved in \cite{PestovU}. In higher dimensions, simple Riemannian manifolds
with boundary are boundary rigid under a priori constant
curvature assumptions on the manifold or special symmetries
\cite{BCG}, \cite{Gr}. Several local (in the metric) results near the Euclidean metric are known \cite{SU-MRL}, \cite{CDS}; in  \cite{LassasSU} one of the metrics is close to a flat and the other one has an explicit curvature bound; and in \cite{BI},  one of the metrics is a priori close to the flat one and the other one is arbitrary. The most general result in this direction  (outside a fixed conformal class, the setting
of \cite{SUV_localrigidity}) is the generic local (with respect to the
metric) one proven in \cite{SU-JAMS}, i.e.\ one is asking whether simple 
metrics  with  the same boundary distance
function, a priori close to a given one,  are isometric; the authors give an affirmative answer in a generic case.
Surveys of some of the results can be found in \cite{ Croke04,I, Sh-book, SU-Kawai}.

First we analyze the local boundary rigidity problem for  compact Riemannian
manifolds  $(M,g)$ of dimension $n\geq 3$ with a strictly convex boundary. In fact, compactness is not essential for the local results. 
More precisely, for suitable relatively open $O\subset 
M$, including appropriate small neighborhoods of any given point on
$\pa M$ or all of $\pa M$ if $\pa M$ is compact, we show that if for
two metrics $g_1,g_2$ on $M$, $d_{g_1}|_{U\times U}=d_{g_2}|_{U\times
  U}$ for a suitable open set $U$ containing $O\cap\pa M$, then
$g_1=\psi^*g_2$ on $O$ for some diffeomorphism $\psi$ fixing $\pa M$
(pointwise, as we understand throughout this paper). 

\begin{thm}\label{thm:local-impr}
Suppose that $(M,g)$ is an $n$-dimensional Riemannian manifold with boundary, $n\geq
3$, and assume that $\bo$ is strictly convex at some $p\in\bo$ with respect to each of the two metrics $g$ and $\hat g$. 

(i) If $d_g|_{U\times U}=d_{\hat g}|_{U\times U}$, for some neighborhood  $U$ of $p$  in $\pa M$, then  there is a neighborhood $O$ of $p$ in $M$  and a diffeomorphism $\psi:O\to \psi(O)$ fixing $\pa M\cap O$ pointwise  such
that $g|_O=\psi^*\hat g|_O$.

(ii) Furthermore, if the boundary is  everywhere strictly convex
with respect to each of the two metrics $g$ and $\hat g$ and
$d_g|_{\bo\times\bo}=d_{\hat g}|_{\bo\times\bo}$, then  there is a
neighborhood $O$ of $\bo$ in $M$ and a diffeomorphism $\psi:O\to \psi(O)$ fixing $\pa M\cap O$ pointwise  such
that $g|_O=\psi^*\hat g|_O$.
\end{thm}

This theorem becomes more precise regarding the open sets discussed
above if we consider $M$ (not nece\-ssarily compact) as a subset of a manifold without boundary
$\tilde M$, extend $g$ to $\tilde M$; see Figure~\ref{fig:loc-bdy-1}. 
Our more precise theorem then, to which the above theorem reduces, is the following.

\begin{figure}[ht]
\begin{center}
\includegraphics[scale=0.9]{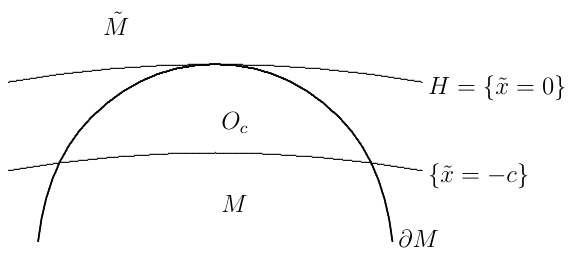}
\end{center}
\caption{The geometry of the local boundary rigidity problem.}
\label{fig:loc-bdy-1}
\end{figure}

\begin{thm}\label{thm:local-pr}
Suppose that $(M,g)$ is an $n$-dimensional Riemannian manifold with
boundary, considered as a domain in $(\tilde M,g)$, $n\geq
3$, $H$ a hypersurface, and $\tilde x$ the signed distance function
from $H$, defined near $H$.
Suppose that $\{\tilde x\geq 0\}\cap
M\subset\pa M$, and for some $\delta>0$, $M\cap \{\tilde x\geq
-\delta\}$ is compact, 
$\pa M$ is strictly convex in $ M\cap \{\tilde x>-\delta\}$, the zero level set
of $\tilde x$ is strictly concave from the superlevel sets in a
neighborhood of $M$.

Suppose also that $\hat g$ is a Riemannian metric on $M$ with
respect to which $\pa M$ is also strictly convex in $M\cap
\{\tilde x>-\delta\}$.

Then there exists $c_0>0$ such that for any $0<c<c_0$, with
$O=O_c=\{\tilde x>-c\}\cap M$, if
$d_g|_{U\times U}=d_{\hat g}|_{U\times U}$ for some open set $U$ in $\pa M$
containing $\overline{\{\tilde x>-c\}\cap \pa M}$, then
there exists a diffeomorphism $\psi:O\to \psi(O)$ fixing $\pa M$ pointwise such
that $g|_O=\psi^*\hat g|_O$. 
\end{thm}

Thus, relative to the level sets of $\tilde x$, the signed distance
function of $H$, we have a very precise statement of where $d_g$ and
$d_{\hat g}$ need to agree on $\pa M$ for us to be able to conclude
their equality, up to a diffeomorphism, on $O=O_c=\{\tilde x>-c\}\cap
M$.

We remark that $c$, thus $O$, can be chosen uniformly for a class of $g$ and $\hat g$ 
with uniformly bounded $C^k$ norms with some $k$.  One can define $C^k$ norms of functions and tensor fields by using a fixed finite atlas or by covariant differentiation w.r.t.\ a fixed metric, as in \cite{LassasSU}. From now on, we measure closeness of metrics or boundedness in $C^k$, $k\gg1$.

The slight enlargement, $U$ of $O\cap\pa M$ plays a role because
we need to extend $\hat g$ to $\tilde M$ in a compatible manner, for 
which we need to recall that if $U$ is an open set in $\pa M$ such 
that $d_g|_{U\times U}=d_{\hat g}|_{U\times U}$ then for any compact 
subset $K$ of $U$ (such as $\overline{O\cap\pa M}$) there is a 
diffeomorphism $\psi_0$ on $M$ such that $\psi_0$ is the identity on a 
neighborhood of $K$ in $\pa M$ and such that $\psi_0^*\hat g$ and 
$g$ agree to infinite order on a neighborhood of $K$ in $M$ \cite{LassasSU, SU-lens}. Replacing $\hat g$ by $\psi_0^*\hat g$, then one 
can extend $\hat g$ to $\tilde M$ in an identical manner with $g$. 
In fact, the diffeomorphism $\psi$ is constructed explicitly: it is locally
given by geodesic normal coordinates of $\hat g$ relative to
$H=\{\tilde x=0\}$; due to the extension process from $M$ to $\tilde
M$, $\psi$ is the identity outside $M$. We refer to section~\ref{sec_7.1} for more details. 

The second problem we study is the \textit{lens rigidity} one.  To define the lens data, we first introduce the manifolds $\partial_\pm SM$, defined as the sets of all vectors $(\x,v)$ with $\x\in\bo$, $v$ unit in  the metric $g$,  and  pointing outside/inside $M$.  We define the \textit{scattering relation}
\[
\cL: \partial_- SM \longrightarrow \partial_+SM 
\]
in the following way: for each $(\x,v)\in \partial_- SM$, $\cL(\x,v)=(q,w)$, where $(q,w)$ are the exit point and  direction, if exist, of the maximal unit speed geodesic $\gamma_{\x,v}$ in the metric $g$, issued from $(\x,v)$. Strict convexity of $\pa M$ is not needed \cite{SU-lens} but it is a convenient assumption for a unambiguous definition of $\cL$, as a continuous map at least, and we assume it from now on. Let 
\[
\ell: \partial_-SM \longrightarrow \R\cup\infty
\]
be its length, possibly infinite. If $\ell<\infty$, we call $M$ non-trapping. The maps $(\cL,\ell)$ together are called \textit{lens relation} (or lens data). 
We identify vectors on $\partial_\pm SM$ with their projections on the unit ball bundle $B\bo$ (each one identifies the other uniquely) 
and think of $\cL$, $\ell$ as defined on the latter with values in itself again, and in $\R\cup\infty$, respectively. With this modification, 
any diffeomorphism fixing $\bo$ pointwise does not change the lens relation.

The {lens rigidity} problem is whether the scattering relation
$\cL$ (and possibly, $\ell$) determine $(M,g)$ up
to an isometry.
The  lens rigidity problem with partial data
is whether we can determine the metric near some $p$ from $\cL$ known
near the unit sphere  $S_p\bo$ considered as a subset of
$\partial_-SM$, i.e., for vectors with base points close to $p$ and
directions pointing into $M$ close to ones tangent to $\bo$, up to an
isometry as above.

Assuming that $\bo$ is strictly convex at $p\in\bo$ with respect to $g$,
the  boundary rigidity and the  lens rigidity problems with partial
data are equivalent: knowing $d=d_g$ near $(p,p)$ is equivalent to knowing
$\cL$ in some neighborhood of  $S_p\partial M$. 
The size of that
neighborhood  depends on a priori bounds of the derivatives of
the metrics with which we work.  This equivalence was first noted by
Michel \cite{Michel}, since the tangential gradients of $d(p,q)$ on
$\bo\times\bo$ give us the tangential projections of $-v$ and $w$, see
also \cite[sec.~3]{SU-lens} and  \cite[sec.~2]{S-Serdica}. Note that  knowledge of $\ell$ may not be  needed for the lens rigidity problem (if $\cL$ is given only, then the problem is called \textit{scattering rigidity} in some works) in some situations.  For example, for simple manifolds,  $\ell$ can be recovered  from either $d$ or $\cL$; and this includes non-degenerate cases of non-strictly convex boundaries,  see for example the proof of  \cite[Theorem~5.2]{SUV_localrigidity}; see \cite{Wen_2015} for a more general result.  Also, in \cite{SUV_localrigidity} it is shown
 that the lens rigidity problem makes sense even if we do not assume a priori knowledge of $g|_{T\bo}$. 

In fact, that relation of the two rigidity problems is used in our proofs  of the first two boundary rigidity theorems. The explicit way we use the equality of $d_g|_{U\times U}$ and $d_{\hat g}|_{U\times U}$ is via the pseudolinearization formula of Stefanov and Uhlmann  \cite{SU-MRL}, see Lemma~\ref{SU-identity}, which relies on the equality of the partial lens  data.

Vargo \cite{V} proved that non-trapping real-analytic manifolds satisfying an additional mild condition are lens rigid.
Croke has shown that if a manifold is lens rigid, a finite quotient of it
is also lens rigid \cite{Croke04}.  He has also shown that the torus is lens
rigid \cite{Croke_scatteringrigidity}. Stefanov and Uhlmann have shown lens rigidity locally near a generic class of non-simple metrics \cite{SU-lens} satisfying an additional microlocal assumption. In a recent work, Guillarmou \cite{Colin14} proved that the lens data
determine the conformal class for Riemannian surfaces with hyperbolic
trapped sets, no conjugate points and strictly convex boundary, and
deformational rigidity in all dimensions under these conditions. 
The only result we know for the lens rigidity problem with incomplete (but not local) data is for real-analytic metric and metric close to them satisfying the microlocal condition in the next sentence \cite{SU-lens}. While in \cite{SU-lens}, the lens relation is assumed to be known on a subset only, the geodesics issued from that subset cover the whole manifold and their conormal bundle is required to cover $T^*M$. In contrast, in this paper, we have localized information.

We then  prove the following global consequence of our local results,
in which (and also below) we {\em assume that each connected component
  of $M$ has non-trivial boundary}, or, which is equivalent in terms of
proving the result, $M$ is connected with non-trivial boundary. As above, we assume $M\subset \tilde M$ with some open $\tilde M$.

\begin{thm}\label{thm:global}
Assume that $(M,g)$ is a compact $n$-dimensional Riemannian manifold, $n\geq 3$, with strictly convex boundary; $\foliation$ is a smooth function with non-vanishing
differential whose level sets are strictly concave from the superlevel sets; and $\{\foliation\geq 0\}\cap M\subset\pa M$. 
Suppose also that $\hat g$ is another Riemannian metric on $M$ so that $\bo$ is strictly convex w.r.t.\ $\hat g$ as well and suppose that
the lens relations of $g$ and $\hat g$ are the same.

Then there exists a diffeomorphism $\psi:M\to M$ fixing $\pa M$ such
that $g=\psi^*\hat g$.
\end{thm}

The assumptions of the theorem are for instance satisfied if
$\foliation$ is the distance function for $g$ from a point outside
$M$, near $M$, in $\tilde M$, minus the supremum of this distance function on $M$, on a
simply connected manifold $\tilde M$
and if $(\tilde M,g)$ has no focal points (near $M$), see Corollary~\ref{cor_B}.

Theorem~\ref{thm:global} can be viewed as a complete solution of the problem initiated by Herglotz \cite{Herglotz} since, as we mentioned above, his condition \r{Her} is a foliation condition.

We formulate a semiglobal result as well, whose proof is actually
included in the proof of the global Theorem~\ref{thm:global} below in Section~\ref{sec:rigidity}. We refer to Figure~\ref{fig:loc-bdy-semiglobal_pic} for an illustration of the theorem.

\begin{thm} \label{thm:semiglobal}
Suppose that $M$ is a compact $n$-dimensional Riemannian mani\-fold with a strictly convex boundary, $n\geq 3$. Let $\foliation$ be a smooth function on $M$ with $[-T,0]$
in its range with $T>0$,  $\{\foliation=0\}\subset\bo$ and $d\foliation\not=0$ on
$\{-T\le\foliation\le0\}$. Assume that each  hypersurface  $\{\foliation=t\}$,
$-T\le t\le 0$, is strictly convex and let
$M_0$ be their union.   Let $D\subset \partial_-SM$ be a neighborhood of the compact set of all $\beta\in  \partial_-S M$ 
which are initial points of geodesics $\gamma_\beta$ tangent to the level surfaces of the foliation. 

Suppose also that $\hat g$ is a Riemannian metric on $M$ with respect to which $\bo$ is also strictly convex 
and suppose that the lens relations of $g$ and $\hat g$ are the same on $D$.  
Then there exists a diffeomorphism $\psi: M_0\to \psi(M_0)$ fixing $\pa M$ pointwise such
that $g=\psi^*\hat g$.
\end{thm}

The strict convexity of $\pa M$ is used only to show that the jets of $g$ and $\hat g$ in boundary normal coordinates coincide. This is true, without convexity,  under the mild assumption of no conjugate pairs of points on $\bo$ \cite{SU-lens} which holds automatically for points close enough on a fixed geodesic, which, with a more general definition of the lens relation for non-strictly convex boundaries  as in \cite{SU-lens} would allow us to remove the strict convexity assumption of $\bo$ in the theorem but we will not pursue this. 

\begin{figure}[h!] 
  \centering
  \includegraphics[scale=1.0,page=1]{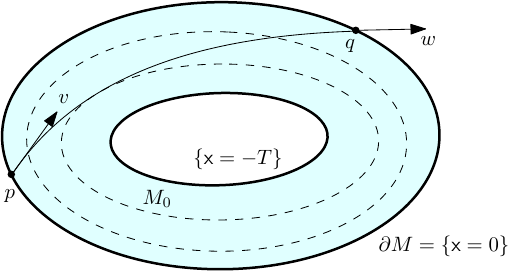}
\caption{The scattering relation $(\x,v)\mapsto
  (q,w)$ restricted to geodesics in the foliation for the semi-global result.}
  \label{fig:loc-bdy-semiglobal_pic}
\end{figure}

A special important case arises when there exists a strictly convex
function, which may have a critical point $x_0$ in $M$ (if so, it is
unique). Then we can apply  Theorem~\ref{thm:semiglobal} in the
exterior of $x_0$; which would create a priori a possible singularity
of the diffeomorphism at $x_0$. In Section~\ref{AppendixB}, we show
that this singularity is removable and obtain a global theorem under
that assumption, see Theorem~\ref{thm:convex-critical-pt}.
 This condition was extensively studied in \cite{PSUZ} (see also the
 references there). In particular Lemma 2.1 of \cite{PSUZ} shows that
 such a function exists  if the sectional curvature of the manifolds
 is non-negative or if  the manifold is simply connected and the
 curvature is non-positive. Manifolds satisfying one of these
 conditions are lens rigid:

\begin{cor}\label{cor_B} 
Let $(M,g)$ be a compact Riemannian manifold with a strictly convex boundary of dimension $n\ge3$  satisfying any of the conditions

(a) $(M,g)$ is simply connected with a non-positive sectional curvature;

(b) $(M,g)$ is simply connected and has no focal points;

(c) $(M,g)$ has non-negative sectional curvature.

\noindent Then if $g_1$ is another metric on $M$ with respect to which
$\bo$ is also strictly convex  and with the same lens data, $(M,g)$ is isometric to $(M,g_1)$ with an isometry fixing the boundary pointwise. 
\end{cor}

Note that (c) can be replaced by the weaker condition of a lower
negative bound of the sectional curvature; depending on some geometric
invariants of $(M,g)$, see \cite{PSUZ}. 

As  mentioned earlier, the lens rigidity problem and the boundary
rigidity problem are equivalent for simple manifolds (which are simply connected). Therefore we
have proved Michel's conjecture in dimension $n\geq 3$ under
conditions corresponding to those of Corollary~\ref{cor_B}. More precisely:

\begin{cor}\label{cor_C} 
Let $(M,g)$ be a compact simple Riemannian manifold with a strictly convex boundary of dimension $n\ge3$  satisfying any of the conditions

(a) $(M,g)$ has non-positive sectional curvature;

(b) $(M,g)$ has no focal points;

(c) $(M,g)$ has non-negative sectional curvature. 

\noindent If $g_1$ is another metric on $M$ with respect to which
$\bo$ is also strictly convex  and with the same
boundary distance function, $(M,g)$ is isometric to $(M,g_1)$ with an
isometry fixing the boundary pointwise. Thus, these classes of
Riemannian manifolds are boundary rigid.
\end{cor}

\textbf{Acknowledgments.} The authors thank Gabriel Paternain, Mikko Salo and the referees for their valuable suggestions.

\section{The approach}\label{sec:approach}
This paper relies crucially on the papers
\cite{UV:local,SUV_localrigidity,SUV:Tensor} both in terms of the
approach and in terms of the results; indeed, these three papers can
be thought of as being part of a process that culminates with the
present result. Thus, we start by discussing these briefly.

The rough picture is that via a linearization procedure, the boundary
rigidity problem connects to the geodesic X-ray transform. In the
general problem we study here this is the X-ray transform on symmetric
2-tensors explored in \cite{SUV:Tensor}.  However, in the simpler case
of boundary rigidity in a fixed conformal class of metrics, which was
proved in \cite{SUV_localrigidity}, it connects to the X-ray transform
on functions. The key analytic ideas in the latter setting were
introduced in \cite{UV:local}. Relative to \cite{UV:local}, the fixed
conformal class boundary rigidity problem, \cite{SUV_localrigidity}, required moving to a
nonlinear setting. On the other hand, the symmetric 2-tensor X-ray
problem is still linear but has a gauge invariance; dealing with this
was the key point in \cite{SUV:Tensor}. Finally the present paper must
combine the ability to deal with the gauge invariance with the
ability to work on a non-linear problem. We go through these
ingredients one by one.

\subsection{The X-ray transform on functions, \`a la \cite{UV:local}}
On a Riemannian manifold $(M,g)$, the geodesic X-ray transform of 2-tensors
is a map $\CI(M)\to\CI(SM)$
\[
If(\beta)=\int_{\gamma_\beta} f(\gamma_\beta(s))\,ds,
\]
where for $\beta\in SM$, $\gamma_\beta$ is the lifted geodesic through
$\beta$. A key question is if from $If$ we can recover $f$, which can take various
forms: injectivity, stability estimates, or perhaps even a
construction of a left inverse. Since $I$ is a Fourier integral
operator, one general approach is to consider the normal operator,
$I^*I$. 
The operator 
$$
L_0v(z)=\int_{S_zM} v(\gamma_{z,\zeta})\,d\zeta 
$$
is actually $I^*$ with a suitable natural parameterization of the space of the geodesics \cite{S-Serdica}

Under the assumption that $M$ has no conjugate points, and working on
the extension
$\tilde M$, $L_0I$ is a
pseudodifferential operator of order $-1$, and moreover it is
elliptic for $n\ge2$, see \cite{SU-Duke, SU-JAMS}. (These requirements can be somewhat relaxed by
microlocalization, see \cite{SU-lens}.) Then there is a parametrix $G$
such that $GL_0I$ differs from the identity operator (when restricted
to distributions supported in $M$) by a smoothing operator. While this
is sufficient for a semi-Fredholm theory, it does not rule out a
potentially large finite dimensional nullspace.

The key advance of \cite{UV:local} was to consider a localized
problem, which introduced a small parameter, as we now explain. This
small parameter is what enables us to rule out the potential large
nullspace and thus to construct a left inverse of $L_0I$, where $L_0$ is a
localized version of the $L_0$ above. Concretely then, suppose we have a convex
foliation, concave from the super-level sets, given by the level sets of a function
$\foliation$ of non-vanishing differential. For a fixed value $\level$
(we use the typeface $\level$ here to distinguish it from the
conformal class factor we discuss next), we
consider the level set $\foliation=-\level$ as an {\em artificial
  boundary}, and consider the region
$\Omega_\level=\{\foliation>-\level\}\cap M$ for the purpose of finding $f|_{\Omega_\level}$
from the information given by the $If(\beta)$ for those $\beta$ for which the geodesic
through $\beta$ stays in $\Omega_\level$ until it hits $\pa M$, i.e.\ for
$\Omega_\level$-localized geodesics. Let $x_\level=\foliation+\level$ be a boundary
defining function for $\{\foliation>-\level\}$ in $\tilde M$. In order to implement this analytically, we
need to add a cutoff to the definition of $L_0$:
$$
L_0v(z)=\int_{S_z M} \chi(z,\zeta) v(\gamma_{z,\zeta})\,d\zeta.
$$
Here $\chi$ localizes to a subset of geodesics that are `almost tangent' to level
sets of $\foliation$. The precise type of operator one obtains depends
on the precise way one implements the almost tangency. We take this so
that on the support of $\chi$, the tangent vector to $\gamma_{z,\zeta}$ at $z$ encloses an angle
$\lesssim x_\level$ with the level sets of $\foliation$, i.e.\ the geodesics
become tangent to the level sets as one approaches the artificial
boundary at a rate that is roughly proportional to the distance to the
artificial boundary. The concavity assumption on the super-level sets
implies that these geodesics are indeed $\Omega_\level$ local. One could in
fact take a somewhat larger angle from tangency just for the concavity
considerations, but our choice ensures that $L_0I$, or more precisely
$e^{-\digamma/x_\level}L_0Ie^{\digamma/x_\level}$, where $\digamma>0$, is a particularly
well-behaved elliptic  pseudodifferential operator: it is in Melrose's
scattering pseudodifferential algebra which has a powerful symbolic
structure and which we discuss in some detail in
Section~\ref{sec:normal-gauge}. Effectively this means that
analytically the artificial boundary acts like a region near infinity
in Euclidean space. On the other hand, the parameter $\digamma$ means
that we are working on exponentially weight spaces, so the estimates
on $f$ (from $If$) will be exponentially weak as one approaches the
artificial boundary since $e^{-\digamma/x_\level}L_0Ie^{\digamma/x_\level}$ should
be thought of as being applied to $e^{-\digamma/x_\level}f$. The key
point is that the level set parameter
$\level$ becomes a new tool: by taking $\level$ sufficiently small, one can
assure that not only is the error of a parametrix `smoothing' (really,
`Schwartzifying' in the asymptotically Euclidean interpretation) but is
actually small as an operator, so the identity plus this error can be
inverted.

Note that the ellipticity now requires $n\ge3$ because we deal with ``almost tangent'' (to the actual or to the artificial boundary) geodesics only. If $n=2$, we get ellipticity on codirections close to normal ones only. 

In order to invert the X-ray transform globally then one has a layer
stripping procedure, in which first one recovers $f$ in
$\foliation\geq-\level_1$, $\level_1>0$ small, then in $-\level_1\geq\foliation\geq
-\level_2$, $\level_2-\level_1>0$ small, etc. Since we can control the step size,
compactness considerations result in global injectivity, stability, etc.

\subsection{Boundary rigidity in a fixed conformal class, \`a la
  \cite{SUV_localrigidity}}
If we have a fixed conformal class, i.e.\ we study multiples
$c^{-2}g_0$ of a
background metric $g_0$, then the linearization (in $c$) of the boundary
distance function around a certain $c_0$ is an X-ray transform of
$\delta c$. 

As mentioned already in the introduction, we actually use the lens
information. This gives rise to a formula, called the
pseudolinearization formula in \cite{SU-MRL}, for the difference of the
cotangent bundle
coordinates of the point $\tilde Z(t,z)$, resp.\ $Z(t,z)$, of the time
$t$ Hamilton flows emanating from a boundary point in the same
direction, i.e.\ from $z=(x,\xi)$, $\xi=g_x(\zeta)$:
\begin{equation}\label{eq:gen-pseudolin}
\tilde Z(t,z) - Z(t,z)= 
\int_0^t \frac{\partial\tilde Z}{\partial z} (t-s,Z 
(s,z))\big(\tilde V - V\big)(Z (s,z))\,ds;
\end{equation}
here $\tilde V$ and $V$ are the Hamilton vector fields given by
$\tilde c^{-2} g_0$ and $c^{-2}g_0$.
If the lens relations are the same, then taking $t$ as the time $\tau(x,\xi)$ at which
the respective flows both reach the boundary at the same point, the
left hand side vanishes. Expressing the Hamilton vector field in terms
of the factors $\tilde c,c$ and their first derivatives, and taking
the momentum (i.e.\ $\xi$) component of $Z$, we obtain a formula
for the integral of the first
derivatives of $\tilde c-c$ and $\tilde c-c$ itself. Since
\eqref{eq:gen-pseudolin} integrates the difference of the Hamilton 
vector fields along the trajectory $Z(.,z)=Z(.,x,\xi)$, i.e.\ along a
bicharacteristic, i.e.\ a
lifted geodesic, this turns to
be an X-ray transform with a weight (essentially given by the
prefactor in \eqref{eq:gen-pseudolin}). 
Namely if we write $f=c^2-\tilde
c^2$, we obtain
\begin{equation}\label{eq:sound-speed-J}
J_if(\gamma):= \int \left( A_i^j(X(t),\Xi(t))(\partial_{x^j}f)(X(t))+
  B_i (X(t),\Xi(t))f(X(t)) \right)\,dt=0,\ \qquad i=1,\ldots,n,
\end{equation}
for any bicharacteristic  $\gamma = (X(t),\Xi(t))$ (related to the
speed $c$) in our set $\Omega_\level$, where
\begin{equation*}\begin{aligned}
A_i^j\left(x,\xi\right) = &-\frac12 \frac{\partial\tilde  \Xi_i}{\partial \xi_j}(\tau(x,\xi),(x,\xi)) c^{-2}(x),
\\
B_i\left(x,\xi\right) = &\frac{\partial \tilde\Xi_i}{\partial x^j}(\tau(x,\xi),(x,\xi))g_0^{ik}(x) \xi_k  -\frac12 \frac{\partial\tilde  \Xi_i}{\partial \xi_j}(\tau(x,\xi),(x,\xi))  (\partial_{x^j}   g_0^{-1}(x))\xi\cdot\xi.
\end{aligned}\end{equation*}
This way, we deal with the geometry of a single metric directly, and the geometry of the other one affects the weight. 
At the boundary of $M$ we have $A_i^j(x,\xi)=-\frac{1}{2}c^{-2}\delta_i^j$. Then
the transform given by just the $A_i^j$ term gives rise to an elliptic
pseudodifferential operator by taking $L_0$ essentially as above (since
we have $n$ components corresponding to the $n$ derivatives, really the $n$ by $n$
matrix version, $L_0\Id_n$), 
while the $B_i$ terms
can be absorbed using a Poincar\'e-type inequality at least for
sufficiently small domains (the foliation parameter is near $0$). This shows that if
$Jf$ vanishes then so does $f$, i.e.\ $c=\tilde c$, proving the local
version of the boundary rigidity in a fixed conformal class.

\subsection{The X-ray transform on tensors, \`a la \cite{SUV:Tensor}}
The geodesic X-ray transform of 2-tensors
along the geodesics of a metric $g$ is a map $\CI(M;\Sym^2T^*M)\to\CI(SM)$
\[
If(\beta)=\int_{\gamma_\beta} f(\gamma_\beta(s))(\dot\gamma_\beta(s),\dot\gamma_\beta(s))\,ds,
\]
and in this transform the symmetric 2-tensor $f$ is evaluated on the tangent vector of
$\gamma_\beta$ in both slots.

The key difference between the X-ray transform on tensors and on
scalar functions is {\em not} that tensors are sections of a bundle:
after all, locally this is just a transform of a matrix function, and
these were analyzed above for the fixed conformal class boundary
rigidity. Rather, the issue is the gauge invariance, which is to say
that if $f$ is a {\em potential tensor}, i.e.\ is the symmetric differential of
a one-form vanishing on the boundary, $f=\dsymm v$, then $If=0$. (In the
analogous one-form setting, this is simply the fundamental theorem of
calculus.) The standard way of fixing this gauge invariance is
adding a gauge condition, and the most standard (due to the
ellipticity we are about to discuss) gauge condition is the
{\em solenoidal gauge condition}, $\delta^s f=0$, where $\delta^s$ is
(negative) divergence. Working globally, taking a background metric
$g_0$ (possibly equal to $g$, but this is not needed), one uses this by replacing the operator $L_0$
above by
$$
L_2v(z)=\int_{S_zM} v(\gamma_{z,\zeta}) g_0(\zeta)\otimes g_0(\zeta)\,d\zeta,
$$
and rather than just taking $L_2I$, one considers $L_2I+\dsymm Q\delta^s$,
where $Q$ is an order $-3$ pseudodifferential operator. This is
elliptic for a suitable choice of $Q$, and applied to tensors in the
solenoidal gauge the second term vanishes, so if $If=0$, then one
concludes that $f$ is smooth, and indeed that there is a finite
dimensional nullspace. There are some additional difficulties near the boundary since solenoidal  tensors extended as zero outside $M$ may not be solenoidal anymore. 

The localized version is quite similar, with the main difference that
the {\em weighted solenoidal gauge} also has an exponential weight:
$\delta^s(e^{-2\digamma/x_\level}f)=0$. Concretely, let
$\delta^s_\digamma=e^{\digamma/x_\level}\delta^s e^{-\digamma/x_\level}$,
$\dsymmw=e^{-\digamma/x_\level}\dsymm e^{\digamma/x_\level}$. Then the
analogue of $L_2I+\dsymm Q\delta^s$ is
\[
A_\digamma=N_\digamma+\dsymmw
Q\delta^s_\digamma,\qquad N_\digamma=e^{-\digamma/x_\level}L_2Ie^{\digamma/x_\level},
\]
where $L_2$ again has a cutoff $\chi$. Again, this
can be arranged to be elliptic for suitable $\chi$ and $Q$ and
suitably large $\digamma>0$, and thus is invertible up to a smoothing
(`Schwartzifying') error by applying a parametrix $G_\digamma$. Now, for again sufficiently small indexed
level set, i.e.\ sufficiently small $\level$, chosen as the artificial
boundary, the error $G_\digamma A_\digamma-\Id$ is not just
`smoothing'/Schwartzifying, but is actually small, so it can be
removed as in the scalar case, i.e.\ we may assume $G_\digamma A_\digamma=\Id$.
If $f$
is in this exponential solenoidal gauge, then applying $A_\digamma$
to $e^{-\digamma/x_\level}f$ gives
$$
A_\digamma e^{-\digamma/x_\level}f  =N_\digamma e^{-\digamma/x_\level}f =e^{-\digamma/x_\level}L_2If,
$$
which thus is determined by $If$, hence the same for
$$
e^{-\digamma/x_\level}f=G_\digamma A_\digamma e^{-\digamma/x_\level}f=G_\digamma N_\digamma e^{-\digamma/x_\level}f =G_\digamma e^{-\digamma/x_\level}L_2If.
$$

We actually suppressed an issue here: putting a tensor $f$ into
solenoidal gauge by adding a potential term, $\dsymm v$, requires solving a
weighted Laplace-type equation on one forms (with a weight,
essentially $e^{-\digamma/x_\level}$, singular at the artificial boundary), which is almost as involved as the
argument we outlined. Part of the issue is that the solution $v$ of
this equation necessarily depends on the whole domain on which we are
solving this Laplace-type equation, and in the actual inversion
procedure a few different domains (in $x_\level\geq 0$) are considered due to the extended
(to $\tilde M$)
nature of the parametrix construction, so these must be related and
the behavior of the Laplace-type operator at artificial boundary
(which is also in Melrose's scattering algebra) also taken into
account in the solution procedure. In particular, as we mentioned above, the extension of a solenoidal tensor, extended as zero
outside $M$, may not be solenoidal anymore, which is ultimately the
reason that the Laplace-type equation must be solved in a number of domains.

\subsection{Boundary rigidity}
One immediate issue with general boundary rigidity (as opposed to the
fixed conformal class one) and localization is that
even if
we have two metrics $g$ and $\tilde g$ with the same lens relation,
it may well happen that $g$ and $\tilde g$ are different due to the
diffeomorphism invariance (the analogue of the above gauge invariance
for the tensor X-ray transform). Therefore, we cannot really expect to be able
to make a statement that in some fixed region they are the same `up to
diffeomorphism': the diffeomorphism deforms the region itself. 
The localization however is an essential part of assuring
the lack of null space of the modified normal operators, at least by our methods. This
already complicates the general boundary rigidity problem.

One can try to circumvent this difficulty by putting the metrics in a
certain gauge in order to eliminate the diffeomorphism
invariance; then we want to prove that they are equal. 
Given the symmetric 2-tensor discussion above, one may
want to put them in a (weighted) solenoidal gauge with respect to a background
metric. An immediate issue of arranging the solenoidal gauge for our
local problems is that it
requires solving an
elliptic PDE, essentially a weighted Laplace-Beltrami equation on one-forms, with
the weight singular at the boundary of $\Omega_\level$ (essentially
$e^{-\digamma/x_\level}$), which again comes back to the point that
one should know the corresponding regions for the two metrics from the
start! Thus
the extension of the solenoidal gauge to non-linear problems appears problematic.

Instead we use the {\em normal gauge} in a product-decomposition of
the underlying manifold, which for
the linear problem means working with tensors (differences of two metrics) whose normal components
vanish (for 2-tensors, this means normal-normal and tangential-normal
components; in the 1-form problem discussed below this means the normal
component). We can pull back each metric by a (metric dependent) local
diffeomorphism so that each new metric is in normal coordinates
relative to a hypersurface, see section~\ref{sec_7.1}. If this is
done, then their difference is in the normal gauge.
An addition of symmetric derivatives of one-forms vanishing at $\pa M$, i.e.\ of potential tensors, does not change the X-ray transform. 
In the normal gauge, this linear invariance disappears and we want to prove injectivity. The
operator $e^{-\digamma/x_\level}L'_2Ie^{\digamma/x_\level}$ however is not
elliptic even restricted to tangential-tangential tensors, i.e.\
tensors in this normal gauge, as noticed already in \cite{SU-MRL}. 
Here $L_2'$ is the analogue of $L_2$
replacing $g_0(\zeta)\otimes g_0(\zeta)$ by its tangential-tangential component,
so that the output is a tangential-tangential tensor. However, there is a major gain: putting an arbitrary
one-form or tensor into the normal gauge by adding a potential tensor
requires solving what amounts to an evolution equation, so this itself
is not an elliptic process (though it is {\em much} simpler than dealing with the
non-ellipticity of the X-ray transform in this gauge). The evolutionary nature allows one
to work locally, since the property of being in the normal gauge is
{\em independent of the choice of the artificial boundary}. 
Thus, we
have a well-behaved gauge condition for the non-linear problem, but at
the cost of losing the ellipticity of our modified normal operator.

Going back to the linear setting, namely that of the X-ray
transform on tensors, if one would like to recover a tensor $f$ which is in the
normal gauge from $If$, it is thus
easier to put $f$ in the solenoidal gauge first, by adding a term
$\dsymm v$.  Then we recover
$f+\dsymm v$ from $N_\digamma e^{-\digamma/x_\level}(f+\dsymm v)=N_\digamma e^{-\digamma/x_\level} f$, hence from $If$, using the solenoidal gauge estimate, i.e.\ the original tensor $f$ up to a
potential term. Then argue that in fact this determines $f$
due to the vanishing of its normal components.
We in fact present this in Section~\ref{sec:Fredholm-geodesic},
together with actual estimates for $f$ in terms $N_\digamma
e^{-\digamma/x_\level}(f+\dsymm v)=N_\digamma e^{-\digamma/x_\level}
f$. These estimates are non-elliptic, with a natural loss of
derivatives in the tangential to the foliation direction; see
Theorem~\ref{thm:normal-gauge-2-tensor-Fredholm-est} and its
Corollary~\ref{cor:normal-gauge-2-tensor-Fredholm}, which gives a
direct left invertibility statement for $N_\digamma$ on tensors in the
normal gauge as a map between appropriate generalized Sobolev spaces.

This approach of using
the solenoidal result for a problem in the normal gauge does {\em not} work for
the pseudolinearization directly, however, because with $J$ being the
generalized X-ray transform of the Stefanov-Uhlmann formula in
Lemma~\ref{SU-identity}, namely the tensorial analogue of $J$ in
\eqref{eq:sound-speed-J} in our fixed conformal class setting, $J$
is {\em not} expected to annihilate potential tensors since
$J$ is {\em not} the actual tensorial X-ray transform. 
Indeed, once the normal coordinates are
fixed, and we are working in a fixed region (so we expect $g=\tilde
g$, without diffeomorphism issues), we
can make the tangential-tangential tensor $g-\tilde g$ solenoidal
relative to a reference metric in the
fixed region, changing $g-\tilde g$ by a potential
term $\dsymm v$ by enforcing
$\delta^s_\digamma (e^{-\digamma/x_\level}(g'-\tilde g'))=0$, but this
{\em eliminates} the identity $J(g'-\tilde g')=0$.

So for our boundary rigidity problem, relying on the
pseudolinearization formula, one needs to argue more directly for the
left invertibility of the weighted transform $J$ in the normal
gauge. The most direct way to proceed would be to deal with the lack
of ellipticity of $e^{-\digamma/x_\level}L'_2Je^{\digamma/x_\level}$
in some way. While in principle the latter is relatively benign, it
gets worse with the order of the tensor: for one-forms it should be
roughly real principal type, except that it is really real principal
type times its adjoint (so quadratic vanishing at the characteristic
set, but with extra structure); in the case of symmetric 2-tensors we have
quadratic vanishing in the first place so quartic once one looks at
the operator times its adjoint.

This large degeneracy, however, can be improved as follows.
We complement the operator $L'_2$ by a larger
collection of operators $L'_j$, $j=0,1$. All $L'_j$ will be
similar integrals, but mapping to different spaces, not just to tangential-tangential
2-tensors; in fact, they can be considered as the parts of the
original $L_2$ mapping into other
components, such as normal-tangential, so altogether one considers $L_2 I=(L_0' I,L_1' I, L_2'I)$. 
After the exponential conjugation this becomes a
pseudodifferential operator between different bundles
(tangential-tangential symmetric tensors to all symmetric
tensors). This is {\em still} not `elliptic' (here meaning having an
injective principal symbol), but the failure of
ellipticity is less pronounced than for the conjugate of
$L_2'I$. Indeed,  for the related
one-form problem (in the normal gauge) this approach easily
gives self-contained results, such as semi-Fredholm theory; we sketch this in
Section~\ref{sec:1-form} using the microlocal real principal type and
radial point tools as in \cite{Vasy-Dyatlov:Microlocal-Kerr} and
\cite{Vasy:Minicourse}.
However, for symmetric 2-tensors in the normal gauge the degeneracy is
still quadratic, and thus harder to deal with for a direct
semi-Fredholm theory, though the improved
structure gives rise to {\em precise mapping properties} of the
{\em operator itself} on suitable Sobolev
spaces with extra regularity properties.

So, instead of proceeding this way, in the 2-tensor setting we combine the very direct approach to the
pseudolinearization transform $J$ and the
relationship between the solenoidal and normal gauge results for the
{\em actual} $X$-ray transform $I$. This can be done because for
$I$ we have an actual left inverse, and as we show in
Section~\ref{sec:Fredholm-weights}, for small $\level>0$, the operator
$N_\digamma$ induced by $I$ is close to the operator $\tilde
N_\digamma$ induced by $J$ {\em as a map between the function spaces of the
left invertibility result}. Due to the invertibility of $N_\digamma$,
we conclude the same for $\tilde N_\digamma$.

Ultimately, this means that the general analysis of tensorial X-ray
transforms in a manner that is suitable for the weighted version,
which is done in Sections~\ref{sec:2-tensor}, is used as the {\em
  regularity theory} for the actual X-ray transform in the normal
gauge, to obtain the sharp results in Section~\ref{sec:Fredholm-geodesic}, as well as to have desired mapping (including perturbation
stability) properties of the weighted transform.
These results are then used in
Section~\ref{sec:rigidity} to prove the actual boundary
rigidity results.

{\em A notational warning: from Section~\ref{sec:1-form}, the maps $L'_j$
  of this last section are denoted by $L_j$, and $L$ takes the place
  of $L_2$ (or $L_1$ in the one-form setting).}

\section{The transform in the normal gauge}\label{sec:normal-gauge}
\subsection{The scalar operator $L$}
We first recall the
definition of $L$ from \cite{SUV:Tensor} and \cite{UV:local}. For
this, it is convenient to consider $M$ as a domain in a larger
manifold without boundary $\tilde M$ by extending $M$ and the metric
across $\pa M$. The basic input is a function $\tilde x$ whose
level sets near the zero level set are strictly concave, from the side of
superlevel sets (at least near the 0-level set) (it suffices
if this only holds on the intersection of these level sets with $M$) whose 0 level set only
intersects $M$ at $\pa M$; an example would be
the negative of a
boundary defining function of our strictly convex domain. We also need
that $\{\tilde x\geq -c\}\cap M$ is compact for $c\geq 0$ sufficiently
small, and we let
$$
\Omega=\Omega_c=\{\tilde x>-c\}\cap M
$$
be the region in which, for small $c>0$, we want to recover a tensor
in normal gauge from its X-ray transform.
In the
context of the elliptic results, both for functions, as in
\cite{UV:local}, and in the tensor case, as in \cite{SUV:Tensor}, this
function $\tilde x$
need not have any further connections with the metric $g$ for which we
study the X-ray transform. {\em However, for obtaining optimal
  estimates in our normal gauge, which is crucial for a perturbation
  stable result, it will
  be important that the metric itself is in the normal gauge near
  $\{\tilde x=0\}\cap M$, i.e.\ writing the region as a subset of
  $(-\delta_0,\delta_0)_{\tilde x}\times Y$ with respect to a product
  decomposition, the metric is of the form $g=d\tilde x^2+h(\tilde x,y,dy)$.}

Concretely $L$ is defined as follows in \cite{SUV:Tensor}. Near $\pa\Omega$, one can use
coordinates $(x,y)$, with $x=x_c=\tilde x+c$ as before, $y$
coordinates on $\pa\Omega$, or better yet $H=\{\tilde x=0\}$. Correspondingly, elements of $T_pM$ can be
written as $\lambda\,\pa_x+\omega\,\pa_y$. The unit speed geodesics which
are close to being tangential to level sets of $\tilde x$ (with the
tangential ones being given by $\lambda=0$) through a
point $p=(x,y)$ can be parameterized by say $(\lambda,\omega)$ (with the
actual unit speed being a positive multiple of this) where
$\omega$ is unit length with respect to a metric on $H$ (say a
Euclidean metric if one is working in local coordinates). These have
the form (cf.\ \cite[Equation~(3.17)]{UV:local})
\begin{equation}\label{eq:geod-form}
(x+\lambda t+\alpha(x,y,\lambda,\omega)t^2+O(t^3),y+\omega t+O(t^2);
\end{equation}
the strict concavity of the level sets of $\tilde x$, as viewed from the
super-level sets means that $\alpha(x,y,0,\omega)$ is positive.
Thus, by this concavity, (for $\lambda$ sufficiently small) $\frac{d^2}{dt^2}\tilde x\circ\gamma$ is
bounded below by a positive constant along geodesics in $\Omega_c$, as
long as $c$ is small, which in turn means that, for sufficiently small
$C_1>0$, geodesics with
$|\lambda|<C_1\sqrt{x}$ indeed remain in $x\geq 0$ (as long as they
are in $M$). Thus, if $If$ is known along $\Omega$-local geodesics,
meaning geodesic segments with endpoints on $\pa M$, contained within $\Omega$, it
is known for geodesics $(x,y,\lambda,\omega)$ in this range. As in
\cite{UV:local} we use a smaller range $|\lambda|<C_2 x$ because of
analytic advantages, namely the ability work in the well-behaved
scattering algebra even though in principle one might obtain stronger
estimates if the larger range is used (polynomial rather than
exponential weights).
Thus,
for $\chi$ smooth, even, non-negative, of compact support, to be specified, in the function case
\cite{UV:local} considered the operator
$$
Lv(z)=x^{-2}\int \chi(\lambda/x)v(\gamma_{x,y,\lambda,\omega})\,d\lambda\,d\omega,
$$
where $v$ is a (locally, i.e.\ on $\supp\chi$, defined) function on the space of geodesics, here parameterized
by $(x,y,\lambda,\omega)$. (In fact, $L$ had a factor $x^{-1}$
only in \cite{UV:local}, with another $x^{-1}$ placed elsewhere; here we simply combine
these, as was also done in \cite[Section~3]{SUV_localrigidity}. Also,
the particular measure $d\lambda\,d\omega$ is irrelevant; any smooth positive
multiple would work equally well.)
The key result was that $L I$ is a pseudodifferential operator
of a certain class
on
\begin{equation}\label{eq:artificial-bdy-mfld}
X=\{x\geq 0\},
\end{equation}
considered as a manifold with boundary; note that only a
neighborhood of $\Omega$ in $\tilde M$ actually matters here due to the
support of the functions to which we apply $I$. An important point is that
the {\em artificial boundary} that we introduced, $\{x=0\}$, is what
is actually important, the original boundary of $M$ simply plays a
role via constraining the support of the functions $f$ we
consider.

\subsection{Scattering pseudodifferential operators}
More precisely then, the pseudodifferential operator class
is that of {\em scattering pseudodifferential operators}, introduced
by Melrose in \cite{RBMSpec} in this generality, but having precedents
in $\RR^n$ in the works of Parenti and Shubin \cite{Pa,Sh}, and in this case it is
also a special case of H\"ormander's Weyl calculus with product type
symbols \cite{Hor}. Thus, on $\RR^n$ the class of symbols $a\in S^{m,l}$ one
considers are ones with the behavior
$$
|D_z^\alpha D_\zeta^\beta a(z,\zeta)|\leq C_{\alpha,\beta}\langle
z\rangle^{l-|\alpha|}\langle\zeta\rangle^{m-|\beta|},\ \alpha,\beta\in\NN^n,
$$
quantized in the usual way, for instance as
$$
Au(z)=(2\pi)^{-n}\int e^{i(z-z')\cdot\zeta}a(z,\zeta)u(z')\,dz'\,d\zeta,
$$
understood as an oscillatory integral; one calls $A$ a scattering
pseudodifferential operator of order $(m,l)$. A typical example of such an
$A$ is a scattering differential operator of order $m$, thus of order $(m,0)$ as
a scattering pseudodifferential operator: $A=\sum_{|\alpha|\leq m}
a_\alpha(z)D^\alpha_z$, where for each $\alpha$, $a_\alpha$ is a $0$-th
order symbol on $\RR^n$: $|D^\gamma a_\alpha(z)|\leq C_{\alpha\gamma}\langle
z\rangle^{-|\gamma|}$, $\gamma\in\NN^n$. A special case is when each
$a_\alpha$ is a classical symbol of order $0$, i.e.\ it has an
expansion of the form $\sum_{j=0}^\infty a_{\alpha,j}(z/|z|) |z|^{-j}$
in the asymptotic regime $|z|\to\infty$. These operators form an
algebra, i.e.\ if $a\in S^{m,l}$, $b\in
S^{m',l'}$, with corresponding operators $A=\Op(a)$, $B=\Op(b)$, then
$AB=\Op(c)$ with $c\in S^{m+m',l+l'}$; moreover $c-ab\in
S^{m+m'-1,l+l'-1}$. Correspondingly it is useful to introduce the
{\em principal symbol}, which is just the class $[a]$ of $a$ in
$S^{m,l}/S^{m-1,l-1}$, suppressing the orders $m,l$ in the notation of
the class; then $[c]=[a][b]$.
Notice that this algebra is
commutative to leading order both in the differential and decay sense, i.e.\ if $a\in S^{m,l}$, $b\in
S^{m',l'}$, with corresponding operators $A=\Op(a)$, $B=\Op(b)$, then
$[A,B]=\Op(c)$, $c\in S^{m+m'-1,l+l'-1}$,
$$
c- \frac{1}{i}\sum_{j=1}^n \Big(\frac{\pa a}{\pa \zeta_j}\frac{\pa
  b}{\pa z_j}-\frac{\pa a}{\pa z_j}\frac{\pa b}{\pa
  \zeta_j}\Big) \in S^{m+m'-2,l+l'-2}.
$$
We introduce 
$$
H_a b=\sum_{j=1}^n \Big(\frac{\pa a}{\pa \zeta_j}\frac{\pa
  b}{\pa z_j}-\frac{\pa a}{\pa z_j}\frac{\pa b}{\pa
  \zeta_j}\Big),
$$
where $H_a=\sum_{j=1}^n \Big(\frac{\pa a}{\pa \zeta_j}\frac{\pa}{\pa z_j}-\frac{\pa a}{\pa z_j}\frac{\pa}{\pa
  \zeta_j}\Big),$ is the Hamilton vector field of $a$.
These operators
also act on weighted Sobolev spaces, $H^{s,r}=\langle
z\rangle^{-r}H^s(\RR^n)$ in the sense that for $a\in S^{m,l}$,
$\Op(a):H^{s,r}\to H^{s-m,r-l}$ in a continuous linear manner.

In order to extend this to
manifolds with boundary, it is useful to compactify $\RR^n$ radially
(or geodesically) as a ball $\overline{\RR^n}$; different points on
$\pa\overline{\RR^n}$ correspond to going to infinity in different
directions in $\RR^n$. Concretely this is achieved by identifying,
say, the exterior of the closed unit ball with
$(1,\infty)_r\times\sphere^{n-1}_\omega$ via `spherical coordinates', which
in turn is identified with $(0,1)_x\times\sphere^{n-1}_\omega$ via the map
$r\mapsto r^{-1}$, to which we glue the boundary $x=0$, i.e.\ we
consider it as a subset of $[0,1)_x\times\sphere^{n-1}_\omega$. (More
formally, one takes the disjoint union of $[0,1)_x\times\sphere^{n-1}$
and $\RR^n$, and identifies $(0,1)\times\sphere^{n-1}$ with the
exterior of the closed unit ball, as above.) Note that for this
compactification of $\RR^n$ a classical symbol of order $0$ on $\RR^n$
is simply a $\CI$ function on $\overline{\RR^n}$; the asymptotic
expansion $\sum_{j=0}^\infty a_{\alpha,j}(z/|z|) |z|^{-j}$ above is
actually Taylor series at $x=0$: $\sum_{j=0}^\infty x^j
a_{\alpha,j}(\omega)$.

It is also instructive to see what happens to scattering vector fields
in this compactification: $V=\sum_{|\alpha|=1}a_\alpha D^\alpha$. A
straightforward computation shows that $D_j$ becomes a vector field on
$\overline{\RR^n}$ which is of the form $xV'$, where $V'$ a smooth
vector field tangent to $\pa \overline{\RR^n}$. In fact, when $a_\alpha$
is classical of order $0$, such $V$ correspond {\em exactly} to the vector
fields on $\overline{\RR^n}$ of the form $xV'$, $V'$ a smooth
vector field tangent to $\pa \overline{\RR^n}$. We use the notation
$\Vsc(\overline{\RR^n})$ for the collection of these vector fields on
$\overline{\RR^n}$. The corresponding scattering differential operators are
denoted by $\Diffsc(\overline{\RR^n})$, and the scattering
pseudodifferential operators by
$\Psisc^{m,l}(\overline{\RR^n})$. Finally, the weighted Sobolev spaces
become weighted scattering Sobolev spaces,
$\Hsc^{s,r}(\overline{\RR^n})=H^{s,r}$; for $s\geq 0$ integer thus
elements are tempered distributions $u$ with $x^{-r}V_1\ldots V_k u\in
L^2(\RR^n)$ for all $V_j\in\Vsc(\overline{\RR^n})$, $1\leq j\leq k$
and $k\leq s$ (including $k=0$).

If $a\in S^{0,0}$ is classical (both in the $z$ and $\zeta$ sense),
i.e.\ it is (under the identification above) an element of
$\CI(\overline{\RR^n_z}\times \overline{\RR^n_\zeta})$, the principal
symbol $[a]$ can be considered as the restriction of $a$ to
$$
\pa(\overline{\RR^n_z}\times \overline{\RR^n_\zeta})=(\overline{\RR^n_z}\times \pa\overline{\RR^n_\zeta})\cup
(\pa\overline{\RR^n_z}\times \overline{\RR^n_\zeta}),
$$
since if its
restriction to the boundary vanishes then $a\in S^{-1,-1}$. Here
$\overline{\RR^n_z}\times \pa\overline{\RR^n_\zeta}$ is {\em fiber
  infinity} and $\pa\overline{\RR^n_z}\times \overline{\RR^n_\zeta}$
is {\em base infinity}. Then the principal
symbol of $\Op(a)\Op(b)$ is $ab$. The case of general orders $m,l$ can
be reduced to this by removing fixed elliptic factors, such as
$\langle\zeta\rangle^m\langle z\rangle^l$. The commutator version is
that is $a\in S^{1,1}$, classical, then $H_a$ is a smooth vector field
on $\overline{\RR^n_z}\times \overline{\RR^n_\zeta}$ tangent to all
boundary faces. In general, we define the rescaled Hamilton vector
field $\scH_a$ by removing the
elliptic factor $\langle\zeta\rangle^{m-1}\langle z\rangle^{l-1}$:
$$
\scH_a=\langle\zeta\rangle^{-m+1}\langle z\rangle^{-l+1}H_a.
$$

In addition to the leading order behavior captured by the principal 
symbol, one can also talk about the behavior of $a$ modulo 
$S^{-\infty,-\infty}$ microlocally; this is most natural from our
compactified perspective. Thus, the {\em operator wave front set},
$\WFsc'(\Op(a))$, is a subset of $\pa(\overline{\RR^n_z}\times
\overline{\RR^n_\zeta})$, with a point $\alpha\in \pa(\overline{\RR^n_z}\times
\overline{\RR^n_\zeta})$ {\em not being} in $\WFsc'(\Op(a))$ if there exists
a neighborhood of $\alpha$ in $\overline{\RR^n_z}\times
\overline{\RR^n_\zeta}$ restricted to which $a$ is in
$S^{-\infty,-\infty}$. This notion then possesses the usual properties
of wave front sets, for instance
$$
\WFsc'(\Op(a)\Op(b))\subset\WFsc'(\Op(a))\cap\WFsc'(\Op(b)).
$$
In the same vein, one can talk about {\em ellipticity} at a point $\alpha\in \pa(\overline{\RR^n_z}\times
\overline{\RR^n_\zeta})$, meaning that $a$ is invertible, in
$S^{-m,-l}$, when restricted to a neighborhood of $\alpha$.

One similarly has a {\em wave front set} $\WFsc(u)$ for tempered
distributions $u$: $\alpha\in \pa(\overline{\RR^n_z}\times
\overline{\RR^n_\zeta})$ is {\em not} in $\WFsc(u)$ if there is a
symbol $a\in S^{0,0}$ such that $a$ is elliptic at $\alpha$ and
$\Op(a)u$ is Schwartz.

The extension of $\Psisc(\overline{\RR^n})$ to manifolds with
boundary $X$, with the result denoted by $\Psisc(X)$, is then via
local coordinate charts, identifying open sets of $X$ and
$\overline{\RR^n}$ (as in the standard theory of pseudodifferential
operators on manifolds for $X^\circ$ and $\RR^n$), with the following additional
requirement. When we restrict the Schwartz kernel of any element of
$\Psisc(X)$ to the product of disjoint open sets in the
left and right factors $X$ of $X\times X$, it vanishes
to infinite order at the boundary of either factor, i.e.\ is, when
localized to such a product, in $\dCI(X\times X)$. 
Note that open subsets of $\overline{\RR^n}$ near
$\pa \overline{\RR^n}$ behave like asymptotic cones in view of the
compactification. Notice that in the context of our problem this means
that even though for $g$, $\{x=0\}$ is at a `finite' location (finite
distance from $\pa M$, say), analytically we push it to infinity by
using the scattering algebra. Returning to the general discussion, one also needs to allow vector bundles; this is done
as for standard pseudodifferential operators, using local
trivializations, in which one simply has a matrix of scalar
pseudodifferential operators. For more details in the present context
we refer to \cite{UV:local,SUV:Tensor}. For a complete discussion we
refer to \cite{RBMSpec} and to \cite{Vasy:Minicourse}.

This is also a good point to introduce the notation $\Vb(X)$ on a
manifold with boundary: this is the collection, indeed Lie algebra, of
smooth vector fields on $X$ tangent to $\pa X$. Thus,
$\Vsc(X)=x\Vb(X)$ if $x$ is a boundary defining function of $X$. This
class will play a role in the appendix. Note that if $y_j$ are local
coordinates on $\pa X$, $j=1,\ldots,n-1$, then
$x\pa_x,\pa_{y_1},\ldots,\pa_{y_{n-1}}$ are a local basis of elements of $\Vb(X)$, with $\CI(X)$ coefficients; the analogue for
$\Vsc(X)$ is $x^2\pa_x,x\pa_{y_1},\ldots,x\pa_{y_{n-1}}$. These vector
fields are
then exactly the local sections of vector bundles $\Tb X$, resp.\ $\Tsc X$, with the
same bases. The dual bundles $\Tb^* X$, resp.\ $\Tsc^* X$, then have
bases $\frac{dx}{x},dy_1,\ldots,dy_{n-1}$, resp.\
$\frac{dx}{x^2},\frac{dy_1}{x},\ldots,\frac{dy_{n-1}}{x}$. Thus,
scattering covectors have the form $\xi
\frac{dx}{x^2}+\sum_{j=1}^{n-1}\eta_j\frac{dy_j}{x}$. Tensorial
constructions apply as usual, so for instance one can construct
$\Sym^2\Tsc^* X$; for $p\in X$, $\alpha\in \Sym^2\Tsc^* X$ gives a
bilinear map from $\Tsc_p X$ to $\Cx$. Notice also that with this
notation $\scH_a$
is an element of $\Vb(\overline{\RR^n_z}\times
\overline{\RR^n_\zeta})$, or in general
$\scH_a\in\Vb(\overline{\Tsc^*}X)$, where $\overline{\Tsc^*}X$ is the
{\em fiber-compactification} of $\Tsc^*X$, i.e.\ the fibers of
$\Tsc^*X$ (which can be identified with $\RR^n$) are compactified as
$\overline{\RR^n}$. Again, see \cite{Vasy:Minicourse} for a more
detailed discussion in this context.

\subsection{The tensorial operator $L$}\label{subsec:L0}
In \cite{SUV:Tensor}, with $v$ still a locally defined function on the space of geodesics, 
for one-forms we considered the map $L$
\begin{equation}\label{eq:L-forms}
L v(z)=\int \chi(\lambda/x)v(\gamma_{x,y,\lambda,\omega})g_{\scl}(\lambda\,\pa_x+\omega\,\pa_y)\,d\lambda\,d\omega,
\end{equation}
while for 2-tensors
\begin{equation}\label{eq:L-tensors}
L v(z)=x^{2}\int \chi(\lambda/x)v(\gamma_{x,y,\lambda,\omega})g_{\scl}(\lambda\,\pa_x+\omega\,\pa_y)\otimes g_\scl(\lambda\,\pa_x+\omega\,\pa_y)\,d\lambda\,d\omega,
\end{equation}
so in the two cases $L$ maps into one-forms, resp.\ symmetric
2-cotensors. Here
$g_{\scl}$, {\em of no relation to $g$}, is a {\em scattering metric} (smooth section of
$\Sym^2\Tsc^*X$) used to
convert vectors into covectors, of the form
$$
g_\scl=x^{-4}\,dx^2+x^{-2} h,
$$
with $h$ being a boundary metric in a warped product decomposition of
a neighborhood of the boundary. Recall that the Euclidean metric becomes
such  a scattering metric when $\RR^n$ is radially compactified;
indeed, this was the reason for Melrose's introduction of this
pseudodifferential algebra: generalizing asymptotically Euclidean metrics. While the product decomposition near
$\pa X$ relative to which $g_{\scl}$ is a warped product
did not need to have any relation to the underlying metric $g$ we are
interested in, in our normal gauge discussion we use $g_{\scl}$ which
is warped product in the product decomposition in which $g$ is in a
normal gauge.

We note here that geodesics of a scattering metric $g_\scl$ are the projections
to $X$ of the integral curves
of the Hamilton vector field $H_{g_\scl}$; it is actually better to
consider $\scH_{g_\scl}$ (which reparameterizes these), for one has a
non-degenerate flow on $\Tsc^*X$ (and indeed
$\overline{\Tsc^*}X$). 
Note that if one is interested in {\em finite
  points} at base infinity, i.e.\ points in $\Tsc^*_{\pa X}X$, it
suffices to renormalize $H_{g_\scl}$ by the weight, i.e.\ consider
$x^{-1}H_{g_\scl}$ which we also denote by  $\scH_{g_\scl}$.

With $L$ defined as in \eqref{eq:L-forms}-\eqref{eq:L-tensors}, it is shown in \cite{SUV:Tensor} that the
exponentially conjugated operator
$$
N_\digamma=e^{-\digamma/x}L I e^{\digamma/x}
$$
is an element of
$\Psisc^{-1,0}(X)$ (with values in $\Tsc^*X$ or $\Sym^2\Tsc^*X$), and for (sufficiently large, in the case of two
tensors) $\digamma>0$, it is elliptic both at finite points at spatial
infinity $\pa X$, i.e.\ points in $\Tsc^*_p X$, $p\in\pa X$, and at
fiber infinity on the kernel of the principal symbol of the adjoint,  relative to $g_{\scl}$, of
the conjugated
symmetric gradient
$$
\dsymmw=e^{-\digamma/x}\dsymm e^{\digamma/x}
$$
of $g$ (so $\dsymm$ is the symmetric gradient of $g$), namely on the kernel of the
principal symbol of
$$
\delta^s_\digamma=e^{\digamma/x}\delta^s e^{-\digamma/x},\qquad
\delta^s=(\dsymm)^*.
$$
This allows one to conclude
that
$$
N_\digamma+\dsymmw Q\delta_\digamma^s \in\Psisc^{-1,0}(X;\Sym^2\Tsc^*X,\Sym^2\Tsc^*X)
$$
is elliptic, over a
neighborhood of $\Omega$ (which is what is relevant), for suitable
$Q\in\Psisc^{-3,0}(X;\Tsc^*,\Tsc^*X)$. The rest of \cite{SUV:Tensor}
deals with arranging the solenoidal gauge and using the parametrix for
this elliptic operator; this actually involves two extensions from
$\Omega$. It also uses that when $c>0$ used in defining
$\Omega$ is small, the error of the parametrix when sandwiched between
relevant cutoffs arising from the extensions is small, and thus the
appropriate error term can actually be removed by a convergent Neumann
series. The reason this smallness holds is that, similarly to the discussion
in the scalar setting in \cite{UV:local}, the map
$$
c\mapsto
N_\digamma+d_\digamma^s Q\delta_\digamma^s\in\Psisc^{-1,0}(X_c)
$$
is continuous, meaning that if one takes a fixed space, say $X_0$,
and identifies $X_c$ (for $c$ small) with it via a translation, then
the resulting map into $\Psisc^{-1,0}(X_0)$ is
continuous. Furthermore, the ellipticity (over a fixed neighborhood of the
image of $\Omega_c$) also holds uniformly in $c$,
and thus one has a parametrix with an error which is uniformly bounded
in $\Psisc^{-\infty,-\infty}(X_0)$, thus when localized to $x<c$ (the
image of $\Omega_c$ under the translation) it is bounded by a
constant multiple of $c$ in any weighted Sobolev operator norm, and
thus is small when $c$ is small.

As in the proof of boundary rigidity in the fixed conformal class setting of
\cite{SUV_localrigidity}, it is also important to see how
$N_\digamma$ (and $\dsymmw Q\delta_\digamma^s$) depend on the
metric $g$. 
Completely analogously to the scalar case, see \cite[Proposition~3.2]{SUV_localrigidity} and the remarks preceding it connecting $g$ to $\Gamma_\pm$ in the notation of that paper, we have the following. 
That dependence is  continuous in the same sense as above, as long as
$g$ is close in a $C^k$-sense (for suitable $k$) to a fixed metric $g_0$ (in the region
we are interested in), i.e.\ any seminorm in $\Psisc^{-1,0}(X_0)$ is
controlled by some seminorm of $g$ in $\CI$ in the relevant region.

\subsection{Ellipticity of $N_{\digamma}$ at finite points, i.e.\
  at points in $\Tsc^*_{\pa X}X$}
An inspection of the proof of \cite[Lemma~3.5]{SUV:Tensor} shows that $N_{\digamma}$ is
elliptic at finite points even on tangential tensors (the kernel of
the restriction to the normal component, rather than the kernel of the
principal symbol of $\delta^s_\digamma$); in the case of symmetric
2-cotensors this holds for sufficiently large $\digamma>0$ as in Lemma~3.5 of \cite{SUV:Tensor}. Indeed, in the case of
one-forms, in Lemma~3.5 of \cite{SUV:Tensor} the principal symbol of $N_{\digamma}$ (at $x=0$) is
calculated to be (see also the next paragraph below regarding how this computation proceeds)
\begin{equation}\begin{aligned}\label{eq:one-form-exp-3}
&(\xi^2+\digamma^2)^{-1/2}\\
&\int_{\sphere^{n-2}}\nu^{-1/2}
\begin{pmatrix}-\frac{\nu(\xi+i\digamma)}{\xi^2+\digamma^2}(\hat Y\cdot\eta)\\\hat
  Y\end{pmatrix}\otimes\begin{pmatrix}-\frac{\nu(\xi-i\digamma)}{\xi^2+\digamma^2}(\hat
  Y\cdot\eta)&\langle\hat
  Y,\cdot\rangle\end{pmatrix}
e^{-(\hat
  Y\cdot\eta)^2/(2\nu(\xi^2+\digamma^2))}\,d\hat Y
\end{aligned}\end{equation}
 for an appropriate choice of $\chi$ (exponentially
decaying, not compactly supported, which is later fixed, as discussed
below), up to an
overall elliptic factor, {\em and in coordinates in which at the point   $y$,
where the symbol is computed, the metric $h$ is the Euclidean metric}. 
Here the block-vector notation corresponds to the decomposition into
normal and tangential components, and where $\nu=\digamma^{-1}\alpha$,
$\alpha=\alpha(0,y,0,\hat Y)$, $\alpha$ as in \eqref{eq:geod-form}. Thus, this
is a superposition of positive (in the sense of non-negative)
operators, which is thus itself positive. Moreover, when restricting
to tangential forms, i.e.\ those with vanishing first components, and
projecting to the tangential components, we get
\begin{equation}\begin{aligned}\label{eq:one-form-exp-4}
&(\xi^2+\digamma^2)^{-1/2}\int_{\sphere^{n-2}}\nu^{-1/2}
\hat
  Y\otimes\langle\hat
  Y,\cdot\rangle
e^{-(\hat
  Y\cdot\eta)^2/(2\nu(\xi^2+\digamma^2))}\,d\hat Y,
\end{aligned}\end{equation}
which is positive definite: indeed, it is certainly non-negative, and when applied to $v$, if $v\neq 0$ is tangential, taking $\hat
Y=v/|v|$ shows the non-vanishing of the integral. The case of
symmetric 2-cotensors is similar; when restricted to
tangential-tangential tensors one simply needs to replace $\hat
  Y\otimes\langle\hat
  Y,\cdot\rangle$ by its analogue $(\hat
  Y\otimes\hat Y)\otimes\langle\hat
  Y\otimes\hat Y,\cdot\rangle$; since tensors of the form $\hat
  Y\otimes\hat Y$ span all tangential-tangential tensors, the
  conclusion follows. Note that one actually has to approximate a
  $\chi$ of compact support by these exponentially decaying $\chi=\chi_0$,
  e.g.\ via taking $\chi_k=\phi(./k)\chi_0$, $\phi\geq 0$ even
  identically $1$ near $0$, of compact support, and letting
  $k\to\infty$; we then have that the principal symbols of the
  corresponding operators converge; thus given any {\em compact}
  subset of $\Tsc^*_{\pa X} X$, for sufficiently large $k$ the
  operator given by $\chi_k$ is elliptic. (This issue does not arise
  in the setting of \cite{SUV:Tensor}, for there one also has
  ellipticity at fiber infinity, thus one can work with the fiber
  compactified cotangent bundle, $\overline{\Tsc^*}_{\pa X} X$.) Of
  course, once we arrange appropriate estimates at fiber infinity to
  deal with the lack of ellipticity of the principal symbol there in
  the current setting (tangential forms/tensors), the estimates also
  apply in a neighborhood of fiber infinity, thus this compact subset
  statement is sufficient for our purposes.

\subsection{The Schwartz kernel of scattering pseudodifferential operators}
Given the results just recalled, it remains to consider the principal symbol, and ellipticity, at fiber infinity. In
\cite{UV:local,SUV:Tensor} this was analyzed using the explicit
Schwartz kernel; indeed this was already the case for the analysis at
finite points considered in the previous paragraph.
In order to connect the present paper with these earlier works we first recall some notation. Instead of the
oscillatory integral definition (via localization, in case of a
manifold with boundary) discussed above, $\Psisc(X)$ can be equally
well characterized by the statement that the Schwartz kernel of
$A\in\Psisc(X)$, which is a priori a tempered distribution on $X^2$,
is a conormal distribution on a certain resolution of $X^2$, called
the scattering double space $X^2_\scl$; again this was introduced by Melrose in \cite{RBMSpec}. Here conormality is both to
the (lifted) diagonal and to the boundary hypersurfaces, of which only
one sees non-trivial, i.e.\ non-infinite order vanishing, behavior, namely the
scattering front face. In order to make this more concrete, we
consider coordinates $(x,y)$ on $X$, $x$ a (local) boundary defining
function and $y=(y_1,\ldots,y_{n-1})$ as
before, and write the corresponding coordinates on $X^2=X\times X$ as
$(x,y,x',y')$, i.e.\ the primed coordinates are the pullback of
$(x,y)$ from the second factor, the unprimed from the first factor. Coordinates on $X^2_\scl$ near the scattering
front face then are
$$
x,\ y,\ X=\frac{x'-x}{x^2},\ Y=\frac{y'-y}{x},\ x\geq 0;
$$
the lifted diagonal is $\{X=0,\ Y=0\}$, while the scattering front
face is $x=0$. In \cite{UV:local,SUV:Tensor} the lifted diagonal was
also blown up, which essentially means that `invariant spherical coordinates' were
introduced around it. Thus, the conormal singularity to the diagonal, which
corresponds to the exponential conjugate of $L_0I$ being a pseudodifferential operator of order
$-1$, becomes a conormal singularity at the new front
face. Concretely, in the region where $|Y|>c|X|$, $c>0$ fixed (but arbitrary), which is the
case on the support of $L_0 I$ for sufficiently small $c$ when the cutoff $\chi$ is compactly
supported, valid `coordinates' ($\hat Y$ below is in $\sphere^{n-2}$) are
\begin{equation}\label{eq:X2sc-blowup-coords}
x,\ y,\ \frac{X}{|Y|},\ \hat Y=\frac{Y}{|Y|},\ |Y|.
\end{equation}
In these coordinates $|Y|=0$ is the new front face, namely the lifted
diagonal, and $x=0$ is still the scattering front face, and
$\big|\frac{X}{|Y|}\big|=\frac{|X|}{|Y|}<c$ in the region of interest. The principal
symbol at base infinity, $x=0$, of an operator $A\in\Psisc^{m,0}(X)$,
evaluated at $(0,y,\xi,\eta)$,
is simply the $(X,Y)$-Fourier transform of the restriction of its Schwartz
kernel to the scattering front face, $x=0$, evaluated at
$(-\xi,-\eta)$; the computation giving \eqref{eq:one-form-exp-3} and
its 2-tensor analogue is exactly the computation of this Fourier transform.

We also introduce the notation
$$
S=\frac{X-\alpha(\hat Y)|Y|^2}{|Y|},\ \hat Y=\frac{Y}{|Y|},
$$
and remark that $S$ is a smooth function of the coordinates in
\eqref{eq:X2sc-blowup-coords}. Then
the Schwartz kernel of $N_{\digamma}$ at the  
scattering front face $x=0$ is, as in \cite[Lemma~3.4]{SUV:Tensor},
given by
\begin{equation*}\begin{aligned}
&e^{-\digamma X}|Y|^{-n+1}\chi(S)\Big(\Big(S\frac{dx}{x^2}+\hat
Y\cdot\, \frac{dy}{x}\Big) \Big((S+2\alpha|Y|)(x^2\pa_x)+\hat Y\cdot(x\pa_{y})\Big) \Big)
\end{aligned}\end{equation*}
on one forms, respectively
\begin{equation*}\begin{aligned}
&e^{-\digamma X}|Y|^{-n+1}\chi(S)\\
&\qquad \Big(\Big(\Big(S\frac{dx}{x^2}+\hat
Y\cdot\, \frac{dy}{x}\Big) \otimes \Big(\Big(S\frac{dx}{x^2}+\hat
Y\cdot\, \frac{dy}{x}\Big)\Big)\Big)
\Big)\\
&\qquad\qquad\Big(\Big((S+2\alpha|Y|) (x^2\pa_x)+\hat Y\cdot(x\pa_{y})\Big)\otimes\Big((S+2\alpha|Y|) (x^2\pa_x)+\hat Y\cdot(x\pa_{y})\Big) \Big)
\end{aligned}\end{equation*}
on 2-tensors,
where $\hat Y$ is regarded as a tangent vector which acts on
covectors. Here
$$
(S+2\alpha|Y|) (x^2\pa_x)+\hat Y\cdot(x\pa_{y})
$$
maps
one forms to scalars, thus
$$
\Big((S+2\alpha|Y|) (x^2\pa_x)+\hat
Y\cdot(x\pa_{y})\Big)\otimes \Big((S+2\alpha|Y|) (x^2\pa_x)+\hat Y\cdot(x\pa_{y})\Big)
$$
maps symmetric 2-tensors to scalars, while
$S\frac{dx}{x^2}+\hat Y\cdot\, \frac{dy}{x}$ maps scalars to one
forms, so
$$
\Big(S\frac{dx}{x^2}+\hat Y\cdot\, \frac{dy}{x}\Big)\otimes
\Big(S\frac{dx}{x^2}+\hat Y\cdot\, \frac{dy}{x}\Big)
$$
maps scalars to symmetric 2-tensors.
In order to make the notation less confusing, we employ a matrix notation,
\begin{equation*}\begin{aligned}
&\Big(S\frac{dx}{x^2}+\hat Y\cdot\, \frac{dy}{x}\Big)
\Big((S+2\alpha|Y|) (x^2\pa_x)+\hat
Y\cdot(x\pa_{y})\Big)\\
&\qquad=\begin{pmatrix}S (S+2\alpha|Y|)&S\langle\hat
  Y,\cdot\rangle\\\hat Y (S+2\alpha|Y|)&\hat Y \langle\hat
  Y,\cdot\rangle\end{pmatrix},
\end{aligned}\end{equation*}
with the first column and row corresponding to $\frac{dx}{x^2}$, resp.\
$x^2\pa_x$, and the second column and row to the (co)normal vectors.
For 2-tensors, as before, we use a decomposition
$$
\frac{dx}{x^2}\otimes\frac{dx}{x^2},\ \frac{dx}{x^2}\otimes\,\frac{dy}{x},\
\frac{dy}{x}\otimes\frac{dx}{x^2},\ \frac{dy}{x}\otimes \frac{dy}{x},
$$
where the symmetry of the 2-tensor is the statement that the 2nd
and 3rd (block) entries are the same.  For the actual
endomorphism
we write
\[
\begin{aligned} 
&\begin{pmatrix}S^2\\S\langle\hat
  Y,\cdot\rangle_1\\S\langle\hat
  Y,\cdot\rangle_2\\\langle \hat Y,\cdot\rangle_1 \langle \hat
  Y,\cdot\rangle_2\end{pmatrix}
\begin{pmatrix}(S+2\alpha|Y|)^2\hat Y_1\hat Y_2&(S+2\alpha|Y|)\hat
  Y_1\hat Y_2 \langle \hat Y,\cdot\rangle_1&(S+2\alpha|Y|)\hat Y_1\hat
  Y_2 \langle \hat Y,\cdot\rangle_2&\hat Y_1\hat Y_2 \langle \hat
  Y,\cdot\rangle_1 \langle \hat Y,\cdot\rangle_2\end{pmatrix}.
\end{aligned}
\]
Here we write subscripts $1$ and $2$ for clarity on $\hat Y$ to denote
whether it is acting on the first or the second factor, though this
also immediately follows from its position within the matrix.

In the next two sections we further analyze these operators first in
the 1-form, and then in the 2-tensor setting, although the oscillatory
integral approach will give us the precise results we need.

\section{One-forms and Fredholm theory in the normal gauge}\label{sec:1-form}
We first consider the X-ray transform on 1-forms in the normal
gauge. The overall form of the transform is similar in the 2-tensor
case, but it is more delicate since it is not purely dependent on a
principal symbol computation, so the 1-form transform will be a
useful guide.

Since we intend to work with tangential forms and tensors, we start
by defining $L_0$ analogously to $L$, but without the normal component
in the output. Thus,
\begin{equation}\label{eq:L0-forms}
L_0 v(z)=\int \chi(\lambda/x)v(\gamma_{x,y,\lambda,\omega})g_{\scl}(\omega\,\pa_y)\,d\lambda\,d\omega,
\end{equation}
while for 2-tensors
\[
L_0 v(z)=x^{2}\int \chi(\lambda/x)v(\gamma_{x,y,\lambda,\omega})g_{\scl}(\omega\,\pa_y)\otimes g_\scl(\omega\,\pa_y)\,d\lambda\,d\omega.
\]
Hence in the two cases $L_0$ maps into {\em tangential} one-forms, resp.\
{\em tangential-tangential} symmetric
2-cotensors,
where $g_{\scl}$ is a scattering metric (smooth section of
$\Sym^2\Tsc^*X$) used to
convert vectors into covectors, of the form
$$
g_\scl=x^{-4}\,dx^2+x^{-2} h,
$$
with $h$ being a boundary metric in a warped product decomposition of
a neighborhood of the boundary, and with $g_\scl$ of no relation to $g$. Then we have
$$
g_\scl(\omega\,\pa_y)=x^{-2} h(\omega\,\pa_y),
$$
explaining the appearance of the diverse powers of $x$ in the above
formulae. 
In other words, $L_0$ is the composition of $L$, see \eqref{eq:L-forms} and \eqref{eq:L-tensors},  with
projection to the tangential forms, resp.\ tangential-tangential
tensors, using the product structure.

Then we define
$$
N_{0,\digamma}=e^{-\digamma/x}L_0 I e^{\digamma/x}
$$
acting on {\em tangential} one forms, resp.\ symmetric
2-tensors. Thus, $N_{0,\digamma}$ is the restriction of $N_{\digamma}$
to tangential one forms or two tensors, composed with projection to the tangential forms, resp.\ tangential-tangential
tensors.

\subsection{The lack of ellipticity of the principal symbol in the
  one-form case}
The standard principal symbol of $N_{\digamma}$ is that of the conormal singularity
at the diagonal, i.e.\ $X=0$, $Y=0$. Writing $(X,Y)=Z$,
$(\xi,\eta)=\zeta$, we would need
to evaluate the $Z$-Fourier transform of the Schwartz kernel of $N_{\digamma}$ as $|\zeta|\to\infty$. This
was discussed in \cite{UV:local} around Equation~(3.8), including
connecting it to the earlier computation of Stefanov and Uhlmann
\cite{SU-Duke}. Concretely, the leading
order behavior, as $|\zeta|\to\infty$, of this Fourier
transform can be obtained by working on the
blown-up space of the diagonal, with coordinates $|Z|,\hat
Z=\frac{Z}{|Z|}$ (as well as $z=(x,y)$), and integrating the
restriction of the Schwartz kernel to the front face, $|Z|^{-1}=0$,
after removing the singular factor $|Z|^{-n+1}$, along the {\em
  equatorial sphere} corresponding to $\zeta$, and given by $\hat
Z\cdot\zeta=0$. Now, in our setting, in view of the
infinite order vanishing, indeed compact support, of the Schwartz
kernel as $X/|Y|\to\infty$ (and $Y$ bounded),
we may work in semi-projective coordinates, i.e.\ in spherical
coordinates in $Y$, but $X/|Y|$ as the additional tangential variable,
$|Y|$ the defining function of the front face. The equatorial
sphere then becomes $(X/|Y|)\xi+\hat Y\cdot\eta=0$, with the integral
relative to an appropriate positive density. With $\tilde
S=X/|Y|$, keeping in mind that terms with extra vanishing factors at
the front face, $|Y|=0$ can be dropped, we thus need to integrate
\[
\begin{pmatrix}\tilde S^2&\tilde S\langle\hat  
  Y,\cdot\rangle\\ \tilde S\hat Y&\hat Y \langle\hat  
  Y,\cdot\rangle\end{pmatrix}\chi(\tilde S)=\begin{pmatrix} \tilde S\\
  \hat Y\end{pmatrix}\otimes\begin{pmatrix}\tilde S&\hat Y\end{pmatrix}\chi(\tilde S),
\]
on this equatorial sphere in the case of one-forms. Now, for $\chi\geq 0$ this matrix is a
positive multiple of the projection to the span of $(\tilde S,\hat
Y)$. As $(\tilde S,\hat Y)$ runs through the $(\xi,\eta)$-equatorial
sphere, we are taking a positive (in the sense of non-negative) linear
combination of the projections to the span of the vectors in this
orthocomplement, with the weight being strictly positive as long as
$\chi(\tilde S)>0$ at the point in question.

Now, for tangential one forms, if we project the result to tangential
one forms, {\em i.e.\ if we replace $N_{\digamma}$ by
  $N_{0,\digamma}$}, this matrix simplifies to
\[
\hat Y \langle\hat  
  Y,\cdot\rangle\chi(\tilde S).
\]
Hence, working at a point $(0,y,\xi,\eta)$ (considered as a
homogeneous object, i.e.\ we are working at fiber infinity) if we show that for
each non-zero tangential vector $w$ there is at least one $(\tilde S,\hat Y)$
with $\chi(\tilde S)>0$ and $\xi\tilde S+\eta\cdot \hat Y=0$ and
$\hat Y \cdot w\neq 0$, we
conclude that the integral of the projections is positive, thus the principal
symbol of our operator is elliptic, on tangential forms. But this is straightforward
if $\chi(0)>0$ {\em and $\xi\neq 0$}:
\begin{enumerate}
\item
if $w\neq 0$ and $w$ is not a multiple of $\eta$, then take $\hat Y$
orthogonal to $\eta$ but not to $w$, $\tilde S=0$,
\item
if $w=c\eta$ with $w\neq 0$ (so $c$ and $\eta$ do not vanish) then
$\hat Y\cdot w=c\hat Y\cdot\eta=-c\xi\tilde S$ under the constraint so
we need non-zero $\tilde S$; but fixing any non-zero $\tilde S$
choosing $\hat Y$ such that $\hat Y\cdot\eta=-\xi\tilde S$
(such
$\hat Y$ exists again as $\eta\in\RR^{n-1}$, $n\geq 3$), $\hat
Y\cdot w\neq 0$ follows. We thus choose $\tilde S$
small enough in order to ensure $\chi(\tilde S)>0$, and apply this
argument to find $\hat Y$.
\end{enumerate}
This shows that the principal symbol is positive definite on
tangential one-forms for $\xi\neq 0$; indeed it shows that on
$\Span\{\eta\}^\perp$, the subspace of $\RR^{n-1}$ orthogonal to $\eta$, we also have positivity even if $\xi=0$. Notice
that if we restrict to $\Span\{\eta\}^\perp$, but do not project the
result to $\Span\{\eta\}^\perp$, the $\Span\{\eta\}$ component
actually vanishes at
$\xi=0$ as the integral is over $\hat Y$ with $\hat Y\cdot\eta=0$,
i.e.\ with $\Pi^\perp$ the projection to $\Span\{\eta\}^\perp$,
$\sigma_{-1,0}(N_{0,\digamma}) \Pi^\perp=\Pi^\perp\sigma_{-1,0}(N_{0,\digamma})
\Pi^\perp$. On the other hand, still for $\xi=0$, with $\Pi^\parallel$ to projection to
$\Span\{\eta\}$,
as the integral is over $\hat Y$ with $\hat Y\cdot\eta=0$, $\sigma_{-1,0}(N_{0,\digamma}) \Pi^\parallel=0$.
Thus, in the decomposition of tangential
covectors into $\Span\{\eta\}^\perp\oplus\Span\{\eta\}$, $\sigma_{-1,0}(N_{0,\digamma})$ (mapping into $\Span\{\eta\}^\perp$) has matrix of
the form, with $O$ denoting behavior as $\xi\to 0$,
$$
\begin{pmatrix} O(1)&O(\xi)\\O(\xi)&O(\xi)\end{pmatrix},
$$
where all terms are order $(-1,0)$ (so they have appropriate elliptic
prefactors) and the $O(1)$ term is elliptic.  In fact, the (1,1) term $\Pi^\parallel\sigma_{-1,0}(N_{0,\digamma})\Pi^\parallel$
is non-negative, so it necessarily is $O(\xi^2)$!
Thus, the difficulty in obtaining a non-degenerate
problem is $\Span\{\eta\}$ when $\xi=0$.

\subsection{The operator $\tilde L_1$: first version}
To deal with $\Span\{\eta\}$ when $\xi=0$, we also consider another
operator. For this purpose it is convenient to replace $\chi$ by a
function $\chi_1$ which
is {\em not} even. It is straightforward to check how this affects the
computation of the principal symbol at fiber infinity: one has to
replace the result by a sum over $\pm$ signs, where both $\hat Y$ and
$S$ are evaluated with both the $+$ sign and the $-$ sign.
Thus, for
instance the Schwartz kernel of $N_{\digamma}$ on one-forms is at
the scattering front face
\begin{equation*}\begin{aligned}
&\sum_\pm e^{-\digamma X}|Y|^{-n+1}\chi_1(\pm S)\Big(\Big(\pm S\frac{dx}{x^2}\pm\hat
Y\cdot\, \frac{dy}{x}\Big) \Big(\pm ( S+2\alpha|Y|)(x^2\pa_x)\pm\hat Y\cdot(x\pa_{y})\Big) \Big).
\end{aligned}\end{equation*}
Here the $\pm$ are all the same, thus the cancel out in the product,
and one is left with $\sum_\pm \chi_1(\pm S)$ times an expression
independent of the choice of $\pm$, i.e.\ only the even part
of $\chi_1$ enters into $N_{\digamma}$ and thus non-even $\chi_1$ are
not interesting for our choice of $L$. Thus, we need to modify
the form of $L$ as well; concretely consider $\tilde L_1$ defined by
\[
\tilde L_1 v(z)=x^{-1}\int \chi_1(\lambda/x)v(\gamma_{x,y,\lambda,\omega})\,d\lambda\,d\omega,
\]
which maps into the scalars! Here the power of $x$ in
front is one lower than that of $L$ on one forms (which is $x^0=1$),
because, as discussed in \cite{SUV:Tensor}, both factors of $\dot\gamma$ in
$I$, which are still present, and $g_{\scl}(\dot\gamma)$, which are no
longer present, give rise to factors of $x^{-1}$ in the integral
expression, and we normalize them by putting the corresponding power
of $x$ into the definition of $L$, with the function case having an
$x^{-2}$ due to the localization itself. Then the Schwartz kernel of
$$
\tilde N_{1,\digamma}=e^{-\digamma/x}\tilde L_1
I e^{\digamma/x}
$$
on the scattering front face is, for not necessarily even $\chi_1$,
\begin{equation*}\begin{aligned}
&\sum_\pm e^{-\digamma X}|Y|^{-n+1}\chi_1(\pm S) \Big(\pm (
S+2\alpha|Y|)(x^2\pa_x)\pm\hat Y\cdot(x\pa_{y})\Big)\\
&=e^{-\digamma X}|Y|^{-n+1} (\chi_1(S)-\chi_1(-S)) \Big((
S+2\alpha|Y|)(x^2\pa_x)+\hat Y\cdot(x\pa_{y})\Big),
\end{aligned}\end{equation*}
so now {\em odd} $\chi_1$ give non-trivial results. In particular, on
tangential one-forms this is 
$$
e^{-\digamma X}|Y|^{-n+1} (\chi_1(S)-\chi_1(-S)) \hat Y\cdot(x\pa_{y}).
$$
The corresponding principal symbol at fiber infinity is still the
integral over the equatorial sphere $\xi\tilde S+\eta\cdot\hat Y=0$ of
$$
(\chi_1(\tilde S)-\chi_1(-\tilde S)) \hat Y
$$
up to an overall elliptic factor. Applied to elements of
$\Span\{\eta\}$, restricted to the equatorial sphere, this is
$$
(\chi_1(\tilde S)-\chi_1(-\tilde S)) \xi\tilde S,
$$
which is twice the even part of $\tilde S\chi_1(\tilde S)$ times $\xi$. Thus,
for odd $\chi_1$, as long as $\chi_1(\tilde S)>0$ for some $\tilde S>0$
and $\chi_1\geq 0$ on $(0,\infty)$, the principal symbol at fiber
infinity, restricted to $\Span\{\eta\}$, is a positive multiple of
$\xi$ (up to an overall elliptic factor). On the other hand, at
$\xi=0$, the integral is simply over $\hat Y$ orthogonal to $\eta$,
and the integral vanishes as the integrand is odd in $\hat
Y$. Correspondingly, in the decomposition
$\Span\{\eta\}^\perp\oplus\Span\{\eta\}$,
$\sigma_{-1,0}(\tilde N_{1,\digamma})$ at fiber infinity is an elliptic multiple of
\[
\begin{pmatrix} b\xi&a\xi\end{pmatrix}
\]
with $a>0$.

\subsection{The operator $L_1$: second version}
There is a different way of arriving at the operator $\tilde L_1$, or rather
a very similar operator $L_1$ which works equally well. Namely, if
one considers $L$ as a map restricted to tangential one forms, but,
unlike $L_0$, mapping not into tangential forms but {\em all}
one-forms, without projecting out the normal,
$\frac{dx}{x^2}$, component, the normal projection $L_1$ of $L$ is exactly $\tilde L_1$ {\em with
  appropriate $\chi_1$}. Indeed,
this component arises from $g_\scl(\lambda \pa_x)$ (as opposed to
$g_\scl(\omega\pa_y)$, cf.\ \eqref{eq:L0-forms}) for a warped product scattering metric
$g_\scl$, which is $\lambda x^{-4}\,dx=x^{-2}(\lambda x^{-2}\,dx)$ (as
opposed to $x^{-2}h(\omega\pa_y)=x^{-1}(x^{-1}h(\omega\pa_y)$), with
the parenthesized factor being a smooth scattering one-form; the
trivialization factors out $x^{-2}\,dx$. Thus, recalling
\eqref{eq:L-forms}, the normal component of $Lv$ is
\[
\int
\chi(\lambda/x)v(\gamma_{x,y,\lambda,\omega})x^{-2}\lambda\,d\lambda\,d\omega=x^{-1}\int
\chi_1(\lambda/x) v(\gamma_{x,y,\lambda,\omega})\,d\lambda\,d\omega,
\]
this is exactly $\tilde L_1$
with $\chi_1(s)=s\chi(s)$.
{\em In this paper, from now on, we shall work with $L_1$ only, and
  not with $\tilde L_1$.} We
also write
$$
N_{1,\digamma}=e^{-\digamma/x}L_1
I e^{\digamma/x},
$$
acting as a map from tangential one forms to scalars.

\subsection{Microlocal projections}
Before we proceed with our computations, it is useful to have a decomposition when
one has an orthogonal projection at the principal symbol level, such
as $\Pi^\perp$ and $\Pi^\parallel$.

\begin{prop}\label{prop:microlocal-bundles}
Suppose that over an open subset $U$ of $\pa\overline{\Tsc^*}X$, a symbol $\Pi$ of order $(0,0)$ is orthogonal projection to a
subbundle of the pullback of a vector bundle $E$, with a Hermitian
inner product, over $X$ to $\Tsc^*X$ by the bundle
projection map, so $\Pi^2=\Pi$ and $\Pi^*=\Pi$. Then for any
$U_1\subset\overline{U_1}\subset U$, there exists
$P\in\Psisc^{0,0}(X)$ such that microlocally on $U_1$, the principal
symbol of $P$ is $\Pi$, and furthermore $P^2=P$, $P^*=P$ microlocally,
i.e.\ $\WFsc'(P^2-P)\cap U_1=\emptyset$, $\WFsc'(P-P^*)\cap U_1=\emptyset$.
\end{prop}

\begin{proof}
This is a standard iterative construction, which is completely
microlocal. We first write down the argument with
$U_1=U=\pa\overline{\Tsc^*}X$, i.e.\ globally, and then simply remark
on its microlocal nature.

One starts by taking any
operator $P_0\in\Psisc^{0,0}$ with principal symbol $\Pi$; one can replace $P_0$ by
$\frac{1}{2}(P_0+P_0^*)$ and thus assume that it is self-adjoint. Now
let $E_1=P_0^2-P_0\in\Psisc^{-1,-1}$ be the error of $P_0$ in being a
projection (note that the principal symbol of $P_0^2-P_0$ in
$\Psisc^{0,0}$ is $\Pi^2-\Pi=0$, hence its membership in $\Psisc^{-1,-1}$). Note
that $P_0E_1=P_0^3-P_0^2=E_1P_0$, so if $e_1$ is the principal symbol
of $E_1$, then $\Pi e_1=e_1\Pi$. Now we want to correct $P_0$ by
adding $P_1\in\Psisc^{-1,-1}$ so that $P_1^*=P_1$ and
$(P_0+P_1)^2-(P_0+P_1)\in\Psisc^{-2,-2}$ has lower order than
$E_1=P_0^2-P_0$; note that $E_1^*=E_1$. We compute this:
$$
(P_0+P_1)^2-(P_0+P_1)=P_0^2-P_0+P_0P_1+P_1P_0-P_1+P_1^2=E_1+P_0P_1+P_1P_0-P_1+F_2,
$$
where $F_2\in\Psisc^{-2,-2}$, so irrelevant for our conclusion on the
improved projection property. Hence,
the membership of $(P_0+P_1)^2-(P_0+P_1)$ in $\Psisc^{-2,-2}$ is
equivalent to
the principal symbol $p_1$ of $P_1$ satisfying $e_1+\Pi
p_1+p_1\Pi-p_1=0$. So let
$$
p_1=-\Pi e_1\Pi+(1-\Pi)e_1(1-\Pi);
$$
notice that $p_1^*=p_1$ since $e_1^*=e_1$ (being the principal symbol
of a symmetric operator).
Then, as $\Pi^2=\Pi$, $\Pi(1-\Pi)=0$,
\begin{equation*}\begin{aligned}
e_1+\Pi
p_1+p_1\Pi-p_1&=e_1-\Pi e_1\Pi-\Pi e_1\Pi+\Pi
e_1\Pi-(1-\Pi)e_1(1-\Pi)\\
&=e_1-\Pi e_1\Pi-(1-\Pi)e_1(1-\Pi)=0
\end{aligned}\end{equation*}
since $e_1=\Pi e_1\Pi+\Pi
e_1(1-\Pi)+(1-\Pi)e_1\Pi+(1-\Pi)e_1(1-\Pi)=\Pi
e_1\Pi+(1-\Pi)e_1(1-\Pi)$ as $e_1$ commutes with $\Pi$, so $\Pi
e_1(1-\Pi)=0$, etc. Thus, $e_1+\Pi
p_1+p_1\Pi-p_1=0$ holds. Taking any $P_1$ with principal symbol $p_1$,
replace $P_1$ by $\frac{1}{2}(P_1+P_1^*)$ so one has self-adjointness
as well (and still the same principal symbol), we have the desired
property $(P_0+P_1)^2-(P_0+P_1)\in\Psisc^{-2,-2}$.

The general inductive procedure is completely similar; in step $j+1$,
$j\geq 0$ (so $j=0$ above), if
$(P^{(j)})^2-P^{(j)}=E_{j+1}\in\Psisc^{-j-1,-j-1}$ and
$(P^{(j)})^*=P^{(j)}$, one finds $P_{j+1}\in\Psisc^{-j-1,-j-1}$
such that $P_{j+1}^*=P_{j+1}$, which one can easily arrange at the
end, and such that
$(P^{(j)}+P_{j+1})^2-(P^{(j)}+P_{j+1}\in\Psisc^{-j-2,-j-2}$; for this one
needs (with analogous notation to above) $e_{j+1}+\Pi p_{j+1}+\Pi
p_{j+1}-p_{j+1}=0$, which is satisfied with $p_{j+1}=-\Pi
e_{j+1}\Pi+(1-\Pi)e_{j+1}(1-\Pi)$ by completely analogous arguments as
above.

An asymptotic summation of $\sum_{j=0}^\infty P_j$ gives the desired
operator $P$ in the global case.

In the local case, when $U$ is a proper subset of
$\pa\overline{\Tsc^*X}$, one simply notes that all the algebraic steps
are microlocal (i.e.\ local in $\pa\overline{\Tsc^*X}$ modulo
$\Psisc^{-\infty,-\infty}$) including the composition of microlocally
defined operators. One thus obtains a sequence of microlocal operators
$P_j$ defined on $U$; taking any $Q\in\Psisc^{0,0}$ with
$\WFsc'(Q)\subset U$, $\WFsc'(\Id-Q)\cap \overline{U_1}=\emptyset$,
one then asymptotically sums $\sum_{j=0}^\infty QP_j$ (with each term
making sense modulo $\Psisc^{-\infty,-\infty}$) to obtain the
globally defined $P$
with the desired properties.
\end{proof}

\begin{rem}\label{rem:microlocal-bundles}
Proposition~\ref{prop:microlocal-bundles} means that if one has
orthogonal projections $\Pi^\perp$ and $\Id-\Pi^\perp$ to orthogonal
subspaces of, say, $\Tsc^*X$, microlocally on $U$, then one can take
$P^\perp$ as guaranteed by the proposition, so $P^\perp$,
$\Id-P^\perp$ are microlocal orthogonal projections, write
$u=u_\perp+u_\parallel$ with $u_\perp=P^\perp v$,
$u_\parallel=(\Id-P^\perp) w$ microlocally on $U_1$ (i.e.\
$\WFsc(u_\perp-P^\perp v)\cap U_1=\emptyset$, etc.), and $u_\perp$,
$u_\parallel$ are microlocally uniquely determined, i.e.\ any other
$u'_\perp$, $u'_\parallel$ satisfy $\WFsc(u'_\perp-u_\perp)\cap
U_1=\emptyset$, etc. Indeed, for such $u_\parallel$, $P^\perp
u_\parallel$ has $\WFsc$ disjoint from $U_1$, so $P^\perp u=P^\perp
u^\perp=(P^\perp)^2 v=P^\perp v=u_\perp$ microlocally on $U_1$, and
similarly for $u_\parallel$. Since operators with wave front sets
disjoint from the region we are working on are irrelevant for our
considerations, we may legitimately write one forms as
$$
\begin{pmatrix}u_0\\u_1\end{pmatrix},
$$
where $u_0$ is microlocally in $\Ran P^\perp$, $u_1$ in $\Ran
(\Id-P^\perp)$: $u_0=P^\perp u$, $u_1=(\Id-P^\perp)u$.
\end{rem}

\subsection{The principal symbol in the one form setting}
In order to do the computation of the principal symbol of $L_j I$ in $x>0$ in a
smooth (thus uniform) manner down to $x=0$, in a way that also
describes the boundary principal symbol {\em near} fiber infinity (the
previous computations were at fiber infinity only!), it is convenient to
utilize a direct oscillatory integral representation of $L_j I$,
$j=0,1$. With a slight abuse of notation we write
$$
N_\digamma=\begin{pmatrix}N_{0,\digamma}\\N_{1,\digamma}\end{pmatrix};
$$
this is indeed the previous $N_\digamma$ with domain restricted to
tangential one-forms is and with target space decomposed according to
the normal-tangential decomposition of one-forms.

Our initial goal in this section is to prove:

\begin{prop}\label{prop:N-digamma-1-form-structure-basic}
Let $\xi_\digamma=\xi+i\digamma$.
The full symbol of the operator
$$
N_\digamma=\begin{pmatrix}N_{0,\digamma}\\N_{1,\digamma}\end{pmatrix},
$$
with domain restricted to
tangential one-forms is, relative to the $\Span\{\eta\}$-based decomposition of the domain,
$$
\begin{pmatrix}
a_{00}^{(0)}&a_{01}^{(1)}\xi_\digamma+a_{01}^{(0)}\\
a_{10}^{(0)}&a_{11}^{(1)}\xi_\digamma+a_{11}^{(0)}\\
\end{pmatrix},
$$
where $a_{ij}^{(k)}\in S^{-1-j,0}$ for all $i,j,k$.

Furthermore, $a_{ij}^{(k)}\in S^{-1-j,0}$ depend continuously on the
metric $g$ (with the $\CI$ topology on $g$) as long as $g$ is
$C^k$-close (for suitable $k$) to a background metric $g_0$ satisfying
the strictly convex assumptions on
the metric, the boundary and the function $x$.
\end{prop}

\begin{rem}
The statement of this proposition would be equally valid with
$\xi_\digamma$ replaced by $\xi$, since one can absorb the difference
into the lower order, in terms of $\xi$-power, terms. The reason we
phrase it this way is that in Proposition~\ref{prop:N-digamma-1-form-structure} this will no longer be
the case due to the order of e.g.\ $a^{(0)}_{01}$ there, with the
decay order being the issue.
\end{rem}

\begin{proof}
We in fact do the complete form computation from scratch, initially using a
general localizer $\tilde\chi$ (potentially explicitly dependent on $x,y,\omega$
as well, with compact support in $\lambda/x$), not just the kind considered above.  Note that we already know
that we have a pseudodifferential operator
$\Ajd=e^{-\digamma/x}L_j Ie^{\digamma/x}\in\Psisc^{-1,0}$, where we
{\em do not restrict} $I$ to tangential forms, and
with $\Ajd$ the component mapping to tangential ($j=0$) or normal
($j=1$) one forms
given by
\begin{equation*}\begin{aligned}
\Ajd f(z)=\int e^{-\digamma/x(z)}
&e^{\digamma/x(\gamma_{z,\lambda,\omega}(t))} x^{-j}\lambda^j
(h(y)\omega)^{\otimes(1-j)}\\
&\tilde\chi(z,\lambda/x,\omega)
f(\gamma_{z,\lambda,\omega}(t))(\dot\gamma_{z,\lambda,\omega}(t))\,dt\,|d\nu|. 
\end{aligned}\end{equation*}
{\em Here $\Ajd$ is understood to apply only to $f$ with
  support in $M$, thus for which the $t$-integral is in a fixed finite
  interval}, where $h(y)\omega$ is the image of $\omega$ under the
metric $h=h(y)$ induced on the level sets of $x$ by $g_{\scl}$ and
where $|d\nu|$ is a smooth positive density in $(\lambda,\omega)$, such
as $|d\lambda\,d\omega|$.
Then $\Ajd$ will be the left quantization of the symbol $a_{j,\digamma}$ where $a_{j,\digamma}$ is
the inverse Fourier transform in $z'$ of the integral. If $K_{\Ajd}$ is the
Schwartz kernel, then in the sense of oscillatory integrals (or
directly if the order of $a_{j,\digamma}$ is sufficiently low)
$$
K_{\Ajd}(z,z')=(2\pi)^{-n}\int e^{i(z-z')\cdot\zeta}a_{j,\digamma}(z,\zeta)\,d\zeta,
$$
i.e.\ $(2\pi)^{-n}$ times the Fourier transform in $\zeta$ of
$(z,\zeta)\mapsto e^{iz\cdot\zeta} a_{j,\digamma}(z,\zeta)$, so taking the inverse
Fourier transform in $z'$ yields
$(2\pi)^{-n}a_{j,\digamma}(z,\zeta)e^{iz\cdot\zeta}$, i.e.
\begin{equation} \label{eq:aj-in-terms-of-kernel}
a_{j,\digamma}(z,\zeta)=(2\pi)^{n}e^{-iz\cdot\zeta}\cF^{-1}_{z'\to\zeta}K_{\Ajd}(z,z').
\end{equation}
Now,
\begin{equation*}\begin{aligned}
K_{\Ajd}(z,z')&=\int e^{-\digamma/x(z)} e^{\digamma/x(\gamma_{z,\lambda,\omega}(t))}x^{-j}\lambda^j
(h(y)\omega)^{\otimes(1-j)}\tilde\chi(z,\lambda/x,\omega)\\
&\qquad\qquad\dot\gamma_{z,\lambda,\omega}(t)\delta(z'-\gamma_{z,\lambda,\omega}(t))\,dt\,|d\nu|\\
&=
(2\pi)^{-n}
\int e^{-\digamma/x(z)} e^{\digamma/x(\gamma_{z,\lambda,\omega}(t))}x^{-j}\lambda^j
(h(y)\omega)^{\otimes(1-j)}\tilde\chi(z,\lambda/x,\omega)\\
&\qquad\qquad\dot\gamma_{z,\lambda,\omega}(t)e^{-i\zeta'\cdot(z'-\gamma_{z,\lambda,\omega}(t))}\,dt\,|d\nu|\,|d\zeta'|;
\end{aligned}\end{equation*}
as remarked above, the $t$ integral is actually over a fixed finite
interval, say $|t|<T$, or one may explicitly insert a compactly supported cutoff
in $t$ instead. (So the only non-compact domain of integration is in
$\zeta'$, corresponding to the Fourier transform.)
Thus, taking the inverse Fourier transform in $z'$ and evaluating at $\zeta$
gives
\begin{equation*}\begin{aligned}
a_{j,\digamma}(z,\zeta)=
\int e^{-\digamma/x(z)} &e^{\digamma/x(\gamma_{z,\lambda,\omega}(t))}x^{-j}\lambda^j
(h(y)\omega)^{\otimes(1-j)}\tilde\chi(z,\lambda/x,\omega)\\
&\qquad \dot\gamma_{z,\lambda,\omega}(t) e^{-iz\cdot\zeta} e^{i\zeta\cdot\gamma_{z,\lambda,\omega}(t)}\,dt\,|d\nu|.
\end{aligned}\end{equation*}
Translating into sc-coordinates, writing $(x,y)$ as local coordinates,
scattering covectors as $\xi\frac{dx}{x^2}+\eta\cdot\frac{dy}{x}$, and
$\gamma=(\gamma^{(1)},\gamma^{(2)})$, with $\gamma^{(1)}$ the $x$
component, $\gamma^{(2)}$ the $y$ component, we obtain
\begin{equation}\begin{aligned}\label{eq:1-form-kernel-sc-form}
&a_{j,\digamma}(x,y,\xi,\eta)\\
&=
\int e^{-\digamma/x} e^{\digamma/\gamma^{(1)}_{x,y,\lambda,\omega}(t)}x^{-j}\lambda^j
(h(y)\omega)^{\otimes(1-j)}\tilde\chi(x,y,\lambda/x,\omega)
\dot\gamma_{x,y,\lambda,\omega}(t)\\
&\qquad\qquad\qquad\qquad\qquad e^{i(\xi/x^2,\eta/x)\cdot(\gamma^{(1)}_{x,y,\lambda,\omega}(t)-x, \gamma^{(2)}_{x,y,\lambda,\omega}(t)-y)}\,dt\,|d\nu|
\end{aligned}
\end{equation}
and
$$
\gamma_{x,y,\lambda,\omega}(t)=(x+\lambda t+\alpha
t^2+t^3\Gamma^{(1)}(x,y,\lambda,\omega,t),y+\omega t+t^2\Gamma^{(2)}(x,y,\lambda,\omega,t)). 
$$
As a scattering tangent vector, i.e.\ expressed in terms of
$x^2\pa_x$ and $x\pa_y$, so as to act on sections of
$\Tsc^*X$, recalling that the $x$ coordinate of the point
we are working at is $\gamma^{(1)}_{x,y,\lambda,\omega}(t)$,
$$
\dot\gamma_{x,y,\lambda,\omega}(t)=\gamma^{(1)}_{x,y,\lambda,\omega}(t)^{-1}(\gamma^{(1)}_{x,y,\lambda,\omega}(t)^{-1} (\lambda+2\alpha t+t^2\tilde\Gamma^{(1)}(x,y,\lambda,\omega,t)),\omega+t\tilde\Gamma^{(2)}(x,y,\lambda,\omega,t)),
$$
with $\Gamma^{(1)},\Gamma^{(2)},\tilde\Gamma^{(1)},\tilde\Gamma^{(2)}$
smooth functions of $x,y,\lambda,\omega,t$.
We recall from \cite{UV:local} that we need to work in a sufficiently
small region so that there are no geometric complications. Thus the interval of integration in $t$, i.e., $T$, is such that (with
the dot denoting $t$-derivatives)
$\ddot\gamma^{(1)}(t)$ is uniformly bounded below by a positive constant in the
region over which we integrate, see
the discussion in \cite{UV:local} above Equation~(3.1). Then $T$ is 
further reduced in Equations~(3.3)-(3.4) so that the map sending $(x,y,\lambda,\omega,t)$ to
the lift of $(x,y,\gamma_{x,y,\lambda,\omega}(t))$ in the resolved
space $X^2$ with the diagonal being blown up, is a diffeomorphism in $t\geq
0$, as well as $t\leq 0$. In the present paper the restriction to small $T$ will occur
in a closely related manner, when dealing with the stationary phase expansion.

We change the variables of integration to $\hat t=t/x$, and
$\hat\lambda=\lambda/x$, so the $\hat\lambda$ integral is in fact over
a fixed compact interval, but the $\hat t$ one is over $|\hat t|<T/x$
which grows as $x\to 0$.
We get that the
phase is
$$
\xi(\hat\lambda\hat t+\alpha\hat t^2+x\hat
t^3\Gamma^{(1)}(x,y,x\hat\lambda,\omega,x\hat t))+\eta\cdot(\omega\hat t+x\hat t^2\Gamma^{(2)}(x,y,x\hat\lambda,\omega,x\hat t)),
$$
while the exponential damping factor (which we regard as a Schwartz
function, part of the amplitude, when one regards $\hat t$ as a
variable on $\RR$) is
\begin{equation*}\begin{aligned}
&-\digamma/x+\digamma/\gamma^{(1)}_{x,y,\lambda,\omega}(t)\\
&=
-\digamma(\lambda t+\alpha
t^2+t^3\Gamma^{(1)}(x,y,\lambda,\omega,t))x^{-1}(x+\lambda t+\alpha
t^2+t^3\Gamma^{(1)}(x,y,\lambda,\omega,t))^{-1}\\
&=-\digamma(\hat\lambda\hat t+\alpha\hat
t^2+\hat t^3 x\hat\Gamma^{(1)}(x,y,x\hat\lambda,\omega,x\hat t)),
\end{aligned}\end{equation*}
with $\hat\Gamma^{(1)}$ a smooth function.
The only subtlety in applying the stationary phase lemma is that the
domain of integration in $\hat t$ is not compact, so we need to
explicitly deal with the region $|\hat t|\geq 1$, say, assuming that the
amplitude is Schwartz in $\hat t$, uniformly in the other variables. Notice that as long
as the first derivatives of the phase in the integration variables
have a lower bound $c |(\xi,\eta)|\,|\hat t|^{-k}$ for some $k$, and for some $c>0$, the standard
integration by parts argument gives the rapid decay of the integral in
the large parameter $|(\xi,\eta)|$. At $x=0$ the phase
is
$\xi(\hat\lambda\hat t+\alpha\hat t^2)+\hat t\eta\cdot\omega$;
if $|\hat
t|\geq 1$, say, the $\hat\lambda$ derivative is $\xi\hat t$, which is
thus bounded below by $|\xi|$ in magnitude. The only place where
one may not have rapid decay is at $\xi=0$ (meaning, in the spherical
variables, $\frac{\xi}{|(\xi,\eta)|}=0$). In this region one may use
$|\eta|$ as the large variable to simplify the notation slightly. The
phase is then with $\hat\xi=\frac{\xi}{|\eta|}$, $\hat\eta=\frac{\eta}{|\eta|}$,
$$
|\eta|(\hat\xi(\hat\lambda\hat t+\alpha\hat t^2)+\hat t\hat\eta\cdot\omega),
$$
with parameter differentials (ignoring the overall $|\eta|$ factor)
$$
\hat\xi\hat t\,d\hat\lambda,(\hat t\hat\eta+\hat
t^2\hat\xi\pa_\omega\alpha)\cdot
\,d\omega,\; (\hat\xi(\hat\lambda+2\alpha\hat
t)+\hat\eta\cdot\omega)\,d\hat t.
$$
With $\hat\Xi=\hat\xi\hat t$ and $\rho=\hat t^{-1}$ these are
$$
\hat\Xi\,d\hat\lambda,\hat t(\hat\eta+\hat\Xi\pa_\omega\alpha)\cdot
\,d\omega,\; (\hat\Xi(\rho\hat\lambda+2\alpha)+\hat\eta\cdot\omega)\,d\hat t,
$$
and now for critical points $\hat\Xi$ must vanish (as we already knew
from above), then the last of these gives that $\hat\eta\cdot\omega$
vanishes, but then the second gives that there cannot be a critical
point (in $|\hat t|\geq 1$). While this argument was at $x=0$, the full
phase derivatives
are
\begin{equation*}\begin{aligned}
&(\hat\xi\hat t(1+x\hat t\pa_\lambda\alpha+x^2\hat
t^2\pa_\lambda\Gamma^{(1)})+\hat\eta\cdot x^2\hat
t^2\pa_\lambda\Gamma^{(2)})\,d\hat\lambda,\\
&(\hat t\hat\eta+x\hat t^2\hat\eta\cdot\pa_\omega\Gamma^{(2)}+\hat
t^2\hat\xi\pa_\omega\alpha+x\hat t^3\hat\xi\pa_\omega\Gamma^{(1)})\cdot
\,d\omega,\\
&(\hat\xi(\hat\lambda+2\alpha\hat
t+3x\hat t^2\Gamma^{(1)}+x^2\hat
t^3\pa_t\Gamma^{(1)}) +\hat\eta\cdot\omega+2x\hat t\Gamma^{(2)}+x^2\hat
t^2\pa_t\Gamma^{(2)})\,d\hat t, 
\end{aligned}\end{equation*}
i.e.
\begin{equation*}\begin{aligned}
&(\hat\Xi(1+t\pa_\lambda\alpha+t^2\pa_\lambda\Gamma^{(1)})+\hat\eta\cdot t^2\pa_\lambda\Gamma^{(2)})\,d\hat\lambda, \\
&\hat t(\hat\eta+\hat\eta\cdot t\pa_\omega\Gamma^{(2)}+\hat\Xi\pa_\omega\alpha+t\hat\Xi\pa_\omega\Gamma^{(1)})\cdot
\,d\omega,\\
&(\hat\Xi(\hat\lambda\rho+2\alpha+3t\Gamma^{(1)}+
t^2\pa_t\Gamma^{(1)})+\hat\eta\cdot\omega+2t\Gamma^{(2)}+t^2\pa_t\Gamma^{(2)})\,d\hat t,
\end{aligned}\end{equation*}
and now all the additional terms are small if $T$ is small (where
$|t|<T$), so the lack of critical points in the $x=0$ computation
implies the analogous statement (in $|\hat t|>1$) for the general computation.

This implies that one can use the standard {\em parameter-dependent} stationary phase lemma, see
e.g.\ \cite[Theorem~7.7.6]{Hor}. At $x=0$, the
stationary points of the phase are $\hat t=0$,
$\xi\hat\lambda+\eta\cdot\omega=0$, which remain critical points for
$x$ non-zero due to the $x\hat t^2$ vanishing of the other terms, and
when $T$ is small, so $x\hat t$ is small, there are no other critical
points. (One can see this in a different way: above we worked with
$|\hat t|\geq 1$, but for any $\ep>0$, $|\hat t|\geq\ep$ would have
worked equally.) These critical points lie on a smooth codimension 2
submanifold of the parameter space. At $x=0$,
$\xi=0$, in whose neighborhood we are focusing on, since this is where
$N_{0,\digamma}$ is not elliptic, this submanifold is given by the
vanishing of $(\hat t,\omega^\parallel)$, with
$\omega^\parallel=\omega\cdot\hat\eta$ the $\hat\eta$ component of
$\omega$. Moreover, the $(\hat t,\omega^\parallel)$-Hessian matrix
there is
$\begin{pmatrix}
  0&|\eta|\\|\eta|&0\end{pmatrix}$,
which is elliptic. We thus use the stationary phase lemma in the
$(\hat t,\omega^\parallel)$ variables.
This gives that all terms of the
form $\hat
t x$ times smooth functions will have contributions which are 1 differentiable and 1 decay
order lower than the main terms, while $\hat
t^3 x$-type terms will have contributions which are 2 differentiable and 1 decay
order lower than the main terms. For us in this section only the
principal terms matter, unlike in the 2-tensor case considered in the
next section, so any $O(x\hat t)$ terms are actually ignorable for our
purposes.
Moreover,
when evaluated on tangential forms (which is our
interest here, as we are analyzing $N_{j,\digamma}$),
$\dot\gamma_{x,y,\lambda,\omega}(t)$ can be replaced by
\begin{equation*}\begin{aligned}
\dot\gamma^{(2)}_{x,y,\lambda,\omega}&=\gamma^{(1)}_{x,y,\lambda,\omega}(t)^{-1} (\omega+\hat t
x\tilde\Gamma^{(2)}(x,y,x\hat\lambda,\omega,x\hat t))\\
&=x^{-1}(\omega+\hat t
x\hat\Gamma^{(2)}(x,y,x\hat\lambda,\omega,\hat t))
\end{aligned}\end{equation*}
with $\hat\Gamma^{(2)}$ smooth.

Notice that $\Njd P^\perp$, $\Njd P^\parallel$, with
$P^\perp$, resp.\ $P^\parallel$, the microlocal orthogonal
projection with principal symbol $\Pi^\perp$, resp.\ $\Pi^\parallel$,
cf.\ Proposition~\ref{prop:microlocal-bundles} and Remark~\ref{rem:microlocal-bundles},
will have principal symbol given by the composition of principal
symbols.
Thus, with $\tilde\chi=\chi(\lambda/x)=\chi(\hat\lambda)$,
we have that on
\begin{equation*}\begin{aligned}
&\Span\{\eta\}^{\perp}\ (k=0), \text{resp.}\ \Span\{\eta\}\ (k=1),
\end{aligned}\end{equation*}
writing the sections in $\Span\{\eta\}$ factors explicitly as a
multiple
of $\frac{\eta}{|\eta|}$,
\begin{equation}\begin{aligned}\label {eq:aj-1-form-pre-xi}
a_{j,\digamma}(x,y,\xi,\eta) &=\int e^{i(\xi x^{-2}(\gamma^{(1)}_{x,y,x\hat\lambda,\omega}(x\hat t)-x)+\eta x^{-1}(
  \gamma^{(2)}_{x,y,x\hat\lambda,\omega}(x\hat t)-y))} e^{-\digamma(\hat\lambda\hat t+\alpha\hat
t^2)} \\
&\qquad   \hat\lambda^j(h(y)\omega)^{\otimes(1-j)}\chi(\hat\lambda)|\eta|^{-k}
(x\dot\gamma^{(2)}_{x,y,x\hat\lambda,\omega}(x\hat t)\cdot\eta)^k(x\dot\gamma^{(2)}_{x,y,x\hat\lambda,\omega}(x\hat t)\cdot)^{\otimes(1-k)}
\,d\hat t\,d\hat \lambda\,d\omega\\
&=\int e^{i(\xi(\hat\lambda\hat t+\alpha\hat t^2+x\hat
t^3\Gamma^{(1)}(x,y,x\hat\lambda,\omega,x\hat t))+\eta\cdot(\omega\hat
t+x\hat t^2\Gamma^{(2)}(x,y,x\hat\lambda,\omega,x\hat t)))} e^{-\digamma(\hat\lambda\hat t+\alpha\hat
t^2)}\\ &\qquad
\hat\lambda^j(h(y)\omega)^{\otimes(1-j)}\chi(\hat\lambda)|\eta|^{-k}(
\omega\cdot\eta)^k
\big(\omega\cdot\big)^{\otimes(1-k)}\,d\hat t\,d\hat \lambda\,d\omega,
\end{aligned}\end{equation}
up to errors that are
$O(x\langle\xi,\eta\rangle^{-1})$ relative to
the a priori order, $(-1,0)$, arising from the $0$-th order symbol in
the oscillatory integral and the 2-dimensional space in which the
stationary phase lemma is applied.

Now we want to see, for $k=1$ (since the $k=0$ statement is trivial), that
$(x\dot\gamma^{(2)}_{x,y,x\hat\lambda,\omega}(x\hat t)\cdot\eta)^k$ ,
while an order $k$ symbol, in this oscillatory integral is actually equivalent to the sum of terms
over $\ell$, $0\leq\ell\leq k$, each of which is the product of
$\xi^\ell$ and an order $0$ symbol, essentially due to the structure
of the set of critical points of the phase. In order to avoid having
to specify the latter in $x>0$, we proceed with a direct integration
by parts argument.
Notice that
$$
(x\dot\gamma^{(2)}_{x,y,x\hat\lambda,\omega}(x\hat t)\cdot\eta)e^{i \eta x^{-1}(
  \gamma^{(2)}_{x,y,x\hat\lambda,\omega}(x\hat t)-y)}=x\pa_{\hat t}e^{i \eta x^{-1}(
  \gamma^{(2)}_{x,y,x\hat\lambda,\omega}(x\hat t)-y)},
$$
integration by
parts gives that \eqref{eq:aj-1-form-pre-xi} is, with $k=1$,
\begin{equation*}\begin{aligned}
a_{j,\digamma}(x,y,\xi,\eta)&=\int e^{i \eta x^{-1}(
  \gamma^{(2)}_{x,y,x\hat\lambda,\omega}(x\hat t)-y)}\\
&\qquad\qquad x^k\pa_{\hat t}^k\Big(e^{i(\xi x^{-2}(\gamma^{(1)}_{x,y,x\hat\lambda,\omega}(x\hat t)-x))}e^{-\digamma(\hat\lambda\hat t+\alpha\hat
t^2)}(x\dot\gamma^{(2)}_{x,y,x\hat\lambda,\omega}(x\hat t)\cdot)\Big)\\
&\qquad\qquad \hat\lambda^j(h(y)\omega)^{\otimes(1-j)}\chi(\hat\lambda)|\eta|^{-k}\,d\hat t\,d\hat \lambda\,d\omega.
\end{aligned}\end{equation*}
Expanding the derivative, if $\ell$ derivatives hit the first
exponential (the phase factor) and thus $k-\ell$ the second (the
amplitude) one obtains $\xi^\ell$ times the oscillatory factor
$e^{i(\xi x^{-2}(\gamma^{(1)}_{x,y,x\hat\lambda,\omega}(x\hat t)-x))}$
times a symbol of order $0$. Notice that
$$
x\pa_{\hat t}(x^{-2}(\gamma^{(1)}_{x,y,x\hat\lambda,\omega}(x\hat t)-x))=\hat\lambda+2\alpha \hat t+\hat
t^2 x\tilde\Gamma^{(1)}(x,y,x\hat\lambda,\omega,x\hat t), 
$$
so in view of the overall weight
$|\eta|^{-k}$, we deduce that, modulo terms one order down (so subprincipal), in terms
of the differential order, $a_{j,\digamma}$ is a sum of terms of the
form of symbols of order $(-k-1,0)$ times $\xi^\ell$, $0\leq\ell\leq
k$. Here the first order is $-k-1$ since stationary phase itself, in the
two variables, gives an extra factor of $|\eta|^{-1}$, corresponding
to the square root of the absolute value of the determinant of the Hessian.

We remark here that $\gamma$, and thus $N_{j,\digamma}$, depend
continuously on the metric $g$, and furthermore the same is true for
$a_j$ and the decomposition into components as in the statement of the proposition.
\end{proof}

Analyzing the proof of Proposition~\ref{prop:N-digamma-1-form-structure-basic} at $x=0$ more precisely, we have
\begin{equation*}\begin{aligned}
&a_{j,\digamma}(0,y,\xi,\eta)\\
&=
\int e^{i(\xi(\hat\lambda\hat t+\alpha\hat t^2)+\eta\cdot(\omega\hat t))}e^{-\digamma(\hat\lambda\hat t+\alpha\hat
t^2)}\hat\lambda^j(h(y)\omega)^{\otimes(1-j)}\chi(\hat\lambda)|\eta|^{-k} (\omega\cdot\eta)^k(\omega\cdot)^{\otimes(1-k)}\,d\hat t\,d\hat
\lambda\,d\omega\\
&=
\int e^{i((\xi+i\digamma)(\hat\lambda\hat t+\alpha\hat t^2)+\eta\cdot(\omega\hat t))}\hat\lambda^j(h(y)\omega)^{\otimes(1-j)}\chi(\hat\lambda) |\eta|^{-k}  (\omega\cdot\eta)^k(\omega\cdot)^{\otimes(1-k)}\,d\hat t\,d\hat
\lambda\,d\omega\\
&=
\int_{\sphere^{n-2}}|\eta|^{-k}  (\omega\cdot\eta)^k(h(y)\omega)^{\otimes(1-j)}(\omega\cdot)^{\otimes(1-k)}\Big(\int e^{i((\xi+i\digamma)(\hat\lambda\hat t+\alpha\hat t^2)+(\eta\cdot\omega)\hat t)}\hat\lambda^j\chi(\hat\lambda) \,d\hat t\,d\hat
\lambda\Big)\,d\omega.
\end{aligned}\end{equation*}
We recall that $\alpha=\alpha(x,y,\lambda,\omega)$ so at $x=0$,
$\alpha(0,y,0\cdot\hat\lambda,\omega)=\alpha(0,y,0,\omega)$, and it is a
quadratic form in $\omega$.

Some of the computations below become notationally simpler if we assume that
the coordinates are such that at $y$ at which the principal symbol is
computed $h$ is the Euclidean metric. {\em We thus assume this from
  now on; note that even the integration by parts arguments are
  unaffected, as $h$ would not be differentiated, since it is a
  prefactor of the integral used in the integration by parts.}

We now apply the projection $P^\perp$ (quantization of the projection to
$\Span\{\eta\}^\perp$ as in Proposition~\ref{prop:microlocal-bundles}) from the left: for the
tangential, resp.\ normal
components we apply $P^\perp$, resp.\
$\Id$, which means for the symbol computation that we compose with
$\Pi^\perp$, resp.\
$I$ from the left. This replaces $(h(y)\omega)^{\otimes{1-j}}=\omega^{\otimes(1-j)}$ by
$((h(y)\omega)^\perp)^{\otimes(1-j)}=(\omega^\perp)^{\otimes(1-j)}$ with the result
\begin{equation*}\begin{aligned}
\tilde a_{j,\digamma}(0,y,\xi,\eta)&=
\int_{\sphere^{n-2}}|\eta|^{-k}  (\omega\cdot\eta)^k (\omega^\perp)^{\otimes(1-j)} (\omega^\perp\cdot)^{\otimes(1-k)}\\
&\quad \times\Big(\int e^{i((\xi+i\digamma)(\hat\lambda\hat t+\alpha\hat t^2)+(\eta\cdot\omega)\hat t)}\hat\lambda^j\chi(\hat\lambda) \,d\hat t\,d\hat
\lambda\Big)\,d\omega,
\end{aligned}\end{equation*}
where we used that $(\omega\cdot)^{\otimes(1-k)}$ is being applied to
the $\eta$-orthogonal factors, so it may be written as
$(\omega^\perp\cdot)^{\otimes(1-k)}$. This means that at $\xi=0$ the overall
parity of the integrand in $\omega^\perp$ is $(-1)^{j+k}$ apart from
the appearance of $\omega^\perp$ in the exponent (via $\alpha$) of
$e^{-\digamma(\hat\lambda\hat t+\alpha\hat t^2)}$, which due to the
$\hat t^2$ prefactor of $\alpha$, giving quadratic vanishing at the
critical set, only contributes one order lower
terms, so modulo these the integral vanishes when $j$ and $k$ have the
opposite parity. This proves that $N_\digamma$,
when composed with the projections as described, has the following
form:

\begin{prop}\label{prop:N-digamma-1-form-structure-proj}
Let $\xi_\digamma=\xi+i\digamma$.
The symbol of the operator 
$$
\begin{pmatrix}P^\perp N_{0,\digamma}\\N_{1,\digamma}\end{pmatrix},
$$
with domain restricted to
tangential 1-forms, relative to the $\Span\{\eta\}$-based decomposition of the
domain, at $x=0$
has the form
$$
\begin{pmatrix}
a_{00}^{(0)}&a_{01}^{(1)}\xi_\digamma+a_{01}^{(0)}\\
a_{10}^{(1)}\xi_\digamma+a_{10}^{(0)}&a_{11}^{(1)}\xi_\digamma+a_{11}^{(0)}\\
\end{pmatrix},
$$
where $a_{ij}^{(k)}\in S^{-1-\max(i,j),0}$ for all $i,j,k$. Moreover,
this restriction
depends continuously on $\chi$ in these spaces when $\chi$ is
considered as an element of the Schwartz space.
\end{prop}

We can compute the leading terms quite easily: for $j=k=0$ this is
\begin{equation*}\begin{aligned}
&\tilde a_{0,\digamma}(0,y,\xi,\eta)\\
&=
\int_{\sphere^{n-2}}\omega^\perp(\omega^\perp\cdot)\Big(\int e^{i((\xi+i\digamma)(\hat\lambda\hat t+\alpha\hat t^2)+(\eta\cdot\omega)\hat t)}\chi(\hat\lambda) \,d\hat t\,d\hat
\lambda\Big)\,d\omega\\
&=
\int_{\sphere^{n-2}}\omega^\perp(\omega^\perp\cdot)\Big(\int e^{i((\xi\hat\lambda\hat t+\alpha\hat t^2)+(\eta\cdot\omega)\hat t)}e^{-\digamma(\hat\lambda\hat t+\alpha\hat t^2)}\chi(\hat\lambda) \,d\hat t\,d\hat
\lambda\Big)\,d\omega
\end{aligned}\end{equation*}
which at the critical points of the phase, $\hat t=0$,
$\xi\hat\lambda+\eta\cdot\omega=0$, where $\omega^\perp$ and
$\hat\lambda$ give variables along the critical set, gives, up to an
overall elliptic factor,
$$
\int_{\sphere^{n-3}}\omega^\perp(\omega^\perp\cdot)\Big(\int\chi(\hat\lambda) \,d\hat
\lambda\Big)\,d\omega^\perp,
$$
which is elliptic for $\chi\geq 0$ with $\chi(0)>0$. (Note here that
when $n=3$, the integral over $\sphere^{n-3}$ is a sum over two points.)

On the other hand, for $j=k=1$, 
\begin{equation*}\begin{aligned}
&\tilde a_{1,\digamma}(0,y,\xi,\eta)\\
&=
\int_{\sphere^{n-2}}|\eta|^{-1}  (\omega\cdot\eta)\Big(\int e^{i((\xi+i\digamma)(\hat\lambda\hat t+\alpha\hat t^2)+(\eta\cdot\omega)\hat t)}\hat\lambda\chi(\hat\lambda) \,d\hat t\,d\hat
\lambda\Big)\,d\omega,
\end{aligned}\end{equation*}
Writing $i(\omega\cdot\eta)e^{i (\eta\cdot\omega)\hat t}=\pa_{\hat
  t}e^{i (\eta\cdot\omega)\hat t}$ and integrating by parts yields
\begin{equation}\begin{aligned}\label{eq:1-form-11-entry-comp-1}
&\tilde a_{1,\digamma}(0,y,\xi,\eta)\\
&=
i\int_{\sphere^{n-2}}|\eta|^{-1}  \Big(\int e^{i((\xi+i\digamma)(\hat\lambda\hat t+\alpha\hat t^2)+(\eta\cdot\omega)\hat t)}(\xi+i\digamma)(\hat\lambda+2\alpha\hat t)\hat\lambda\chi(\hat\lambda) \,d\hat t\,d\hat
\lambda\Big)\,d\omega\\
&=
i|\eta|^{-1} (\xi+i\digamma)
\int_{\sphere^{n-2}} \Big(\int e^{i((\xi+i\digamma)(\hat\lambda\hat t+\alpha\hat t^2)+(\eta\cdot\omega)\hat t)}(\hat\lambda+2\alpha\hat t)\hat\lambda\chi(\hat\lambda) \,d\hat t\,d\hat
\lambda\Big)\,d\omega. 
\end{aligned}\end{equation}
Now the integral (the factor after $|\eta|^{-1}(\xi+i\digamma) $) at the critical points of the phase $\hat t=0$,
$\xi\hat\lambda+\eta\cdot\omega=0$, gives, up to an
overall elliptic factor,
$$
\int_{\sphere^{n-3}}\omega^\perp (\omega^\perp\cdot)\Big(\int\hat\lambda^2\chi(\hat\lambda) \,d\hat
\lambda\Big)\,d\omega^\perp,
$$
modulo $S^{-2,0}$,
i.e.\ for the same reasons as in the $j=k=0$ case above, when
$\chi\geq 0$, $\chi(0)>0$, \eqref{eq:1-form-11-entry-comp-1} is an elliptic
multiple of $|\eta|^{-1}(\xi+i\digamma)$!

Finally, when $j=0$, $k=1$, we have
\begin{equation*}\begin{aligned}
&\tilde a_{0,\digamma}(0,y,\xi,\eta)\\
&=
\int_{\sphere^{n-2}}|\eta|^{-1}  \omega^\perp (\omega\cdot\eta)\Big(\int e^{i((\xi+i\digamma)(\hat\lambda\hat t+\alpha\hat t^2)+(\eta\cdot\omega)\hat t)}\chi(\hat\lambda) \,d\hat t\,d\hat
\lambda\Big)\,d\omega,
\end{aligned}\end{equation*}
which, using  $i(\omega\cdot\eta)e^{i (\eta\cdot\omega)\hat t}=\pa_{\hat
  t}e^{i (\eta\cdot\omega)\hat t}$ as above, gives
\[\begin{aligned} 
&\tilde a_{0,\digamma}(0,y,\xi,\eta)\\
&=
i|\eta|^{-1} (\xi+i\digamma)
\int_{\sphere^{n-2}} \Big(\int e^{i((\xi+i\digamma)(\hat\lambda\hat t+\alpha\hat t^2)+(\eta\cdot\omega)\hat t)}(\hat\lambda+2\alpha\hat t)\chi(\hat\lambda) \,d\hat t\,d\hat
\lambda\Big)\,d\omega. 
\end{aligned}
\]
Now the leading term of the integral, due to the contributions from the critical points,
is (up to an overall elliptic factor)
$$
\int_{\sphere^{n-3}}\omega^\perp \Big(\int\hat\lambda\chi(\hat\lambda) \,d\hat
\lambda\Big)\,d\omega^\perp,
$$
modulo $S^{-2,0}$,
which vanishes for $\chi$ even, so for such $\chi$, the $(0,1)$ entry
has principal symbol which at $x=0$ is a multiple of $\xi_\digamma$,
and the multiplier is in $S^{-3,0}$ (one order lower than the
previous results).

In summary, we have the following result:

\begin{prop}\label{prop:N-digamma-1-form-structure}
Suppose $\chi\geq 0$, $\chi(0)>0$, $\chi$ even.
Let $\xi_\digamma=\xi+i\digamma$.
The full symbol of the operator 
$$
\begin{pmatrix}P^\perp N_{0,\digamma}\\N_{1,\digamma}\end{pmatrix},
$$
with domain restricted to
tangential 1-forms, relative to the $\Span\{\eta\}$-based decomposition of the
domain, at $x=0$
has the form
$$
\begin{pmatrix}
a_{00}^{(0)}&a_{01}^{(1)}\xi_\digamma+a_{01}^{(0)}\\
a_{10}^{(1)}\xi_\digamma+a_{10}^{(0)}&a_{11}^{(1)}\xi_\digamma+a_{11}^{(0)}\\
\end{pmatrix},
$$
where $a_{ij}^{(k)}\in S^{-1-\max(i,j),0}$ for all $i,j,k$, and $a_{00}^{(0)}$ and
$a_{11}^{(1)}$ (these are the multipliers of the
leading terms along the diagonal) are elliptic in $S^{-1,0}$ and
$S^{-2,0}$, respectively and $a_{01}^{(0)},a_{11}^{(0)}\in S^{-2,-1}$, i.e.\ in
addition to the statements in the previous propositions vanish at
$x=0$ and $a_{01}^{(1)}$ also has one lower differential order at $x=0$: $a_{01}^{(1)}\in S^{-3,0}+S^{-2,-1}$.
\end{prop}

\begin{cor}\label{cor:1-form-pre-post-mult}
By pre- and postmultiplying
$$
\begin{pmatrix} P^\perp N_{0,\digamma}\\N_{1,\digamma}\end{pmatrix}
$$
by elliptic operators in $\Psisc^{0,0}$,
one can arrange that the full principal symbol of the resulting
operator is of the form
$$
\begin{pmatrix} T&0\\0&\tilde a(\xi+i\digamma)+\tilde b\end{pmatrix},
$$
with $T=a_{00}^{(0)}$, resp.\ $\tilde a$ elliptic in $S^{-1,0}$, resp.\ $S^{-2,0}$,
near $\xi=0$ at fiber infinity, and $\tilde b\in S^{-2,-1}$.

Furthermore, $\tilde a,\tilde b,T$ in the indicated spaces depend continuously on the
metric $g$ (with the $\CI$ topology on $g$) as long as $g$ is
$C^k$-close (for suitable $k$) to a background metric $g_0$ satisfying
the strictly convex assumptions on
the metric, the boundary and the function $x$.
\end{cor}

\begin{proof}
Let $T=a_{00}^{(0)}$. By multiplying from the left by the elliptic symbol
$$
\begin{pmatrix} 1&0\\-(a_{10}^{(1)}\xi_\digamma+a_{10}^{(0)}) T^{-1}&1\end{pmatrix}
$$
we obtain
$$
\begin{pmatrix} T&a_{01}^{(1)}\xi_\digamma+a_{01}^{(0)} \\0&\tilde
  a_{11}^{(1)}\xi_\digamma+\tilde a_{11}^{(0)}\end{pmatrix},\qquad \tilde
a_{11}^{(k)}=a_{11}^{(k)}- (a_{10}^{(1)}\xi_\digamma+a_{10}^{(0)})T^{-1}a_{01}^{(k)},
$$
so  $\tilde a_{11}$ has the same properties as $a_{11}$ for $\xi$ near $0$ (the
case of interest), in particular the ellipticity of $\tilde
a_{11}^{(1)}$ follows from the one differential
order lower behavior (at $x=0$) than a priori expected for $a_{01}^{(1)}$, stated
in Proposition~\ref{prop:N-digamma-1-form-structure}, while the
vanishing of $\tilde a_{11}^{(0)}$ from that of $a_{01}^{(0)}$ (at $x=0$). Multiplying from the right by
$$
\begin{pmatrix} 1&-T^{-1}(a_{01}^{(1)}\xi_\digamma+a_{01}^{(0)}) \\0&1\end{pmatrix}
$$
we obtain
$$
\begin{pmatrix} T&0\\0&\tilde a_{11}^{(1)}(\xi+i\digamma)+\tilde a_{11}^{(0)}\end{pmatrix},
$$
as desired.
\end{proof}

\subsection{Analysis at radial points}
We now have a principally diagonal real principal type system, and thus in $x>0$
the standard propagation of singularities results applies. The
boundary behavior is also not hard to see due to the leading order
decoupling: one has {\em radial points} in the second (index $1$)
component. We recall here that radial points for an operator with real
scalar principal symbol are points at which the
Hamilton vector field of the (homogeneous with respect to dilations) principal symbol is tangent to the dilation orbits of the
cotangent bundle. This means that H\"ormander's propagation of
singularities theorem is vacuous there, since the bicharacteristic
through such a point is exactly the dilation orbit. In the
compactified perspective, in which the fibers of the (here:
scattering) cotangent bundle are compactified, so the `standard'
(differential regularity) microlocal analysis takes place on the
boundary of the fibers (which in turn can be identified with the
cosphere bundle), the (rescaled) Hamilton
vector field vanishes at such radial points.

In general, when the principal symbol is real, for such radial points there is a threshold regularity below
which one can propagate estimates towards the radial points and above
which one can propagate estimates away from the radial points. In our
case the standard principal symbol (at fiber infinity) is real, but
the principal symbol at $\pa X$, while real {\em at} fiber infinity, is not
so at finite points, thus {\em near} fiber infinity. In such a
situation even the weight does not help, and the imaginary part of the
principal symbol (which of course is only non-zero at $\pa X$) must
have the correct sign. Fortunately, this is the case for us since, as 
we have seen, the
principal symbol of the second component, in both senses, is an elliptic multiple of
$\xi+i\digamma$, $\digamma>0$. To illustrate why this is the correct
sign, note that $\xi+i\digamma$ is the principal symbol of $x^2D_x+i\digamma$,
whose nullspace contains functions like $e^{-\digamma/x}a(y)$, which
are exponentially decaying as $\digamma>0$, and indeed these give the
asymptotic behavior of solutions of the {\em inhomogeneous equation} as $x\to 0$, i.e.\ one can expect
Fredholm properties in the polynomially weighted Sobolev
spaces. A different connection one can make is with the standard propagation of
singularities: when the imaginary part of the principal symbol is
non-negative, one can still propagate estimates in the backward
direction along the bicharacteristics; in this case the usual
principal symbol in $x>0$, where this applies, is real, but this fact
illustrates the consistency of our present result (propagating to
$x=0$ is backward propagation and $\digamma>0$) with other phenomena. That
the Fredholm statement holds for operators of this type follows from
the following proposition, which we state for bundle valued
pseudodifferential operators for use in the 2-tensor setting:

\begin{prop}\label{prop:not-quite-real-princ-scalar}
Suppose for two vector bundles $\tilde E,\tilde F$,
$P\in\Psisc^{1,0}(X;\tilde E,\tilde F)$ has principal symbol
$(\xi+i\digamma)\tilde p$ in $x<\ep_0$, $\ep_0>0$,
where $\tilde p$ is elliptic in $S^{0,0}(\Tsc^*X;\Hom(\tilde E,\tilde F))$ and $\digamma>0$.

Suppose first that $B_j\in\Psisc^{0,0}$, $\WFsc'(B_1)$ contained near the radial
set $L$ of $\xi+i\digamma$ at fiber infinity at $x=0$, $B_2$ is
elliptic at $x=\ep_0/2$, $B_3$ is elliptic at fiber infinity for $x\in
[0,\ep_0/2]$ and on $\WFsc'(B_1)$. Then
for all $s,r,M,N$ we have estimates
$$
\|B_1 u\|_{s,r}\leq C(\|B_2 u\|_{s,r}+\|B_3 Pu\|_{s,r}+\|u\|_{-N,-M}),
$$
with $\|\cdot\|$ the norm in $\Hsc^{s,r}$, etc.

Suppose now instead that $B_j\in\Psisc^{0,0}$, $\WFsc'(B_1)$ contained
near $L$, $B_3$ is elliptic at fiber infinity for $x\in
[0,\ep_0/2]$ and on $\WFsc'(B_1)$. Then
for all $s,r,M,N$ we have estimates
$$
\|B_1 u\|_{s,r}\leq C(\|B_3 P^*u\|_{s,r}+\|u\|_{-N,-M}).
$$

In both cases, $u$ can be any distribution for which the right hand
side is finite, understood as $u\in\Hsc^{-N,-M}$, etc.
\end{prop}

\begin{rem}
Note that the statements are trivial unless $N,M$ are sufficiently
large relative to $-s,-r$; the point is that they can be taken
arbitrary.

Also, we emphasize that $s,r$ can take {\em any} value, unlike in the
usual real principal symbol radial point estimates; this is due to the
imaginary part $i\digamma$ of the principal symbol, which {\em is}
principal in the full scattering sense (no additional decay relative
to $\xi$).
\end{rem}

\begin{proof}
By multiplying from the left by an operator whose principal symbol is
$\tilde p^{-1}$ (recall the ellipticity assumption), one may assume
that $\tilde p$ is the identity homomorphism at each point, i.e.\ that
the principal symbol of $P$ is $\xi+i\digamma$ times the identity
operator on the fibers of the vector bundle $\tilde E$. Equip $\tilde
E$ with a Hermitian fiber metric; since $P$ has scalar principal
symbol, so does the adjoint, namely $\xi-i\digamma$ times the
identity. Now
write $P=P_R+iP_I$, with $P_R =\frac{P+P^*}{2}\in\Psisc^{1,0}$ formally self-adjoint,
with principal symbol $\xi$ times the identity,
$P_I\in\Psisc^{0,0}$ formally  skew-adjoint, with principal symbol
{\em at $\pa X$} given by $\digamma$ times the identity, thus is of the form
$\digamma+x\alpha$, $\alpha\in S^{0,0}$.

Due to a standard iterative argument, improving the regularity and decay by $1/2$ in 
each step while shrinking the support of $B_1$ slightly, it suffices to show the estimates under the a priori 
assumption that $u$ is in $\Hsc^{s-1/2,r-1/2}$ on
$\WFsc'(B_1)$. Furthermore, as $\xi\pm i\digamma$ have real principal
symbol in the standard sense, the usual propagation of singularities
theorem applies in $x>0$, which reduces the estimate to the case when
$\ep_0>0$ is fixed but small; we choose it so that
$|x\alpha|<\digamma/2$ for $x\in[0,\ep_0)$.

To prove the first statement of the proposition, with $\rho$ a defining
function of fiber infinity, such as $\rho=|\eta|^{-1}$ near $L$,
consider the scalar symbol
$$
a=\chi(x)x^{-r}\rho^{-s}\chi_1(\xi/\eta)\chi_2(\rho),
$$
where $\chi_1,\chi_2$ are identically $1$ near $0$ and have compact support,
$\chi\equiv 1$ near $0$, $d\chi$ supported near $\ep_0/2$. 
Note that on $\supp d(\chi_1\chi_2)$ we have elliptic estimates (as
$P$ is elliptic there since either $\xi/\eta$ is non-zero, or one is
at finite points where $\digamma$ gives the ellipticity), while on
$\supp d\chi$ we have a priori regularity of $u$ (in terms of control
on $B_2 u$).
Then with
$\tilde A=A^*A$, $A\in\Psisc^{s,r}$ having principal symbol $a$ times
the identity homomorphism, consider
\begin{equation}\label{eq:cx-rad-mod-comm} 
i(P^*\tilde A-\tilde AP)=i[P_R,\tilde A]+(P_I\tilde A+\tilde A P_I).  
\end{equation} 
Now the first term is in $\Psisc^{2s,2r+1}$, the second is in
$\Psisc^{2s,2r}$, so while they have the same differential order, the second
actually dominates in the decay sense, thus at finite points of
$\Tsc^*_{\pa X}X$; at fiber infinity of course they need to be
considered comparable. The principal symbol of the second term, in
$\Psisc^{2s,2r}$,
is $2(\digamma+x\alpha) a^2$, which is positive, bounded
below by $\digamma a^2$, say, for $x$ small, in
particular on $\supp a$ by our arrangements.
The principal symbol
of the first term, on the other hand, is $2xa H_\xi a=2ax(x\pa_x+\eta\pa_\eta)a$. Thus,
in view of $x\pa_x+\eta\pa_\eta$ being a smooth vector field tangent
to all boundaries of $\overline{\Tsc^*}X$, i.e.\ an element of
$\Vb(\overline{\Tsc^*}X)$, the principal symbol of the first term
can be absorbed into $2\digamma a^2$ away from the boundary of the
support of $\chi_1\chi_2$ and $\chi$; at both of those locations, however,
we have a priori/elliptic control. Thus, we have that the principal
symbol of \eqref{eq:cx-rad-mod-comm} is
$$
b^2+e+e_0,
$$
where
\begin{equation*}\begin{aligned}
b^2&=2(\digamma+x\alpha) a^2 +2ax \chi(x)\chi_1(\xi/\eta)\chi_2(\rho)  (x\pa_x+\eta\pa_\eta)(x^{-r}\rho^{-s})\\
e&=2ax x^{-r}\rho^{-s}\chi_1(\xi/\eta)\chi_2(\rho) (x\pa_x)\chi,\\
e_0&=2ax x^{-r}\rho^{-s}\chi(x)(x\pa_x+\eta\pa_\eta) (\chi_1(\xi/\eta)\chi_2(\rho)).
\end{aligned}\end{equation*}
Note that taking the non-negative square root, this indeed gives a
smooth $a$ as one can simply factor our all cutoffs, etc., so one
eventually needs to take the square root of $2\digamma$ plus $x$ times
a smooth function, which is thus strictly positive if $\chi,\chi_2$
have small supports.
Hence, taking
$B\in\Psisc^{s,r}$ with principal symbol $b$,
$E\in\Psisc^{2s,-\infty}$ with principal symbol $e$,
$E_0\in\Psisc^{2s,2r-1}$ with principal symbol $e_0$,
\begin{equation}\label{eq:cx-rad-mod-comm-pos}
i(P^*\tilde A-\tilde AP)=B^*B+E+E_0+F,
\end{equation}
where $E_0$ has $\WFsc'(E_0)$ is in the
elliptic set of $P$, while $\WFsc'(E)$
is near $x=\ep$ and $F\in\Psisc^{2s-1,2r-1}$. This gives
$$
\|Bu\|^2\leq 2|\langle \tilde Au,Pu\rangle|+|\langle Eu,u\rangle|+|\langle E_0u,u\rangle|+|\langle Fu,u\rangle|.
$$
Now, the first term is handled by the Cauchy-Schwartz inequality in a
standard way (cf.\ below),
while the latter terms (iteratively improving regularity for the $F$
term) are controlled by a priori assumptions, proving the estimate, a
priori for $u\in\dCI(X;\tilde E)$.

A standard regularization
argument, see e.g.\ \cite[Proof of
Propositions~2.3-2.4]{Vasy-Dyatlov:Microlocal-Kerr} or \cite{Vasy:Minicourse}, which is normally delicate at radial points, but not in this case,
due to the skew-adjoint part, $P_I$, proves the result. To see this, one replaces
$a$ by $a_\ep=a s_\ep r_\ep$ throughout this computation, where
$$
s_\ep=(1+\ep\rho^{-1})^{-\delta},\ r_\ep=(1+\ep x^{-1})^{-\delta},\
\ep\in(0,1],
$$
where any
$\delta\geq 1$ suffices. This makes the corresponding
$A_\ep\in\Psisc^{s+\delta,r+\delta}$ for $\ep>0$, but uniformly bounded in
$\Psisc^{s,r}$, $\ep\in(0,1]$, with $A_\ep\to A$ as $\ep\to 0$ in
$\Psisc^{s-\delta',r-\delta'}$ for any $\delta'>0$. Then in the analogue of
\eqref{eq:cx-rad-mod-comm},
\begin{equation}\label{eq:cx-rad-mod-comm-reg}
i(P^*\tilde A_\ep-\tilde A_\ep P)=i[P_R,\tilde A_\ep]+(P_I\tilde
A_\ep+\tilde A_\ep P_I),\qquad \tilde A_\ep=A_\ep^* A_\ep,
\end{equation}
the principal symbol of the second term, considered uniformly in
$\Psisc^{2s,2r}$,
is $2(\digamma+x\alpha) a_\ep^2$ which is positive for $x$ small, that
of the first term is $2xa_\ep H_\xi a_\ep=2a_\ep
x(x\pa_x+\eta\pa_\eta)a_\ep$. Now,
$$
ds_\ep=(\delta+1)\ep\rho^{-2}(1+\ep\rho^{-1})^{-\delta-1}\,d\rho=(\delta+1)
s_\ep \ep\rho^{-1}(1+\ep\rho^{-1})^{-1}\frac{d\rho}{\rho},
$$
with a similar computation also holding for $r_\ep$.
Since
$x\pa_x+\eta\pa_\eta\in\Vb(\overline{\Tsc^*}X)$, $\frac{d\rho}{\rho}
(x\pa_x+\eta\pa_\eta)$ is smooth on $\overline{\Tsc^*}X$, while
$\ep\rho^{-1}(1+\ep\rho^{-1})^{-1}$ is uniformly bounded in $S^{0,0}$,
with analogous statements also holding for the $r_\ep$ contributions,
the principal symbol of the first term of \eqref{eq:cx-rad-mod-comm-reg}
can be absorbed into $2\digamma a_\ep^2$ away from the boundary of the
support of $\chi_1\chi_2$ and $\chi$, where $a_\ep$ has a lower bound
$cs_\ep r_\ep$ for some $c>0$. As before, at both of these remaining locations
we have a priori/elliptic control. One then still has the analogue of
\eqref{eq:cx-rad-mod-comm-pos}, which gives for $\ep>0$,
$$
\|B_\ep u\|^2\leq 2|\langle \tilde A_\ep u,Pu\rangle|+|\langle E_\ep
u,u\rangle|+|\langle E_{0,\ep}u,u\rangle|+|\langle F_\ep u,u\rangle|,
$$
and now all terms but the first on the right hand side remain bounded
as $\ep\to 0$ due to the a priori assumptions and elliptic
estimates. On the other hand, one can apply the Cauchy-Schwartz
inequality to the first term of the right hand side, bounding it from above
by $\tilde\ep\|A_\ep u\|^2+\tilde\ep^{-1}\|A_\ep Pu\|^2$, absorbing a small
multiple ($\tilde\ep>0$ small) of $\|A_\ep u\|^2$ into $\|B_\ep u\|^2$ modulo lower order
terms (which are bounded by the a priori assumptions), which is possible as
the principal symbol of
$B_\ep$ is an elliptic multiple of $a_\ep$, and thus one obtains
the uniform boundedness of $\|B_\ep u\|^2$, $\ep\in(0,1]$. This proves $Bu\in
L^2$, completing the proof of the first half of the proposition.

For the second case, we take the same commutant, but now making sure
that $\chi\chi'=-\psi^2$ for a smooth function $\psi$. Then $P^*$ has
skew-adjoint part with the opposite sign of that of $P$, $-P_I$,
and
$$
i(P\tilde A-\tilde AP^*)=i[P_R,\tilde A]-(P_I\tilde A+\tilde A P_I).
$$
has principal symbol
$$
2ax(x\pa_x+\eta\pa_\eta)a-2a^2\digamma=-2b^2-2a^2(\digamma+x\beta)+e,
$$
where $e$ is supported where $d(\chi_1\chi_2)$ is (thus in the
elliptic set), while
$$
b=\psi(x)x^{-r+1}\rho^{-s}\chi_1(\xi/\eta)\chi_2(\rho)
$$
is elliptic for $x\in[0,\ep_0/2]$. Since the two terms not controlled by
elliptic estimates have matching signs, there is no need for a priori
control of $u$ in the sense of propagation, and we obtain the claimed
estimate by going through the regularization argument as above.
\end{proof}

\subsection{Full estimates in the one-form setting}
This gives real principal
type estimates up to $x=0$, which together with the rest of the preceding
discussion gives the coercivity of the system given by $L_0$ and
$L_1$, taking into account that $N_{0,\digamma},N_{1,\digamma}$ are in $\Psisc^{-1,0}$:

\begin{prop}\label{prop:final-one-form-est}
Let $\ep>0$.
For $u$ supported in $x<\ep$, writing $u=(u_0,u_1)$ for the
decomposition relative to $\Span\{\eta\}$ as in Remark~\ref{rem:microlocal-bundles},
we have estimates
$$
\|u_0\|_{s,r}+\|u_1\|_{s-1,r}+\|x^2D_x u_1\|_{s-1,r}\leq C(\|N_{0,\digamma} u\|_{s+1,r}+\|N_{1,\digamma} u\|_{s+1,r}+\|u\|_{-N,-M}).
$$
\end{prop}

\begin{rem}
Here the decomposition $(u_0,u_1)$ is defined only at fiber
infinity and even there only near $\xi=0$. However, away from fiber infinity the estimates for $u_0$
and $u_1$ are in the {\em same} space, and the same is true at fiber
infinity away from $\xi=0$ (in view of the ellipticity of $x^2D_x$ there), so this is irrelevant.
\end{rem}

\begin{proof}
Microlocally away from fiber infinity the estimate holds without
$\|N_{1,\digamma} u\|_{s+1,r}$ even (i.e.\ $\|B_0u\|_{s,r}$ can be so
estimated if $\WFsc'(B_0)$ is disjoint from fiber infinity, with $B_0\in\Psisc^{0,0}$), and it also holds at fiber
infinity away from $\xi=0$ in the same manner; namely for such $B_0$
we have
$$
\|B_0 u\|_{s,r}\leq C(\|N_{0,\digamma} u\|_{s+1,r}+\|u\|_{s-2,r-1}).
$$
Now we write the pre-
and postmultiplied version of
$$
\begin{pmatrix}P^\perp N_{0,\digamma}\\N_{1,\digamma}\end{pmatrix},
$$
defined microlocally near fiber infinity,
as
$$
\begin{pmatrix} A_{00}&A_{01}\\A_{10}&A_{11}\end{pmatrix}
$$
with $A_{ij}\in\Psisc^{-2,-1}$ if $i\neq j$, $A_{00}\in\Psisc^{-1,0}$
elliptic and $A_{11}\in\Psisc^{-1,0}$ satisfying the hypotheses of
Proposition~\ref{prop:not-quite-real-princ-scalar}. We write the
components of $u$ as $(u_0,u_1)$ corresponding to the
decomposition
relative to $\eta$, and we write $\tilde u=(\tilde u_0,\tilde u_1)$ for the modified
decomposition obtained by multiplying $u$ by the postmultiplier of
$\begin{pmatrix}P^\perp
  N_{0,\digamma}\\N_{1,\digamma}\end{pmatrix}$. Since the inverse of
the premultiplier preserves $\Hsc^{s+1,r}\oplus\Hsc^{s+1,r}$, we then
have
$$
\begin{pmatrix}
  A_{00}&A_{01}\\A_{10}&A_{11}\end{pmatrix}\begin{pmatrix}\tilde
  u_0\\\tilde u_1\end{pmatrix}
$$
controlled in $\Hsc^{s+1,r}\oplus\Hsc^{s+1,r}$ by $P^\perp
  N_{0,\digamma} u$ and $N_{1,\digamma}u$ in $\Hsc^{s+1,r}$.
Thus, using the first equation (involving $A_{0j}$),
writing it as $A_0 \tilde u=f_0$,
gives the microlocal elliptic estimate 
\begin{equation*}\begin{aligned}
\|B_1 \tilde u_0\|_{s,r}&\leq C(\|\tilde u_0\|_{s-1,r-1}+\|A_{00}
\tilde u_0\|_{s+1,r})\\
&\leq C(\|\tilde u_0\|_{s-1,r-1}+\|\tilde u_1\|_{s-1,r-1}+\|A_0 \tilde
u\|_{s+1,r}),\qquad A_0=\begin{pmatrix} A_{00}&A_{01}\end{pmatrix},
\end{aligned}\end{equation*}
where $B_1\in\Psisc^{0,0}$ has wave front set near $\xi=0$ at fiber
infinity, elliptic on a smaller neighborhood of $\xi=0$ at fiber infinity.
On the other hand, by
Proposition~\ref{prop:not-quite-real-princ-scalar}, taking into
account the order of $A_{11}\in\Psisc^{-1,0}$ and the support of $\tilde u_1$
(so that the $B_2$ term of the proposition is irrelevant), the second equation gives the estimate
\begin{equation*}\begin{aligned}
\|B_1\tilde u_1\|_{s-1,r}&\leq C(\|\tilde u_1\|_{s-2,r-1}+\|A_{11} \tilde u_1\|_{s+1,r})\\
&\leq C(\|\tilde u_1\|_{s-2,r-1}+\|\tilde u_0\|_{s-1,r-1}+\|A_1 \tilde
u\|_{s+1,r}), \qquad A_1=\begin{pmatrix} A_{10}&A_{11}\end{pmatrix},
\end{aligned}\end{equation*}
with $B_1$ as above. Moreover, since $A_{11}$ is an elliptic multiple
or order $(-2,0)$
of $x^2D_x+i\digamma$ microlocally, we have the microlocal elliptic estimate
\begin{equation*}\begin{aligned}
\|B_1(x^2D_x+i\digamma)\tilde u_1\|_{s-1,r}&\leq C(\|\tilde u_1\|_{s-2,r-1}+\|A_{11}
\tilde u_1\|_{s+1,r})\\
&\leq C(\|\tilde u_1\|_{s-2,r-1}+\|\tilde u_0\|_{s-1,r-1}+\|A_{1}
\tilde u\|_{s+1,r}).
\end{aligned}\end{equation*}
Thus, for $\alpha\geq 1$,
\begin{equation*}\begin{aligned}
&\alpha\|B_0 u\|_{s,r}+\|B_1\tilde u_0\|_{s,r}+\alpha\|B_1\tilde u_1\|_{s-1,r}+\|B_1(x^2D_x+i\digamma)\tilde u_1\|_{s-1,r}\\
&\qquad\leq
C(\alpha\|N_{0,\digamma}u\|_{s+1,r}+\alpha\|u\|_{s-2,r-1}\\
&\qquad\qquad+\|\tilde u_0\|_{s-1,r-1}+\alpha \|\tilde u_0\|_{s-1,r-1}+\|\tilde u_1\|_{s-1,r-1}+\alpha
\|\tilde u_1\|_{s-2,r-1}\\
&\qquad\qquad+\|A_{0} \tilde u\|_{s+1,r}+\|A_{1} \tilde u\|_{s+1,r}).
\end{aligned}\end{equation*}
Taking $\alpha>1$ sufficiently large, $C\|\tilde u_1\|_{s-1,r-1}$ on the
right hand side can be absorbed into the left hand side modulo $\|\tilde u_1\|_{s-2,r-2}$:
$$
\|\tilde u_1\|_{s-1,r-1}\leq C'(\|B_0 u\|_{s-1,r-1}+\|B_1\tilde u_1\|_{s-1,r-1}+\|\tilde u_1\|_{s-2,r-2})
$$
if $B_0$ and $B_1$ are so chosen that at each point at least one of
them is elliptic, as can be done.
This gives the estimate (with a new constant $C$, corresponding to any
fixed sufficiently large value of $\alpha$)
\begin{equation*}\begin{aligned}
&\|B_0 u\|_{s,r}+\|B_1 \tilde u_0\|_{s,r}+\|B_1 \tilde u_1\|_{s-1,r}+\|B_1 x^2D_x
\tilde u_1\|_{s-1,r}\\
&\leq C(\|\tilde u_0\|_{s-1,r-1}+\|\tilde
u_1\|_{s-2,r-1}+\|N_{0,\digamma}u\|_{s+1,r}+\|A_0 \tilde
u\|_{s+1,r}+\|A_1 \tilde u\|_{s+1,r}),
\end{aligned}\end{equation*}
and then the usual iteration in $s,r$ improves the error term to
\begin{equation}\begin{aligned}\label{eq:final-tilde-1-form-est}
&\|B_0 u\|_{s,r}+\|B_1 \tilde u_0\|_{s,r}+\|B_1 \tilde u_1\|_{s-1,r}+\|B_1 x^2D_x
\tilde u_1\|_{s-1,r}\\
&\leq C(\|\tilde u_0\|_{s-k,r-k}+\|\tilde
u_1\|_{s-1-k,r-k}+\|N_{0,\digamma}u\|_{s+1,r}+\|A_0 \tilde
u\|_{s+1,r}+\|A_1 \tilde u\|_{s+1,r})
\end{aligned}\end{equation}
for all $k$.

Now,
$$
\|A_0 \tilde u\|_{s+1,r}+\|A_1 \tilde u\|_{s+1,r}\leq C (\|N_{0,\digamma} u\|_{s+1,r}+\|N_{1,\digamma} u\|_{s+1,r}),
$$
as explained above. Similarly one has microlocal control of
$(u_0,u_1)$ in terms of $\tilde u_0,\tilde u_1$, with the key point
being that the premultiplier of $\begin{pmatrix}P^\perp
  N_{0,\digamma}\\N_{1,\digamma}\end{pmatrix}$, and its inverse, are
upper triangular with top right entry having principal symbol of the
form $(\xi+i\digamma)\tilde c+\tilde d$, with $\tilde c,\tilde d\in S^{-1,0}$, so the
regularity we proved on $\tilde u_1$ only gives rise to contributions
to $u_0$ in $\Hsc^{s,r}$, {\em not} in the space $\Hsc^{s-1,r}$ as one
would a priori expect. Therefore \eqref{eq:final-tilde-1-form-est}
gives the claimed estimate of the proposition.
\end{proof}

For $c>0$ small, the error term on the right hand side of the estimate
of Proposition~\ref{prop:final-one-form-est} can be absorbed into the
left hand side, as in \cite{UV:local}, \cite{SUV:Tensor}. Thus one obtains an invertibility result for
1-forms in the
normal gauge that is analogous to
Corollary~\ref{cor:normal-gauge-2-tensor-Fredholm} below in the
2-tensor setting, {\em but here without using the solenoidal gauge
  results of \cite{SUV:Tensor}}.

\section{The transform on 2-tensors in the normal
  gauge}\label{sec:2-tensor}

\subsection{The operators $L_1$ and $L_2$}
Now we turn to the 2-tensor setting. Recall the only issue
with the transform in this case is the lack of ellipticity of $L_0 I$
at fiber infinity. In this case, as for 1-forms, the
problem is still $\xi=0$, but we have ellipticity of $N_{0,\digamma}$
only on $\Span\{\eta\}^\perp\otimes \Span\{\eta\}^\perp$. A
computation similar to the one above shows the vanishing of the
principal symbol on $\Span\{\eta\}^\perp\otimes_s \Span\{\eta\}$ and
$\Span\{\eta\}\otimes \Span\{\eta\}$. The vanishing is simple
in the first case and quadratic in the second, essentially because as
above in Section~\ref{sec:1-form}, on $\Span\{\eta\}$ one may replace $\hat Y\cdot$ by $\xi \tilde
S$, so the order of vanishing is given by the number of factors of
$\Span\{\eta\}$. Then a similar argument as above directly deals with
$\Span\{\eta\}^\perp\otimes_s \Span\{\eta\}$, namely we just need
to
consider the map
\[
\tilde L_1 v(z)=x\int \chi_1(\lambda/x)v(\gamma_{x,y,\lambda,\omega})g_{\scl}(\omega\,\pa_y)\,d\lambda\,d\omega,
\]
where now we are mapping to (tangential) one-forms rather than scalars in the
one-forms setting.
Again, this is better considered as the normal-tangential component of
the map $L$ restricted to tangential-tangential tensors, for that is, by \eqref{eq:L-tensors}, trivializing the normal
1-forms with $x^{-2}\,dx$ as in the 1-form setting,
$$
L_1v(z)=x^2\int \chi(\lambda/x)v(\gamma_{x,y,\lambda,\omega})\lambda x^{-2} g_{\scl}(\omega\,\pa_y)\,d\lambda\,d\omega,
$$
which gives exactly this result when $\chi_1(s)=s\chi(s)$.
The resulting $N_{1,\digamma}$ still has the same even/odd properties as the
$L_1$ considered in the one form setting due to the odd number $1+2=3$
of vector/one-form factors appearing. Correspondingly, calculations as
above would give a real principal type system if there were no
$\Span\{\eta\}^\perp\otimes \Span\{\eta\}^\perp$ components.

Now, one is then tempted to consider the operator
\[
\tilde L_2 v(z)=\int \chi_2(\lambda/x)v(\gamma_{x,y,\lambda,\omega})\,d\lambda\,d\omega,
\]
mapping scalars to scalars to deal with
$\Span\{\eta\}^\perp\otimes \Span\{\eta\}^\perp$. 
This arises as the normal-normal component of the $L$ restricted to tangential-tangential tensors:
$$
L_2 v(z)=x^2\int \chi(\lambda/x)v(\gamma_{x,y,\lambda,\omega})(\lambda x^{-2})^2 \,d\lambda\,d\omega,
$$
provided we take $\chi_2(s)=s^2\chi(s)$ in this case. 
Unfortunately this produces similar behavior to $L_0$, and while at
the principal symbol level it is not hard to see that the appropriate
rows of the resulting matrix are linearly independent in an
relevant (non-elliptic) sense, see the discussions around \eqref{eq:2-tensor-11-entry-comp-1}, this is not so easy to see at the
subprincipal level, which is needed here.

\subsection{The symbol computation}
In spite of this, for our perturbation result involving
weights, we need to compute the full symbol of $L_2 I$ (more
precisely, the computation involves the symbol modulo terms {\em two}
orders below the leading term in the differential sense, {\em one}
order in the sense of decay, terms with more vanishing are
irrelevant below). Here, $L_2I$ 
is just the normal-normal component of $N_\digamma$ restricted to
tangential-tangential tensors, and we want to find its form,
in particular its precise vanishing properties at fiber infinity at $\xi=0$.
To do so, as in the one-form case, we perform the full symbol computation of \cite{UV:local} without restricting to
tangential-tangential tensors, with $\tilde\chi$ the localizer which is an
arbitrary smooth function on the cosphere bundle (not just the kind considered above for
$\chi$), using the oscillatory integral representation as in
Section~\ref{sec:1-form}, proceeding from scratch.

We already know
that we have a pseudodifferential operator
$$
\Ajd=e^{-\digamma/x}L_j Ie^{\digamma/x}\in\Psisc^{-1,0},
$$
with $I$ {\em not restricted} to tangential-tangential tensors, and with $\Ajd$ the component mapping to tangential-tangential ($j=0$),
tangential-normal ($j=1$) or normal-normal ($j=2$) tensors
given by
\begin{equation*}\begin{aligned}
\Ajd f(z)=\int e^{-\digamma/x(z)}
&e^{\digamma/x(\gamma_{z,\lambda,\omega}(t))} x^{-j}\lambda^j
(h(y)\omega)^{\otimes(2-j)}\\
&\tilde\chi(z,\lambda/x,\omega)
f(\gamma_{z,\lambda,\omega}(t))(\dot\gamma_{z,\lambda,\omega}(t),\dot\gamma_{z,\lambda,\omega}(t))\,dt\,|d\nu|,
\end{aligned}\end{equation*}
{\em where $\Ajd$ is understood to apply only to $f$ with
  support in $M$, thus for which the $t$-integral is in a fixed finite
  interval}.

Now,
\begin{equation*}\begin{aligned}
K_{\Ajd}(z,z')&=\int e^{-\digamma/x(z)} e^{\digamma/x(\gamma_{z,\lambda,\omega}(t))}x^{-j}\lambda^j
(h(y)\omega)^{\otimes(2-j)}\tilde\chi(z,\lambda/x,\omega)\\
&\qquad\qquad(\dot\gamma_{z,\lambda,\omega}(t)\otimes\dot\gamma_{z,\lambda,\omega}(t))\delta(z'-\gamma_{z,\lambda,\omega}(t))\,dt\,|d\nu|\\
&=
(2\pi)^{-n}
\int e^{-\digamma/x(z)} e^{\digamma/x(\gamma_{z,\lambda,\omega}(t))}x^{-j}\lambda^j
(h(y)\omega)^{\otimes(2-j)}\tilde\chi(z,\lambda/x,\omega)\\
&\qquad\qquad(\dot\gamma_{z,\lambda,\omega}(t)\otimes\dot\gamma_{z,\lambda,\omega}(t))e^{-i\zeta'\cdot(z'-\gamma_{z,\lambda,\omega}(t))}\,dt\,|d\nu|\,|d\zeta'|. 
\end{aligned}\end{equation*}
As remarked above, the $t$ integral is actually over a fixed finite
interval, say $|t|<T$, or one may explicitly insert a compactly supported cutoff
in $t$ instead. (So the only non-compact domain of integration is in
$\zeta'$, corresponding to the Fourier transform.)
Thus, using \eqref{eq:aj-in-terms-of-kernel}, so taking the inverse Fourier transform in $z'$ and evaluating at $\zeta$,
gives
\begin{equation}\begin{aligned}\label{eq:general-2-tensor-form-1}
a_{j,\digamma}(z,\zeta)=
\int e^{-\digamma/x(z)} &e^{\digamma/x(\gamma_{z,\lambda,\omega}(t))}x^{-j}\lambda^j
(h(y)\omega)^{\otimes(2-j)}\tilde\chi(z,\lambda/x,\omega)\\
&\qquad (\dot\gamma_{z,\lambda,\omega}(t)\otimes\dot\gamma_{z,\lambda,\omega}(t)) e^{-iz\cdot\zeta} e^{i\zeta\cdot\gamma_{z,\lambda,\omega}(t)}\,dt\,|d\nu|.
\end{aligned}\end{equation}
Translating into sc-coordinates, writing $(x,y)$ as local coordinates,
scattering covectors as $\xi\frac{dx}{x^2}+\eta\cdot\frac{dy}{x}$, and
$\gamma=(\gamma^{(1)},\gamma^{(2)})$, with $\gamma^{(1)}$ the $x$
component, $\gamma^{(2)}$ the $y$ component, we obtain
\begin{equation*}\begin{aligned}
&a_{j,\digamma}(x,y,\xi,\eta)\\
&=
\int e^{-\digamma/x} e^{\digamma/\gamma^{(1)}_{x,y,\lambda,\omega}(t)}x^{-j}\lambda^j
(h(y)\omega)^{\otimes(2-j)}\tilde\chi(x,y,\lambda/x,\omega)
(\dot\gamma_{x,y,\lambda,\omega}(t)\otimes\dot\gamma_{x,y,\lambda,\omega}(t))\\
&\qquad\qquad\qquad\qquad\qquad e^{i(\xi/x^2,\eta/x)\cdot(\gamma^{(1)}_{x,y,\lambda,\omega}(t)-x, \gamma^{(2)}_{x,y,\lambda,\omega}(t)-y)}\,dt\,|d\nu|,
\end{aligned}\end{equation*}
as in \eqref{eq:1-form-kernel-sc-form}.
We recall that
$$
\gamma_{x,y,\lambda,\omega}(t)=(x+\lambda t+\alpha
t^2+t^3\Gamma^{(1)}(x,y,\lambda,\omega,t),y+\omega t+t^2\Gamma^{(2)}(x,y,\lambda,\omega,t))
$$
while as a scattering tangent vector, i.e.\ expressed in terms of
$x^2\pa_x$ and $x\pa_y$,
\begin{equation*}\begin{aligned}
\dot\gamma_{x,y,\lambda,\omega}(t)=\gamma^{(1)}_{x,y,\lambda,\omega}(t)^{-1}(\gamma^{(1)}_{x,y,\lambda,\omega}(t)^{-1}(\lambda+2\alpha
t+t^2\tilde\Gamma^{(1)}(x,y,\lambda,\omega,t)),\omega+t\tilde\Gamma^{(2)}(x,y,\lambda,\omega,t)),
\end{aligned}\end{equation*}
with $\Gamma^{(1)},\Gamma^{(2)},\tilde\Gamma^{(1)},\tilde\Gamma^{(2)}$
smooth functions of $x,y,\lambda,\omega,t$.
Here the interval of integration in $t$, i.e.\ $T$,
will be small due to having to deal with the stationary phase
expansion as in the 1-form case.

Still following the argument in the 1-form case, we change the variables of integration to $\hat t=t/x$, and
$\hat\lambda=\lambda/x$, so the $\hat\lambda$ integral is in fact over
a fixed compact interval, but the $\hat t$ one is over $|\hat t|<T/x$
which grows as $x\to 0$.
We recall that the
phase is
$$
\xi(\hat\lambda\hat t+\alpha\hat t^2+x\hat
t^3\Gamma^{(1)}(x,y,x\hat\lambda,\omega,x\hat t))+\eta\cdot(\omega\hat t+x\hat t^2\Gamma^{(2)}(x,y,x\hat\lambda,\omega,x\hat t)),
$$
while the exponential damping factor (which we regard as a Schwartz
function, part of the amplitude, when one regards $\hat t$ as a
variable on $\RR$) is
\begin{equation*}\begin{aligned}
&-\digamma/x+\digamma/\gamma^{(1)}_{x,y,\lambda,\omega}(t)\\
&=
-\digamma(\lambda t+\alpha
t^2+t^3\Gamma^{(1)}(x,y,\lambda,\omega,t))x^{-1}(x+\lambda t+\alpha
t^2+t^3\Gamma^{(1)}(x,y,\lambda,\omega,t))^{-1}\\
&=-\digamma(\hat\lambda\hat t+\alpha\hat
t^2+\hat t^3 x\hat\Gamma^{(1)}(x,y,x\hat\lambda,\omega,x\hat t)),
\end{aligned}\end{equation*}
with $\hat\Gamma^{(1)}$ a smooth function.
The only subtlety in applying the stationary phase lemma is still that the
domain of integration in $\hat t$ is not compact, but this is handled
exactly as in the 1-form setting, for the 1-form vs.\ 2-tensor values
play no role in the argument.

Therefore one can use the standard stationary phase lemma, with the
stationary points (including the Hessian) having exactly the same structure as in the 1-form
setting. Then at $x=0$, the
stationary points of the phase are $\hat t=0$,
$\xi\hat\lambda+\eta\cdot\omega=0$, which remain critical points for
$x$ non-zero due to the $x\hat t^2$ vanishing of the other terms.  When $T$ is small, so $x\hat t$ is small, there are no other critical
points, so these critical points lie on a smooth codimension 2
submanifold of the parameter space.
This means that all terms of the
form $\hat
t x$ will have contributions which are 1 differentiable and 1 decay
order lower than the main terms, while $\hat
t^3 x$ will have contributions which are 2 differentiable and 1 decay
order lower than the main terms, and thus ignorable for our
purposes. Moreover,
when evaluated on tangential-tangential tensors (which is our
interest here),
$\dot\gamma_{x,y,\lambda,\omega}(t)$ can be replaced by
\begin{equation*}\begin{aligned}
\dot\gamma^{(2)}_{x,y,\lambda,\omega}&=\gamma^{(1)}_{x,y,\lambda,\omega}(t)^{-1} (\omega+\hat t
x\tilde\Gamma^{(2)}(x,y,x\hat\lambda,\omega,x\hat t))\\
&=x^{-1}(\omega+\hat t
x\hat\Gamma^{(2)}(x,y,x\hat\lambda,\omega,\hat t))
\end{aligned}\end{equation*}
with $\hat\Gamma^{(2)}$ smooth.

We recall from the one form discussion that $\Njd P^\perp$, $\Njd P^\parallel$, with
$P^\perp$, resp.\ $P^\parallel$, the microlocal orthogonal
projection with principal symbol $\Pi^\perp$, resp.\ $\Pi^\parallel$,
will have principal symbol given by the composition of principal
symbols, but here we need to compute to the subprincipal level. Moreover, as $\Njd$ is written as a left
quantization, if $P^\parallel,P^\perp$ are written as right
quantizations, the full amplitude is the composition of the full symbols,
evaluated at $(x,y)$ (the left, or `outgoing' variable of
$\Njd$), resp.\ $(x',y')$ (the right, or `incoming', variable
of $P^\perp,P^\parallel$). In addition, to get the full left symbol one simply
`left reduces', i.e.\ eliminates $(x',y')$ by the standard Taylor
series argument at the diagonal $(x,y)=(x',y')$. In the Euclidean
notation, to which the scattering algebra reduces to locally, this involves taking
derivatives of $a_{j,\digamma}$ in the momentum variables and
derivatives of the full symbol of $P^\parallel,P^\perp$ in the
position variables, evaluating the latter at $(x',y')=(x,y)$, with
each derivative reducing the symbolic order both in the differential
and in the decay sense by $1$. 

Thus, with $\tilde\chi=\chi(\lambda/x)=\chi(\hat\lambda)$,
we have that on
\begin{equation*}\begin{aligned}
&\Span\{\eta\}^{\perp}\otimes\Span\{\eta\}^{\perp}\ (k=0),
\ \Span\{\eta\}\otimes_s\Span\{\eta\}^{\perp}\ (k=1),\\
&\text{resp.}\ \{\eta\}\otimes\Span\{\eta\}\ (k=2),
\end{aligned}\end{equation*}
writing the sections in $\Span\{\eta\}$ factors explicitly as multiples
of $\frac{\eta}{|\eta|}$,
\begin{equation}\begin{aligned}\label{eq:general-2-tensor-form-2}
a_{j,\digamma}(x,y,\xi,\eta)\\
=
\int 
&e^{i(\xi x^{-2}(\gamma^{(1)}_{x,y,x\hat\lambda,\omega}(x\hat t)-x)+\eta x^{-1}(
  \gamma^{(2)}_{x,y,x\hat\lambda,\omega}(x\hat t)-y))}\\
&\qquad e^{-\digamma(\hat\lambda\hat t+\alpha\hat
t^2)}\hat\lambda^j(h(y)\omega)^{\otimes(2-j)}\chi(\hat\lambda)|\eta|^{-k}
(x\dot\gamma^{(2)}_{x,y,x\hat\lambda,\omega}(x\hat t)\cdot\eta)^k(x\dot\gamma^{(2)}_{x,y,x\hat\lambda,\omega}(x\hat t)\cdot)^{\otimes(2-k)}
\,d\hat t\,d\hat \lambda\,d\omega\\
=\int &e^{i(\xi(\hat\lambda\hat t+\alpha\hat t^2+x\hat
t^3\Gamma^{(1)}(x,y,x\hat\lambda,\omega,x\hat t))+\eta\cdot(\omega\hat
t+x\hat t^2\Gamma^{(2)}(x,y,x\hat\lambda,\omega,x\hat t)))}\\
&\qquad e^{-\digamma(\hat\lambda\hat t+\alpha\hat
t^2)}\hat\lambda^j(h(y)\omega)^{\otimes(2-j)}\chi(\hat\lambda)|\eta|^{-k}((\omega+\hat
t x\hat\Gamma^{(2)}(x,y,x\hat\lambda,\omega,\hat t))\cdot\eta)^k\\
&\qquad\qquad \big((\omega+\hat t x\hat\Gamma^{(2)}(x,y,x\hat\lambda,\omega,\hat t))\cdot\big)^{\otimes(2-k)}\,d\hat t\,d\hat \lambda\,d\omega,
\end{aligned}\end{equation}
up to errors that are $O(x\langle\xi,\eta\rangle^{-1})$ relative to
the a priori order, $(-1,0)$, arising from the $0$-th order symbol in
the oscillatory integral and the 2-dimensional space in which the
stationary phase lemma is applied. Indeed the error can be
improved to $O(x\langle\xi,\eta\rangle^{-2})$ if the composition with
the projections $P^\parallel\otimes P^\parallel$, etc., is written out
as discussed in the paragraph above. However, we will deal with $k=2$,
when this improvement would be important, in a different manner below.

Notice that
$$
(x\dot\gamma^{(2)}_{x,y,x\hat\lambda,\omega}(x\hat t)\cdot\eta)e^{i \eta x^{-1}(
  \gamma^{(2)}_{x,y,x\hat\lambda,\omega}(x\hat t)-y)}=x\pa_{\hat t}e^{i \eta x^{-1}(
  \gamma^{(2)}_{x,y,x\hat\lambda,\omega}(x\hat t)-y)},
$$
when $k\geq
1$, integration by
parts once gives that this is
\begin{equation*}\begin{aligned}
a_{j,\digamma}(x,y,\xi,\eta)\\
=-\int&e^{i \eta x^{-1}(
  \gamma^{(2)}_{x,y,x\hat\lambda,\omega}(x\hat t)-y)}\\
&\qquad x\pa_{\hat t}\Big(e^{i(\xi x^{-2}(\gamma^{(1)}_{x,y,x\hat\lambda,\omega}(x\hat t)-x))}e^{-\digamma(\hat\lambda\hat t+\alpha\hat
t^2)}(\dot\gamma^{(2)}_{x,y,x\hat\lambda,\omega}(x\hat t)\cdot)^{\otimes(2-k)}(x\dot\gamma^{(2)}_{x,y,x\hat\lambda,\omega}(x\hat t)\cdot\eta)^{k-1}\Big)\\
&\qquad\qquad \hat\lambda^j(h(y)\omega)^{\otimes(2-j)}\chi(\hat\lambda)|\eta|^{-k}\,d\hat t\,d\hat \lambda\,d\omega.
\end{aligned}\end{equation*}
If $k=1$, expanding the derivative, if $\ell$ derivatives (so $\ell=0,1$) hit the first
exponential (the phase term) and thus $k-\ell$ the second (the
amplitude) one obtains $\xi^\ell$ times the oscillatory factor
$e^{i(\xi x^{-2}(\gamma^{(1)}_{x,y,x\hat\lambda,\omega}(x\hat t)-x))}$
times a symbol of order $0$ (notice that $x\pa_{\hat t}(x^{-2}(\gamma^{(1)}_{x,y,x\hat\lambda,\omega}(x\hat t)-x))=\hat\lambda+2\alpha \hat t+\hat
t^2 x\tilde\Gamma^{(1)}(x,y,x\hat\lambda,\omega,x\hat t)$). In view of the overall weight
$|\eta|^{-k}$, we deduce that, modulo terms two orders down, in terms
of the differential order, $a_{j,\digamma}$ is a sum of terms of the
form of symbols of order $(-k-1,0)$ times $\xi^\ell$, $0\leq\ell\leq
k$. Notice that here $\eta$ can be replaced by any other element of
$S^{1,0}$ which has the same principal symbol, i.e.\ differs from
$\eta$ by an element $r$ of $S^{0,-1}$, for one then expands
$(x\dot\gamma^{(2)}_{x,y,x\hat\lambda,\omega}(x\hat
t)\cdot(\eta+r))^k$ into terms involving
$(x\dot\gamma^{(2)}_{x,y,x\hat\lambda,\omega}(x\hat t)\cdot\eta)^{k'}$
and $(x\dot\gamma^{(2)}_{x,y,x\hat\lambda,\omega}(x\hat t)\cdot
r)^{k-k'}$; for the latter factors one does not need an integration by
parts argument to get the desired conclusion, while for the former it
proceeds exactly as beforehand.

If $k=2$, there are subtleties because subprincipal terms are
involved. So to complete the analysis, we use that $I\circ \dsymm=0$, so
$Ie^{\digamma/x}\dsymmw=0$; recall that
$$
\dsymmw=e^{-\digamma/x}\dsymm e^{\digamma/x}.
$$
Concretely, we use:

\begin{lemma}
The
microlocal projection to $\Span\{\eta\}\otimes\Span\{\eta\}$,
$P^\parallel\otimes P^\parallel$, given by  Proposition~\ref{prop:microlocal-bundles}, is (modulo microlocally smoothing terms)
$(\dsymmY d_Y)G(\dsymmY d_Y)^*$, where $G\in\Psisc^{-4,0}(X)$ is a parametrix for the microlocally
elliptic operator $(\dsymmY d_Y)^*\dsymmY d_Y$, and where
$d_Y,\dsymmY$ are considered as elements of $\Psisc^{1,0}$ between various {\em scattering}
bundles, e.g.\ $d^Y v=\sum 
(x\pa_{y_j} v)\,\frac{dy_j}{x}$.
\end{lemma}

\begin{proof}
We just need to note that $(\dsymmY d_Y)G(\dsymmY d_Y)^*$ satisfies
all the requirements of Proposition~\ref{prop:microlocal-bundles}.

Indeed,
it has the correct principal symbol,
$\Pi^\parallel\otimes\Pi^\parallel$, as $d_Y$, $\dsymmY$ have
principal symbol $i^{-1} \eta\otimes\cdot$, so $\dsymmY d_Y$ (acting on scalar functions) has principal symbol
$-\eta\otimes\eta$. Thus its adjoint with respect to $g_{\scl}$ has principal symbol given by
evaluation on $-\eta\otimes\eta$, which is regarded as a 2-tensor via
$g_{\scl}$, hence $(\dsymmY d_Y)^*\dsymmY d_Y$ has principal symbol
$|\eta|^4$ (which is microlocally elliptic away from $\eta=0$).
In combination this gives that $(\dsymmY d_Y)G(\dsymmY d_Y)^*$ has
principal symbol $P^\parallel\otimes P^\parallel$.

Note that $G$ is
microlocally formally self-adjoint since $(\dsymmY d_Y)^*\dsymmY d_Y$
is such, so $(\dsymmY d_Y)G(\dsymmY d_Y)^*$ is also microlocally
formally self-adjoint. Finally, using the microlocal parametrix property of $G$,
$$
\big((\dsymmY d_Y)G(\dsymmY d_Y)^*\big)^2=(\dsymmY d_Y)\big(G(\dsymmY
d_Y)^*(\dsymmY d_Y\big))G(\dsymmY d_Y)^*
$$
microlocally differs from $(\dsymmY d_Y)G(\dsymmY d_Y)^*$ by a
smoothing operator.

This shows that all the properties in
Proposition~\ref{prop:microlocal-bundles} are satisfied, completing
the proof of the lemma.
\end{proof}

As a consequence of this lemma, the
computation on the range of $P^\parallel\otimes P^\parallel$,
amounts to that on the range of $\dsymmY d_Y$. Now, a computation gives that on
tangential (scattering) forms, such as those in the range of $d_Y$, when $g$ is in
the normal gauge,
$$
\dsymmw u=\big(e^{-\digamma/x}(x^2\pa_x +x^2a
)e^{\digamma/x}u\big)\otimes_s \frac{dx}{x^2}+\dsymmY u
$$
for suitable smooth $a$, which means that
$$
I e^{\digamma/x}\dsymmY u=-I e^{\digamma/x}\Big(\big(e^{-\digamma/x}(x^2\pa_x +x^2a
)e^{\digamma/x}u\big)\otimes_s \frac{dx}{x^2}\Big).
$$
Composing with $d_Y$ from the right, i.e.\ taking $u=d_Y v$, and commuting $e^{-\digamma/x}(x^2\pa_x +x^2a
)e^{\digamma/x}$ through $d_Y$, we have that
\begin{equation}\label{eq:convert-dY-normal-dsymm}
I e^{\digamma/x}\dsymmY d_Y v=-I e^{\digamma/x}\Big(\big(d_Y\big(e^{-\digamma/x}(x^2\pa_x +x^2a
)e^{\digamma/x}\big)+x^2\tilde a\big)v\otimes_s \frac{dx}{x^2}\Big),
\end{equation}
with $\tilde a$ smooth. The $x^2\tilde a$ term is two orders lower
than the a priori order, and thus completely negligible for our
purposes. (Even if a one order lower term had been created, it would
not cause any issues: one would either have a $x^2D_x$ factor or a
$d^Y$ factor left, modulo two orders lower terms, and each of these
can be handled as above.)  The advantage
of this rewriting is that we can work with $I e^{\digamma/x}(\cdot\otimes_s \frac{dx}{x^2} )d_Y$, and
we only need to be concerned about it at the principal symbol level;
we obtain an extra factor of $\xi+i\digamma-ix^2\tilde a$ after the composition.
Correspondingly,
$A_{j,\digamma}(\cdot\otimes_s \frac{dx}{x^2} )d_Y$ has principal symbol given by, up to a non-zero
constant factor,
\begin{equation*}\begin{aligned}
&b_{j,\digamma}(x,y,\xi,\eta)\\
&=
\int e^{-\digamma/x} e^{\digamma/\gamma^{(1)}_{x,y,\lambda,\omega}(t)}x^{-j}\lambda^j
(h(y)\omega)^{\otimes(2-j)}\tilde\chi(x,y,\lambda/x,\omega)
(x^2\dot\gamma^{(1)}_{x,y,\lambda,\omega}(t))(x\dot\gamma^{(2)}_{x,y,\lambda,\omega}(t)\cdot
\eta)\\
&\qquad\qquad\qquad\qquad\qquad e^{i(\xi/x^2,\eta/x)\cdot(\gamma^{(1)}_{x,y,\lambda,\omega}(t)-x, \gamma^{(2)}_{x,y,\lambda,\omega}(t)-y)}\,dt\,|d\nu|;
\end{aligned}\end{equation*}
here $x^2\dot\gamma^{(1)}_{x,y,\lambda,\omega}(t)$ appears due to
$\otimes_s \frac{dx}{x^2}$ above in \eqref{eq:convert-dY-normal-dsymm}.
This gives
\begin{equation*}\begin{aligned}
b_{j,\digamma}(x,y,\xi,\eta)\\
=
\int 
&e^{i(\xi x^{-2}(\gamma^{(1)}_{x,y,x\hat\lambda,\omega}(x\hat t)-x)+\eta x^{-1}(
  \gamma^{(2)}_{x,y,x\hat\lambda,\omega}(x\hat t)-y))}\\
&\quad e^{-\digamma(\hat\lambda\hat t+\alpha\hat
t^2)}\hat\lambda^j(h(y)\omega)^{\otimes(2-j)}\chi(\hat\lambda)|\eta|^{-1}
(x\dot\gamma^{(2)}_{x,y,x\hat\lambda,\omega}(x\hat t)\cdot\eta)(x^2\dot\gamma^{(1)}_{x,y,x\hat\lambda,\omega}(x\hat t)\cdot)
\,d\hat t\,d\hat \lambda\,d\omega.
\end{aligned}\end{equation*}
This can be handled exactly as above, so an integration by parts as
above in $\hat t$ gives
\begin{equation*}\begin{aligned}
b_{j,\digamma}(x,y,\xi,\eta)\\
=-\int&e^{i \eta x^{-1}(
  \gamma^{(2)}_{x,y,x\hat\lambda,\omega}(x\hat t)-y)}\\
&\qquad x\pa_{\hat t}\Big(e^{i(\xi x^{-2}(\gamma^{(1)}_{x,y,x\hat\lambda,\omega}(x\hat t)-x))}e^{-\digamma(\hat\lambda\hat t+\alpha\hat
t^2)}(x^2\dot\gamma^{(1)}_{x,y,x\hat\lambda,\omega}(x\hat t))\Big)\\
&\qquad\qquad \hat\lambda^j(h(y)\omega)^{\otimes(2-j)}\chi(\hat\lambda)|\eta|^{-1}\,d\hat t\,d\hat \lambda\,d\omega. 
\end{aligned}\end{equation*}
Again the derivative either produces a $\xi$ factor, or a term
which is one order lower than the a priori order.
Taking into account to the extra factor of
$x^2D_x+i\digamma-ix^2\tilde a$ we had, as well as $G(\dsymmY
d_Y)^*\in\Psisc^{-2,0}$, and also the same continuity properties as in
the 1-form setting, this proves:

\begin{prop}\label{prop:N-digamma-structure-basic}
Let $\xi_\digamma=\xi+i\digamma$.
The full symbol of the operator
$$
N_\digamma=\begin{pmatrix}N_{0,\digamma}\\N_{1,\digamma}\\N_{2,\digamma}\end{pmatrix},
$$
with domain restricted to
tangential-tangential
tensors, relative to the $\Span\{\eta\}$-based decomposition of the domain,
has the form
$$
\begin{pmatrix}
a_{00}^{(0)}&a_{01}^{(1)}\xi_\digamma+a_{01}^{(0)}&a_{02}^{(2)}\xi_\digamma^2+a_{02}^{(1)}\xi_\digamma+a_{02}^{(0)}\\
a_{10}^{(0)}&a_{11}^{(1)}\xi_\digamma+a_{11}^{(0)}&a_{12}^{(2)}\xi_\digamma^2+a_{12}^{(1)}\xi_\digamma+a_{12}^{(0)}\\
a_{20}^{(0)}&a_{21}^{(1)}\xi_\digamma+a_{21}^{(0)}&a_{22}^{(2)}\xi_\digamma^2+a_{22}^{(1)}\xi_\digamma+a_{22}^{(0)}
\end{pmatrix},
$$
where $a_{ij}^{(k)}\in S^{-1-j,0}$ for all $i,j,k$.

Furthermore, $a_{ij}^{(k)}\in S^{-1-j,0}$ depend continuously on the
metric $g$ (with the $\CI$ topology on $g$) as long as $g$ is
$C^k$-close (for suitable $k$) to a background metric $g_0$ satisfying
the strict convexity assumptions on
the metric, the boundary and $x$.
\end{prop}

In
addition, at $x=0$ we have
\begin{equation*}\begin{aligned}
&a_{j,\digamma}(0,y,\xi,\eta)\\
&=
\int e^{i(\xi(\hat\lambda\hat t+\alpha\hat t^2)+\eta\cdot(\omega\hat t))}e^{-\digamma(\hat\lambda\hat t+\alpha\hat
t^2)}\hat\lambda^j(h(y)\omega)^{\otimes(2-j)}\chi(\hat\lambda)|\eta|^{-k} (\omega\cdot\eta)^k(\omega\cdot)^{\otimes(2-k)}\,d\hat t\,d\hat
\lambda\,d\omega\\
&=
\int e^{i((\xi+i\digamma)(\hat\lambda\hat t+\alpha\hat t^2)+\eta\cdot(\omega\hat t))}\hat\lambda^j(h(y)\omega)^{\otimes(2-j)}\chi(\hat\lambda) |\eta|^{-k}  (\omega\cdot\eta)^k(\omega\cdot)^{\otimes(2-k)}\,d\hat t\,d\hat
\lambda\,d\omega\\
&=
\int_{\sphere^{n-2}}|\eta|^{-k}  (\omega\cdot\eta)^k(h(y)\omega)^{\otimes(2-j)}(\omega\cdot)^{\otimes(2-k)}\Big(\int e^{i((\xi+i\digamma)(\hat\lambda\hat t+\alpha\hat t^2)+(\eta\cdot\omega)\hat t)}\hat\lambda^j\chi(\hat\lambda) \,d\hat t\,d\hat
\lambda\Big)\,d\omega.
\end{aligned}\end{equation*}
We recall that $\alpha=\alpha(x,y,\lambda,\omega)$ so at $x=0$,
$\alpha(0,y,0\cdot\hat\lambda,\omega)=\alpha(0,y,0,\omega)$, and it is a
quadratic form in $\omega$.

{\em Again, it is notationally convenient to assume, as we do from now on, that at $y$ at which we
perform the computations below, $h$ is the Euclidean metric.} As in the
one form setting, this does not affect even the integration by parts
arguments below since $h(y)$ would be a prefactor of the integrals.

We now apply the projection $P^\perp$ (quantization of the projection to
$\Span\{\eta\}^\perp$ as in Proposition~\ref{prop:microlocal-bundles}) and its tensor powers from the left: for the
tangential-tangential, tangential-normal, resp.\ normal-normal
components we apply $P^\perp\otimes P^\perp$, resp.\ $P^\perp$, resp.\
$\Id$, which means for the symbol computation (we are working at $x=0$!) that we compose with
$\Pi^\perp\otimes \Pi^\perp$, resp.\ $\Pi^\perp\otimes_s I$, resp.\
$I$ from the left. This replaces $\omega^{2-j}$ by
$(\omega^\perp)^{2-j}$ with the result
\begin{equation*}\begin{aligned}
&\tilde a_{j,\digamma}(0,y,\xi,\eta)\\
&=
\int_{\sphere^{n-2}}|\eta|^{-k}  (\omega\cdot\eta)^k(\omega^\perp)^{\otimes(2-j)}(\omega^\perp\cdot)^{\otimes(2-k)}\Big(\int e^{i((\xi+i\digamma)(\hat\lambda\hat t+\alpha\hat t^2)+(\eta\cdot\omega)\hat t)}\hat\lambda^j\chi(\hat\lambda) \,d\hat t\,d\hat
\lambda\Big)\,d\omega,
\end{aligned}\end{equation*}
where we used that $(\omega\cdot)^{\otimes(2-k)}$ is being applied to
the $\eta$-orthogonal factors, so it may be written as
$(\omega^\perp\cdot)^{\otimes(2-k)}$. This means that at $\xi=0$ the overall
parity of the integrand in $\omega^\perp$ is $(-1)^{j+k}$ apart from
the appearance of $\omega^\perp$ in the exponent (via $\alpha$) of
$e^{-\digamma(\hat\lambda\hat t+\alpha\hat t^2)}$. The latter is due to the
$\hat t^2$ prefactor of $\alpha$, giving quadratic vanishing at the
critical set, only contributes one order lower
terms, so modulo these the integral vanishes when $j$ and $k$ have the
opposite parity. This proves that the first two rows of $N_\digamma$,
when composed with the projections as described, have the following
form:

\begin{prop}\label{prop:N-digamma-structure-proj}
Let $\xi_\digamma=\xi+i\digamma$.
The symbol of the operator 
$$
\begin{pmatrix}(P^\perp\otimes
  P^\perp)N_{0,\digamma}\\(P^\perp\otimes_s I)N_{1,\digamma}\end{pmatrix},
$$
with domain restricted to
tangential-tangential
tensors, relative to the $\Span\{\eta\}$-based decomposition of the
domain, at $x=0$
has the form
$$
\begin{pmatrix}
a_{00}^{(0)}&a_{01}^{(1)}\xi_\digamma+a_{01}^{(0)}&a_{02}^{(2)}\xi_\digamma^2+a_{02}^{(1)}\xi_\digamma+a_{02}^{(0)}\\
a_{10}^{(1)}\xi_\digamma+a_{10}^{(0)}&a_{11}^{(1)}\xi_\digamma+a_{11}^{(0)}&a_{12}^{(2)}\xi_\digamma^2+a_{12}^{(1)}\xi_\digamma+a_{12}^{(0)}\\
\end{pmatrix},
$$
where $a_{ij}^{(k)}\in S^{-1-\max(i,j),0}$ for all $i,j,k$.
\end{prop}

We can compute the leading terms quite easily: for $j=k=0$ this is
\begin{equation*}\begin{aligned}
&\tilde a_{0,\digamma}(0,y,\xi,\eta)\\
&=
\int_{\sphere^{n-2}}(\omega^\perp)^{\otimes 2}(\omega^\perp\cdot)^{\otimes
  2}\Big(\int e^{i((\xi+i\digamma)(\hat\lambda\hat t+\alpha\hat t^2)+(\eta\cdot\omega)\hat t)}\chi(\hat\lambda) \,d\hat t\,d\hat
\lambda\Big)\,d\omega\\
&=
\int_{\sphere^{n-2}}(\omega^\perp)^{\otimes 2}(\omega^\perp\cdot)^{\otimes
  2}\Big(\int e^{i((\xi\hat\lambda\hat t+\alpha\hat t^2)+(\eta\cdot\omega)\hat t)}e^{-\digamma(\hat\lambda\hat t+\alpha\hat t^2)}\chi(\hat\lambda) \,d\hat t\,d\hat
\lambda\Big)\,d\omega.
\end{aligned}\end{equation*}
At the critical points of the phase, $\hat t=0$,
$\xi\hat\lambda+\eta\cdot\omega=0$, where $\omega^\perp$ and
$\hat\lambda$ are variables along the critical set, this gives, up to an
overall elliptic factor,
$$
\int_{\sphere^{n-3}}(\omega^\perp)^{\otimes 2}(\omega^\perp\cdot)^{\otimes
  2}\Big(\int\chi(\hat\lambda) \,d\hat
\lambda\Big)\,d\omega^\perp,
$$
which is elliptic for $\chi\geq 0$ with $\chi(0)>0$.
On the other hand, for $j=k=1$, 
\begin{equation*}\begin{aligned}
&\tilde a_{1,\digamma}(0,y,\xi,\eta)\\
&=
\int_{\sphere^{n-2}}|\eta|^{-1}  (\omega\cdot\eta)(\omega^\perp) (\omega^\perp\cdot)\Big(\int e^{i((\xi+i\digamma)(\hat\lambda\hat t+\alpha\hat t^2)+(\eta\cdot\omega)\hat t)}\hat\lambda\chi(\hat\lambda) \,d\hat t\,d\hat
\lambda\Big)\,d\omega. 
\end{aligned}\end{equation*}
Writing $i(\omega\cdot\eta)e^{i (\eta\cdot\omega)\hat t}=\pa_{\hat
  t}e^{i (\eta\cdot\omega)\hat t}$ and integrating by parts yields
\begin{equation}\begin{aligned}\label{eq:2-tensor-11-entry-comp-1}
&\tilde a_{1,\digamma}(0,y,\xi,\eta)\\
&=
i\int_{\sphere^{n-2}}|\eta|^{-1}  (\omega^\perp) (\omega^\perp\cdot)\Big(\int e^{i((\xi+i\digamma)(\hat\lambda\hat t+\alpha\hat t^2)+(\eta\cdot\omega)\hat t)}(\xi+i\digamma)(\hat\lambda+2\alpha\hat t)\hat\lambda\chi(\hat\lambda) \,d\hat t\,d\hat
\lambda\Big)\,d\omega\\
&=
i|\eta|^{-1} (\xi+i\digamma)
\int_{\sphere^{n-2}} (\omega^\perp) (\omega^\perp\cdot)\Big(\int e^{i((\xi+i\digamma)(\hat\lambda\hat t+\alpha\hat t^2)+(\eta\cdot\omega)\hat t)}(\hat\lambda+2\alpha\hat t)\hat\lambda\chi(\hat\lambda) \,d\hat t\,d\hat
\lambda\Big)\,d\omega,
\end{aligned}\end{equation}
and now the integral (the factor after $|\eta|^{-1}(\xi+i\digamma) $) at the critical points of the phase $\hat t=0$,
$\xi\hat\lambda+\eta\cdot\omega=0$, gives, up to an
overall elliptic factor,
$$
\int_{\sphere^{n-3}}(\omega^\perp) (\omega^\perp\cdot)\Big(\int\hat\lambda^2\chi(\hat\lambda) \,d\hat
\lambda\Big)\,d\omega^\perp,
$$
i.e.\ for the same reasons as in the $j=k=0$ case above, when
$\chi\geq 0$, $\chi(0)>0$, \eqref{eq:2-tensor-11-entry-comp-1} is an elliptic
multiple of $|\eta|^{-1}(\xi+i\digamma)$!

Finally, when $j=0$, $k=1$, we have
\begin{equation*}\begin{aligned}
&\tilde a_{0,\digamma}(0,y,\xi,\eta)\\
&=
\int_{\sphere^{n-2}}|\eta|^{-1}  (\omega^\perp)^{\otimes 2}(\omega^\perp\cdot) (\omega\cdot\eta)\Big(\int e^{i((\xi+i\digamma)(\hat\lambda\hat t+\alpha\hat t^2)+(\eta\cdot\omega)\hat t)}\chi(\hat\lambda) \,d\hat t\,d\hat
\lambda\Big)\,d\omega. 
\end{aligned}\end{equation*}
This, using  $i(\omega\cdot\eta)e^{i (\eta\cdot\omega)\hat t}=\pa_{\hat
  t}e^{i (\eta\cdot\omega)\hat t}$ as above, gives
\begin{equation}\begin{aligned}\label{eq:2-tensor-01-entry-comp-1}
&\tilde a_{0,\digamma}(0,y,\xi,\eta)\\
&=
i|\eta|^{-1} (\xi+i\digamma)
\int_{\sphere^{n-2}} (\omega^\perp)^{\otimes 2}(\omega^\perp\cdot) \Big(\int e^{i((\xi+i\digamma)(\hat\lambda\hat t+\alpha\hat t^2)+(\eta\cdot\omega)\hat t)}(\hat\lambda+2\alpha\hat t)\chi(\hat\lambda) \,d\hat t\,d\hat
\lambda\Big)\,d\omega. 
\end{aligned}\end{equation}
Now the leading term of the integral, due to the contributions from the critical points,
is 
$$
\int_{\sphere^{n-3}}(\omega^\perp)^{\otimes 2}(\omega^\perp\cdot) \Big(\int\hat\lambda\chi(\hat\lambda) \,d\hat
\lambda\Big)\,d\omega^\perp,
$$
which vanishes for $\chi$ even, so for such $\chi$, the $(0,1)$ entry
has principal symbol which at $x=0$ is a multiple of $\xi_\digamma$,
and the multiplier is in $S^{-3,0}$ (one order lower than the
previous results).

In summary, we have the following result:

\begin{prop}\label{prop:N-digamma-structure}
Suppose $\chi\geq 0$, $\chi(0)>0$, $\chi$ even.
Let $\xi_\digamma=\xi+i\digamma$.
The full symbol of the operator 
$$
\begin{pmatrix}(P^\perp\otimes
  P^\perp)N_{0,\digamma}\\(P^\perp\otimes_s I)N_{1,\digamma}\end{pmatrix},
$$
with domain restricted to
tangential-tangential
tensors, relative to the $\Span\{\eta\}$-based decomposition of the
domain, at $x=0$
has the form
$$
\begin{pmatrix}
a_{00}^{(0)}&a_{01}^{(1)}\xi_\digamma+a_{01}^{(0)}&a_{02}^{(2)}\xi_\digamma^2+a_{02}^{(1)}\xi_\digamma+a_{02}^{(0)}\\
a_{10}^{(1)}\xi_\digamma+a_{10}^{(0)}&a_{11}^{(1)}\xi_\digamma+a_{11}^{(0)}&a_{12}^{(2)}\xi_\digamma^2+a_{12}^{(1)}\xi_\digamma+a_{12}^{(0)}\\
\end{pmatrix},
$$
where $a_{ij}^{(k)}\in S^{-1-\max(i,j),0}$ for all $i,j,k$, and $a_{00}^{(0)}$ and
$a_{11}^{(1)}$ (these are the multipliers of the
leading terms along the `diagonal') are elliptic in $S^{-1,0}$ and
$S^{-2,0}$, respectively and $a_{01}^{(0)},a_{11}^{(0)}\in S^{-2,-1}$, i.e.\ in
addition to the above statements vanish at $x=0$, and $a_{01}^{(1)}\in
S^{-3,0}$.
\end{prop}

The problem with this result is that we have too few equations: we
would have needed to prove some non-degeneracy properties of an
operator like $L_2$ to have a self-contained
result.
We deal with this by using our results in the twisted
solenoidal gauge as a background estimate. When doing so, the last
column (corresponding $u_2$) can be regarded as forcing based on the
background estimate. This is not the case for the first two columns,
however, so it is useful to note that they can be diagonalized:

\begin{lemma}\label{lemma:N-digamma-improved}
The first two columns of $\begin{pmatrix}(P^\perp\otimes
  P^\perp)N_{0,\digamma}\\(P^\perp\otimes_s
  I)N_{1,\digamma}\end{pmatrix}$ expanded relative to the $\Span\{\eta\}$-based decomposition of the
domain, can be multiplied from the left by an operator with symbol
$\begin{pmatrix}1&0\\b^{(1)}\xi_\digamma+b^{(0)}&1\end{pmatrix}$ and
from the right by an operator with symbol of the form
$\begin{pmatrix}1&c^{(1)}\xi_\digamma+c^{(0)}\\0&1\end{pmatrix}$
with $b^{(j)}$ and $c^{(j)}$ in $S^{-1,0}$, such that the result has
principal symbol of the form
$$
\begin{pmatrix}
\tilde a_{00}^{(0)}&0\\
0&\tilde a_{11}^{(1)}\xi_\digamma+\tilde a_{11}^{(0)}\\
\end{pmatrix},
$$
with $\tilde a_{00}^{(0)}=a_{00}^{0}$ elliptic in $S^{-1,0}$, $\tilde
a_{11}^{(1)}\in S^{-2,0}$ elliptic, $\tilde a_{11}^{(0)}\in
S^{-2,-1}$.

Furthermore, $\tilde a_{00}^{(0)},\tilde a_{11}^{(k)}$ depend
continuously (in the indicated spaces) on the
metric $g$ (with the $\CI$ topology on $g$) as long as $g$ is
$C^k$-close (for suitable $k$) to a background metric $g_0$ satisfying
the strict convexity assumptions on
the metric, the boundary and $x$.
\end{lemma}

\begin{proof}
The proof is completely parallel to that of Corollary~\ref{cor:1-form-pre-post-mult}.
\end{proof}

This lemma will be used below as the input for the regularity theory
in the normal gauge.

\section{Fredholm theory for 2-tensors in the normal
  gauge}\label{sec:Fredholm}

\subsection{Fredholm theory for the geodesic X-ray transform in the
  normal gauge}\label{sec:Fredholm-geodesic}
We are now ready to discuss Fredholm properties for the 2-tensor
transform in the normal gauge; for this recall that $X$ is defined by
the artificial boundary, see \eqref{eq:artificial-bdy-mfld}.
The solenoidal gauge approach tells us that one can recover the solenoidal part
of $u$ from $N_\digamma u$ in a lossless, in terms of the order of the weighted
Sobolev spaces involved, manner, at least for $c$
small (where $c$ defines the domain $\Omega$): Recall that
$$
\dsymmw=e^{-\digamma/x}\dsymm e^{\digamma/x}
$$
is the conjugate symmetric gradient, and
$$
\delta^s_\digamma=e^{\digamma/x}\delta^s e^{-\digamma/x}
$$
is its adjoint relative to scattering metric $g_{\scl}$.

\begin{thm}\label{thm:lossless-solenoidal}
Let $s=0$.
There exists $c_0>0$ such that for $0<c<c_0$, on
$\Omega_c=\{x_c>0\}\cap M$, $x_c=\tilde x+c$,
one has $u=u^s+\dsymmw v$, where
\begin{equation}\label{eq:lossless-solenoidal}
\|u^s\|_{s,r}\leq C\|N_\digamma u\|_{s+1,r}.
\end{equation}

Furthermore, the constants $c_0$ and $C$ can be taken to be independent of the
metric $g$ as long as $g$ is
$C^k$-close (for suitable $k$) to a background metric $g_0$ satisfying the assumptions on
the metric.
\end{thm}

\begin{rem}
We remark that even the loss, in terms of the order of the weighted
Sobolev spaces involved, in recovering the solenoidal part
of $u$ from $N_\digamma u$ would not be an issue if $\pa_\inter\Omega\cap\pa X=\emptyset$, for
then in \eqref{eq:almost-final-inv} below, in the second appearance of
$P_{\Omega_1\setminus\Omega}$, which is the problematic one,
$\gamma_{\pa_\inter\Omega}P_{\Omega_1\setminus\Omega}$ is lossless as
the loss of weight is irrelevant in this case. Here
$\pa_\inter\Omega=\pa M\cap X$. Thus, even the lossy
estimate would suffice to if we assumed that $\pa_\inter\Omega\cap\pa
X=\emptyset$, i.e.\ we worked globally within the boundary.
\end{rem}

\begin{rem}
It would be straightforward to allow general $s\geq 0$, but this would
require an improvement of the results of \cite{SUV:Tensor} by
developing elliptic boundary regularity theory in the
boundary-scattering setting for the Dirichlet problem for
$\delta^s_\digamma\dsymmw$ for various domains such as $\Omega$. This would proceed by proving b-sc
regularity, `b' (i.e.\ conormal regularity) at $\pa_\inter\Omega$,
`sc' at $\pa X$ at first, and then using the operator to improve the
regularity to full standard Sobolev regularity at $\pa_\inter\Omega$, in the
appropriate uniform sense to $\pa X$, analogously to how one
proves first tangential regularity for standard boundary value
problems (on compact domains with smooth boundary), and then obtains
normal regularity using the operator. Note that even though $\Omega$
is a domain with corners, there are no additional issues at the
corners unlike for standard boundary value problems in domains with
corners, since the scattering operators are very differently
behaved from standard operators at $\pa X$. Since this theory and
improvement are not needed for our
main results, we refrain from developing this theory in the present paper.
\end{rem}

\begin{proof}
Recall first that $\pa_\inter\Omega=\pa M\cap X$ is the internal (in
$X$) part of $\pa\Omega$, and similarly for neighborhoods of $\Omega$,
such as $\Omega_1$, considered in \cite{SUV:Tensor}.

We have the formula
\begin{equation}\begin{aligned}\label{eq:almost-final-inv}
&(\Id+(r_{10}-\dsymmw B_\Omega
\gamma_{\pa_\inter\Omega}
P_{\Omega_1\setminus\Omega}) K_2)^{-1}\\
&\qquad\qquad\circ(r_{10}-\dsymmw B_\Omega
\gamma_{\pa_\inter\Omega}
P_{\Omega_1\setminus\Omega})
\cS_{\digamma,\Omega_1}r_{21}\cS_{\digamma,\Omega_2}
GN_\digamma=\cS_{\digamma,\Omega}
\end{aligned}\end{equation}
from \cite[Equation~(4.20)]{SUV:Tensor}, with the various operators
defined {\em and estimated} in that paper, and for $s=0$ the discussions of that paper
{\em almost} give this estimate: Lemma~4.13 of that paper, which
controls $P_{\Omega_1\setminus\Omega}$, a local left inverse of
$\dsymmw$ on $\Omega_1\setminus\Omega$ with Dirichlet boundary
conditions on $\pa_\inter\Omega_1$, loses decay (relevant for the
second appearance of this operator only in this formula, as $K_2$
gains infinite order decay),
and the result one gets directly is
$$
\|u^s\|_{s,r-\alpha}\leq C\|N_\digamma u\|_{s+1,r}
$$
for $\alpha=2$, which is too weak for the theorem.
However, we improve Lemma~4.13 of \cite{SUV:Tensor} below in the appendix in 
Lemma~\ref{lemma:local-ds-inverse-forms-improved} to a lossless
version, which directly proves \eqref{eq:lossless-solenoidal} for
$s=0$.

Finally the uniformity of the estimate in $g$ follows from the
continuous dependence of $N_{\digamma}$ on $g$, as noted at the end of Section~\ref{subsec:L0}.
\end{proof}

Now, we solve for $v$ in the decomposition $u=u^s+\dsymmw v$ when $u$ is in the normal gauge, i.e.\ its
normal components vanish. As shown in \cite{SUV:Tensor},
in the decomposition of 1-forms, resp.\ symmetric 2-tensors, into
normal and tangential, resp.\ normal-normal, normal-tangential and
tangential-tangential components, the principal symbol of $\dsymmw$ is
\begin{equation*}
\begin{pmatrix}
\xi+i\digamma&0\\
\frac{1}{2}\eta\otimes&\frac{1}{2}(\xi+i\digamma)\\
a&\eta\otimes_s
\end{pmatrix},
\end{equation*}
where $a$ is a smooth bundle map. In fact, if
we use normal coordinates for $g$, then the full operator in the top
right entry (and not just its principal symbol) is
identically $0$, as follows from a Christoffel symbol computation. 
Indeed, 
denoting the index corresponding to the normal variable by $0$, the
Christoffel symbol needed is $\Gamma^i_{00}$ (where $i\neq 0$), which
is given by $\frac{1}{2}g^{ij}$ times $\pa_0 g_{j0}+\pa_0g_{0j}-\pa_j
g_{00}$, and in normal coordinates (relative to a level set of $x$) all the components being
differentiated are constant.
Thus, if $u$ is in the normal gauge, so $u_{NN}=0$ and $u_{NT}=0$, we
get equations for $v_N$ and $v_T$:
\begin{equation}\begin{aligned}\label{eq:normal-gauge-reduction-1}
&u^s_{NN}+A_{NN}v_N=0,\\
&u^s_{NT}+A_{NT}v_T+B_{NT}v_N=0,
\end{aligned}\end{equation}
where $A_{NN}\in\Diffsc^1$ has principal symbol $\xi+i\digamma$,
$B_{NT}\in\Diffsc^1$ has principal symbol $\frac{1}{2}\eta\otimes$,
and $A_{NT}$ has principal symbol $\frac{1}{2}(\xi+i\digamma)$. But
from the first equation of \eqref{eq:normal-gauge-reduction-1}, using
Proposition~\ref{prop:not-quite-real-princ-scalar}, we deduce that
\begin{equation}\label{eq:normal-gauge-reduction-rad-1}
\|v_N\|_{s,r}+\|x^2 D_x v_N\|_{s,r}\leq C\|u^s_{NN}\|_{s,r}\leq C\|N_\digamma u\|_{s+1,r}.
\end{equation}
Then from the second equation of \eqref{eq:normal-gauge-reduction-1} we deduce that
\begin{equation}\begin{aligned}\label{eq:normal-gauge-reduction-rad-2}
&\|v_T\|_{s-1,r}+\|x^2 D_x v_T\|_{s-1,r}\\
&\leq
C(\|u^s_{NT}\|_{s-1,r}+\|B_{NT}v_N\|_{s-1,r})\\
&\leq C(\|N_\digamma u\|_{s,r}+\|v_N\|_{s,r})\leq C \|N_\digamma u\|_{s+1,r}.
\end{aligned}\end{equation}
In fact, applying $x^2D_x$ to the second equation of
\eqref{eq:normal-gauge-reduction-1} and using that
$x^2D_x v_N\in \Hsc^{s,r}$ (with an estimate as above),
we conclude that
$$
A_{NT}(x^2 D_x)v_T=-x^2D_x u^s_{NT}-[x^2 D_x,B_{NT}]v_N-B_{NT}x^2D_x v_N-[x^2D_x,A_{NT}]v_T,
$$
so, using Proposition~\ref{prop:not-quite-real-princ-scalar},  as well as that $x^2D_x$ commutes with $A_{NT}$ at the principal symbol
level, so the commutator is of order $(0,-2)$,
\begin{equation*}\begin{aligned}
&\|x^2D_x v_T\|_{s-1,r}+\|(x^2 D_x)^2 v_T\|_{s-1,r}\\
&\leq
C(\|u^s_{NT}\|_{s,r}+\|v_N\|_{s,r-1}+\|B_{NT}x^2 D_x v_N\|_{s-1,r}+\|v_T\|_{s-1,r-2})\\
&\leq C(\|N_\digamma u\|_{s+1,r}+\|v_N\|_{s,r-1}+\|x^2D_x
v_N\|_{s,r}+\|N_\digamma u\|_{s+1,r})\leq C \|N_\digamma u\|_{s+1,r},
\end{aligned}\end{equation*}
proving \eqref{eq:normal-gauge-reduction-rad-2}, and
where the last inequality also used \eqref{eq:normal-gauge-reduction-rad-1}.
This gives that $u$, which is $u^s+\dsymmw v$, satisfies
\begin{equation}\label{eq:u-coiso-0}
\|u\|_{s-2,r}\leq C(\|N_\digamma u\|_{s+1,r}+\|v\|_{s-1,r})\leq C \|N_\digamma u\|_{s+1,r},
\end{equation}
which is a loss of 2 derivatives relative to the solenoidal
gauge. Notice also that $v$ satisfies
$x^2 D_x v\in \Hsc^{s-1,r}$, thus $\dsymmw v$ satisfies a similar
estimate (here the action of $x^2D_x$ on tangential tensors makes
sense directly):
$$
(x^2D_x)\dsymmw v=\dsymmw(x^2D_x v)+[\dsymmw,x^2D_x]v
$$
implies, as the commutator is in $x\Diffsc^1$,
$$
\|(x^2D_x)\dsymmw v\|_{s-2,r}\leq C(\|x^2D_x
v\|_{s-1,r}+\|v\|_{s-1,r+1})\leq C\|N_\digamma u\|_{s+1,r}.
$$
Hence, also taking advantage of Theorem~\ref{thm:lossless-solenoidal}, 
\[
\|x^2D_x u\|_{s-2,r}\leq C \|N_\digamma u\|_{s+1,r}
\]
as well. Finally $(x^2 D_x)^2v\in\Hsc^{s-1,r}$ as well:
\[
(x^2D_x)^2\dsymmw
v=\dsymmw(x^2D_x)^2v+2[x^2D_x,\dsymmw](x^2D_x v)-[x^2D_x,[\dsymmw,x^2D_x]]v,
\]
so
\[
\|(x^2D_x)^2\dsymmw v\|_{s-2,r}\leq C(\|(x^2 D_x)^2
v\|_{s-1,r}+\|x^2 D_x v\|_{s-1,r-1}+\|v\|_{s-1,r-2}).
\]
This gives
\begin{equation}\label{eq:u-coiso-2}
\|(x^2D_x)^2 u\|_{s-2,r}\leq C \|N_\digamma u\|_{s+1,r},
\end{equation}
i.e.\ $u$ satisfies coisotropic estimates.

Now, $v$ in fact only enters into particular components of $u$ in the
decomposition of $u$ as $(u_0,u_1,u_2)$ corresponding to the decomposition
relative to $\Span\{\eta\}$, and it is then straightforward to obtain
a more precise estimate directly from the argument above. We, however,
proceed differently and instead recover it from
Proposition~\ref{prop:N-digamma-structure} above:
Proposition~\ref{prop:N-digamma-structure} is crucial in any case for
the microlocally weighted transform considered below.

\begin{thm}\label{thm:normal-gauge-2-tensor-Fredholm-est}
There exists $c_0>0$ such that for $0<c<c_0$, on
$\Omega_c=\{x_c>0\}\cap M$, $x_c=\tilde x+c$,
with $s=0$,
we have for $u$ in the normal gauge, written as $u=(u_0,u_1,u_2)$
relative to the $\Span\{\eta\}$-based tensorial decomposition, that
\begin{equation}\begin{aligned}\label{eq:Fredholm-est}
&\|u_0\|_{s,r}+\|u_1\|_{s-1,r}+\|x^2 D_x u_1\|_{s-1,r}\\
&\qquad\qquad+\|u_2\|_{s-2,r}+\|(x^2D_x)
u_2\|_{s-2,r}+\|(x^2D_x)^2
u_2\|_{s-2,r}\\
&\leq C\|N_\digamma u\|_{s+1,r}.
\end{aligned}\end{equation}

Furthermore, the constants $c_0$ and $C$ can be taken to be independent of the
metric $g$ as long as $g$ is
$C^k$-close (for suitable $k$) to a background metric $g_0$ satisfying the assumptions on
the metric.
\end{thm}

\begin{proof}
We use the operator matrix in
Proposition~\ref{prop:N-digamma-structure}, pre- and postmultiplied as
in Lemma~\ref{lemma:N-digamma-improved}, after regarding the $u_2$
terms as forcing. Note that the
postmultiplication preserves the space $\Hsc^{s+1,r}$. Write the new
combination of $u_0$ and $u_1$ given by 
$\begin{pmatrix}1&C^{(1)}(x^2D_x+i\digamma)+C^{(0)}\\0&1\end{pmatrix}^{-1}\begin{pmatrix}u_0\\u_1\end{pmatrix}$
with $C^{(j)}$ in $\Psisc^{-1,0}$ as in
Lemma~\ref{lemma:N-digamma-improved}, as $\begin{pmatrix}\tilde u_0\\
  \tilde u_1\end{pmatrix}$. With
$B_{0,\digamma},B_{1,\digamma}\in\Psisc^{-1,0}$ as the two rows of the
result of Proposition~\ref{prop:N-digamma-structure}, and the tilded
versions $\tilde B_{0,\digamma},\tilde B_{1,\digamma}\in\Psisc^{-1,0}$
arising from the two rows of Lemma~\ref{lemma:N-digamma-improved}, we obtain
pseudodifferential equations, in which we regard the off-diagonal
terms as forcing, i.e.\ put them on the right hand side of the
equation. Thus, the $0$-th row, i.e.\ that of $\tilde B_{0,\digamma}$, yields an elliptic estimate
(keeping in mind the order of $\tilde b_{00}^{(0)}$)
\begin{equation}\begin{aligned}\label{eq:u_0-elliptic-est}
\|\tilde u_0\|_{s,r}\leq C(&\|\tilde u_0\|_{s-1,r-1}+\|\tilde u_1\|_{s-2,r-1}+\|x^2 D_x
\tilde u_1\|_{s-2,r-1}\\
&+\|u_2\|_{s-2,r-1}+\|(x^2D_x)u_2\|_{s-2,r-1}+\|(x^2
D_x)^2u_2\|_{s-2,r-1}+\|\tilde B_{0,\digamma} \tilde u\|_{s+1,r})\\
&\leq C\|N_\digamma u\|_{s+1,r},
\end{aligned}\end{equation}
where we used \eqref{eq:u-coiso-0}-\eqref{eq:u-coiso-2}.

Turning to the 1st row, i.e.\ that of $\tilde B_{1,\digamma}$,
due to the imaginary part of the principal symbol, independently of
the weight $r$, the combination of
Proposition~\ref{prop:not-quite-real-princ-scalar} and standard
real principal type estimates yields
\begin{equation}\begin{aligned}\label{eq:u_1-real-princ-est}
\|\tilde u_1\|_{s-1,r}+\|x^2 D_x \tilde u_1\|_{s-1,r}\leq
C(&\|\tilde u_1\|_{s-2,r-1}+\|\tilde u_0\|_{s-2,r-1}+\|(x^2D_x)\tilde u_0\|_{s-2,r-1}\\
&+\|u_2\|_{s-2,r-1}+\|(x^2
D_x)u_2\|_{s-2,r-1}\\
&\qquad+\|(x^2D_x)^2u_2\|_{s-2,r-1}+\|\tilde B_{1,\digamma} \tilde u\|_{s+1,r})\\
&\leq C\|N_\digamma u\|_{s+1,r}.
\end{aligned}\end{equation}
Together with \eqref{eq:u-coiso-0}-\eqref{eq:u-coiso-2}, \eqref{eq:u_0-elliptic-est}-\eqref{eq:u_1-real-princ-est} imply
\eqref{eq:Fredholm-est} with $(u_0,u_1)$ replaced by $(\tilde
u_0,\tilde u_1)$. Finally,
$$
\begin{pmatrix}u_0\\u_1\end{pmatrix}=\begin{pmatrix}1&C^{(1)}(x^2D_x+i\digamma)+C^{(0)} \\0&1\end{pmatrix}\begin{pmatrix}\tilde
  u_0\\ \tilde u_1\end{pmatrix}
$$
proves the theorem.
\end{proof}

We now consider 
$$
N_\digamma:\cX\to\cY 
$$
where
\begin{equation}\begin{aligned}\label{eq:cX-def}
\cX=\{u=(u_0,u_1,u_2):\ &u_0\in\Hsc^{s,r},\ u_1,x^2D_x
u_1\in\Hsc^{s-1,r}, \\
&u_2,(x^2D_x)u_2,(x^2 D_x)^2u_2\in\Hsc^{s-2,r},\
\supp u\subset\overline\Omega\},
\end{aligned}\end{equation}
with the natural norm (and inner product: this is a Hilbert space), so
elements of $\cX$ are tangential-tangential tensors, and
$$
\cY=\Hsc^{s+1,r}(X;\Sym^2\Tsc^*X).
$$
Notice that this mapping property of $N_\digamma$ follows from
Proposition~\ref{prop:N-digamma-structure-basic}, and that the spaces
are independent of the metric $g$, with the dependence of
$N_\digamma$ on $g$ continuous as a map between these spaces as long
as $g$ is $C^k$-close to a metric $g_0$ satisfying the assumptions on
the metric (with both in the normal gauge).

We then have from Theorem~\ref{thm:normal-gauge-2-tensor-Fredholm-est}:

\begin{cor}\label{cor:normal-gauge-2-tensor-Fredholm}
There exists $c_0>0$ such that for $0<c<c_0$, on
$\Omega_c=\{x_c>0\}\cap M$, $x_c=\tilde x+c$,
and with $\cX,\cY$ as above,
the operator $N_\digamma:\cX\to\cY$ satisfies
\begin{equation}\label{eq:2-tensor-Fredholm-brief}
\|u\|_{\cX}\leq C\|N_\digamma u\|_{\cY},\ u\in\cX,
\end{equation}
so $N_\digamma$ injective and has closed range.

Thus, it has a left
inverse, which we denote by $N_\digamma^{-1}$ with a slight abuse of notation, which is continuous
$\cY\to\cX$.

Furthermore, the constants $c_0$ and $C$ can be taken to be independent of the
metric $g$ as long as $g$ is
$C^k$-close (for suitable $k$) to a background metric $g_0$ satisfying the assumptions on
the metric.
\end{cor}

\begin{proof}
Due to Theorem~\ref{thm:normal-gauge-2-tensor-Fredholm-est}, resulting
in \eqref{eq:2-tensor-Fredholm-brief},
$N_\digamma:\cX\to\cY$ is injective and has closed range. Letting
$\cR$ be this range, being a closed subspace of $\cY$ it is a Hilbert
space, so $N_\digamma:\cX\to\cR$ is invertible, with a continuous
inverse, by the open mapping theorem. Composing this inverse from the
right with the orthogonal projection from $\cY$ to $\cR$ we obtain the
desired left inverse.
\end{proof}

\subsection{Extension to 
weights}  \label{sec:Fredholm-weights}
We are also interested in generalizations of $I$ by adding
weights:
$$
\tilde I f(\beta)=\int_{\gamma_\beta} a(\gamma(s),\dot\gamma(s))f(\gamma(s))(\dot\gamma(s),\dot\gamma(s))\,ds
$$
with the notation of Section~\ref{sec:approach}, so $\beta\in
S^*\tilde M$, $\gamma_\beta$ the geodesic through $\beta$, $a$ a given
weight function.
More generally consider an $N\times N$ system of transforms, $f=(f_1,\ldots,f_N)$,
$$
(\tilde I f)_i(\beta)=\int_{\gamma_\beta} A^j_i(\gamma(s),\dot\gamma(s))f_j(\gamma(s))(\dot\gamma(s),\dot\gamma(s))\,ds
$$
Here we require $A_i^j$ to be smooth, but rather than imposing $C^k$
estimates on $A_i^j$ to measure closeness to the identity weight, we
work with weaker estimates. Namely, with $\ep$ such that $x<\ep$ on
$\overline{\Omega_c}$ (so $\ep>c$), which corresponds to a transform
with data at $x=\ep$, we assume that the derivatives of
$A_i^j=A_i^j(x,y,\lambda,\omega)$ have the property that $A_i^j$ remains
bounded under iterated applications of
\begin{equation}\label{eq:edge-derivs}
x\pa_x,\pa_y,x\pa_\lambda,\pa_\omega,
\end{equation}
where
e.g.\ $\pa_\lambda$ stands for derivative in the third slot. (These
are called ``edge derivatives'' by Mazzeo \cite{Mazzeo:Edge}.) We write
$\|.\|_{C_{\rm sc}^k}$   for the norm on the space of $\CI$ functions $a$ given by the maximum, over products of up to $k$ vector fields on the list
\eqref{eq:edge-derivs}, of the supremum of these products applied to $a$
evaluated on $\overline{\Omega_c}$ in the $(x,y)$ variables,
$|\lambda|\leq\lambda_0$, $\omega\in\sphere^{n-1}$ with $\lambda_0$
chosen so that all the geodesics used in $LI$ have
$|\lambda|\leq\lambda_0$ (so the support of the cutoff $\chi$ lies in
$[-c\lambda_0, c\lambda_0]$). The reason for so weakening the
requirements is that the weights that arise in the pseudolinearization discussed in
the next section are well-behaved in this sense, with the key point being that these weights are a
priori $C^0$ close to (half of) $\delta_i^j$ which would suffice for
elliptic problems, but not $C^k$ close for $k\geq 1$. This is an issue because for our non-elliptic problem closeness in a $C^k$-type norm is needed,
with the crucial gain, however, that the derivatives only need to be
taken relative to the vector fields \eqref{eq:edge-derivs}.

Then, with $L$ defined identically to the case of
$I$ in the first case, and the $N\times N$ diagonal matrix with the
previous $L$ as
the diagonal entry in the second case, we have

\begin{thm}\label{thm:X-ray-weights}
There exists $c_0>0$ such that for $0<c<c_0$, on
$\Omega_c=\{x_c>0\}\cap M$, $x_c=\tilde x+c$,
the operator $\tilde N_\digamma=L\circ\tilde I$ maps
$$
\tilde N_\digamma:\cX^N\to\cY^N.
$$

Moreover, there exist $A_0>0$ and $c_0>0$ such that if $0<c<c_0$ and
$\|A^j_i-\delta^j_i\|_{C_{\rm sc}^k}<A_0$ (or the analogous statement
holds for a constant multiple of $\delta_i^j$,
such as $-\frac{1}{2}\delta^j_i$) then we have
\begin{equation}\label{eq:2-tensor-Fredholm-brief-1}
\|u\|_{\cX^N}\leq C\|\tilde N_\digamma u\|_{\cY^N},\ u\in\cX^N,
\end{equation}
so $\tilde N_\digamma$ injective and has closed range.

Thus, it has a left
inverse, which we denote by $\tilde N_\digamma^{-1}$ with a slight abuse of notation, which is continuous
$\cY^N\to\cX^N$.

Furthermore, the constants $A_0$, $c_0$, and $C$ in \eqref{eq:2-tensor-Fredholm-brief-1}, can be taken to be independent of the
metric $g$ as long as $g$ is
$C^k$-close (for suitable $k$) to a background metric $g_0$ satisfying the assumptions on
the metric.
\end{thm}

\begin{rem}
The space $C_{\rm sc}^k$ is the natural one appearing in the actual application, see Lemma~\ref{lemma:A}, and cannot be replaced there with the classical $C^k$. 
\end{rem}

\begin{proof}
The first part is {\em almost} immediate by explicitly writing out $\tilde
N_\digamma$ as in Section~\ref{sec:2-tensor}. For instance, the
additional weight does not affect the phase function, so the fact that
$\tilde N_{\digamma}$ is in $\Psisc^{-1,0}$ is unaffected, as is the
structure of the principal symbol
computation. Note that in the oscillatory integral computation leading
to the principal symbols the weights
$A_i^j$ are evaluated at
\[
\begin{aligned} 
(\gamma(t),\dot\gamma(t))&=\big(x+\lambda t+\alpha
t^2+t^3\Gamma^{(1)}(x,y,\lambda,\omega,t),y+\omega t+t^2\Gamma^{(2)}(x,y,\lambda,\omega,t),\\
&\qquad\qquad \lambda+2\alpha
t+t^2\tilde\Gamma^{(1)}(x,y,\lambda,\omega,t),\omega+t\tilde\Gamma^{(2)}(x,y,\lambda,\omega,t)\big),
\end{aligned}
\]
with $\Gamma^{(1)},\Gamma^{(2)},\tilde\Gamma^{(1)},\tilde\Gamma^{(2)}$
smooth functions of $x,y,\lambda,\omega,t$. Then one introduces $\hat
t=t/x$, $\hat\lambda=\lambda/x$, so the evaluation is at
\begin{equation*}\begin{aligned}
(\gamma(x\hat t),\dot\gamma(x\hat t))&=\big(x+x^2(\hat\lambda \hat t+\alpha
\hat t^2+x\hat t^3\Gamma^{(1)}(x,y,x\hat\lambda,\omega,x\hat
t)),y+x(\omega \hat t+x\hat
t^2\Gamma^{(2)}(x,y,x\hat\lambda,\omega,x\hat t)),\\
&\qquad\qquad x(\hat\lambda+2\alpha
\hat t+x\hat t^2\tilde\Gamma^{(1)}(x,y,x\hat
\lambda,\omega,t)),\omega+x\hat
t\tilde\Gamma^{(2)}(x,y,x\hat\lambda,\omega,x\hat t)\big).
\end{aligned}\end{equation*}
The stationary phase lemma shows that as long as iterated derivatives of $A_i^j (\gamma(x\hat
t),\dot\gamma(x\hat t))$ in
$x\pa_x,\pa_y,\pa_{\hat\lambda},\pa_\omega,\pa_{\hat t}$ are bounded,
the operator is in the same class as the unweighted one. By the chain
rule these are bounded by the $x\pa_x,\pa_y,x\pa_\lambda,\pa_\omega$
derivatives of $A_i^j$, which are exactly the derivatives giving rise
to the $\mathcal{A}^k$-norms.

There is only one real
subtlety, namely where $I\dsymm=0$ was used in the $k=2$ case (range
of $P^\parallel\otimes P^\parallel$) to deal with subprincipal terms; this is not satisfied for
$\tilde I$. However, $I\dsymm =0$ relies on $\sX\iota_{\sX} u=\dsymm
u(\sX,\sX)$ for all $u$, where $\sX$ is the tangent vector field of a
geodesic, see the discussion in the appendix; the integral of $\sX v$
along the geodesic vanishes for
any function $v$ (such as $v=\iota_{\sX} u$) of compact support by the fundamental theorem of
calculus. Thus, if $f_j=\dsymm u_j$, $\sX=\dot\gamma$,
\begin{equation}\begin{aligned}\label{eq:weighted-IBP}
(\tilde I f)_i(\beta)&=\int_{\gamma_\beta}
A^j_i(\gamma(s),\dot\gamma(s))\sX(\gamma(s))\iota_{\sX(\gamma(s))}u_j(\gamma(s))\,ds\\
&=\int_{\gamma_\beta}
\sX(\gamma(s))\big(A^j_i(\gamma(s),\dot\gamma(s))(\gamma(s))\iota_{\sX(\gamma(s))}u_j(\gamma(s))\big)\,ds\\
&\qquad\qquad-\int_{\gamma_\beta}
\sX(\gamma(s))(A^j_i(\gamma(s),\dot\gamma(s)))\iota_{\sX(\gamma(s))}u_j(\gamma(s))\,ds\\
&=-\int_{\gamma_\beta}
\sX(\gamma(s))(A^j_i(\gamma(s),\dot\gamma(s)))\iota_{\sX(\gamma(s))}u_j(\gamma(s))\,ds\\
&=-\int_{\gamma_\beta}
\tilde A^j_i(\gamma(s),\dot\gamma(s))
u_j(\gamma(s))(\dot\gamma(s))\,ds,\\
&\qquad \tilde A^j_i(\gamma(s),\dot\gamma(s))=(\sX (\gamma(s)))(A^j_i(\gamma(s),\dot\gamma(s)))
\end{aligned}\end{equation}
and now notice that the right hand side is a microlocally weighted
1-form X-ray transform. Crucially this means that $\tilde
N_{j,\digamma} \dsymmw $, while not $0$, is a transform of the same
form {\em with the same $\dot\gamma$, resp.\ $x\dot\gamma^{(2)}$ appearing in the argument
  as in \eqref{eq:general-2-tensor-form-1} and
  \eqref{eq:general-2-tensor-form-2}, albeit only to the first
  power}. Notice that a priori, $\tilde N_{\digamma} \dsymmw\in\Psisc^{0,0}$,
  but \eqref{eq:weighted-IBP} shows that it is in $\Psisc^{-1,0}$, and
  then the appearance of $\dot\gamma$ as mentioned means that the
  principal symbol has the same vanishing at $\xi=0$, since the same
  integration by parts is possible. This shows that the
analogue of Proposition~\ref{prop:N-digamma-structure-basic} holds (with an $N\times N$ matrix of operators, each with the
same structure as in that proposition), which gives the claimed mapping property just as in the
case of $N_\digamma$.

Moreover, if the weight is close to the
identity in the $\mathcal{A}^k$ norm for $k$ sufficiently large, then $N_\digamma\otimes I_N-\tilde N_\digamma$ is {\em small} as an
operator between these Hilbert spaces, and $N_\digamma\otimes
I_N:\cX^N\to\cY^N$ has a left inverse $N_\digamma^{-1}\otimes
I_N$. Correspondingly,
$$
\tilde N_\digamma^{-1}= (\Id+(N_{\digamma}^{-1}\otimes\Id_N) (\tilde N_\digamma-(N_\digamma\otimes\Id_N)))^{-1}(N_\digamma^{-1}\otimes\Id_N),
$$
is the desired left inverse.
\end{proof}

\section{Boundary rigidity}\label{sec:rigidity}

\subsection{Preliminaries}\label{sec_7.1}
Before proceeding with boundary rigidity, we recall from \cite{LassasSU} that if the
boundary distance functions of two metrics $g$, $\hat g$ are the
same on an open set $U_0$ of $\pa M$ and $\pa M$ is strictly
convex with respect to these metrics (indeed, convexity suffices), then for any compact subset $K$ of $U_0$ there is a diffeomorphism of $M$ fixing $\pa M$ such that the pull back of $\hat g$ by this diffeomorphism agrees with $g$ to infinite order at $K$. For a more
general result not requiring convexity, see \cite{SU-lens}. 
Concretely, the local statement is:

\begin{lemma} [\cite{LassasSU}] 
Let   $\bo$ be convex at $p_0$ with respect to $g$ and $\hat g$. Let $d=\tilde d$ on $\bo\times\bo$  near $(p_0,p_0)$. 
Then there exists a local diffeomorphism $\psi$ of a neighborhood of $p_0$ in $M$ to another such neighborhood with $\psi=\Id$ on $\bo$ near $p_0$ so that  
$\partial^\alpha g=\partial^\alpha(\psi^* \hat g)$ on $\bo$ near $p_0$ for every  multiindex $\alpha$. 
\end{lemma} 

The diffeomorphism $\psi$ is constructed by identifying the semigeodesic coordinates, also called boundary normal coordinates, for both metrics. More specifically, let $z'=(z^1,\dots,z^{n-1})$ be local coordinates on $\bo$ near $p_0$, and let for a moment denote by $\gamma_{z',\nu}(s)$ the unit speed geodesic in the metric $g$ with initial point $p=p(z')\in\bo$ and direction the unit outward normal $\nu$ at $p$. Then $\phi: z=(z',z^n)\mapsto \gamma_{z',\nu}(z^n)$ is a local diffeomorphism, and then $z$ are local coordinates near $p_0$. Then $\phi^*g$ is $g$ in the normal gauge to $\bo$ and it satisfies $(\phi^*g)_{in}=\delta_{in}$, $i=1, \dots,n$ and $\bo$ is given locally by $z^n=0$. The distance function restricted to $\bo\times\bo$ near $(p_0,p_0)$ recovers the full jet of $\phi^*g$ at $\bo$ near $p_0$ uniquely. 
Let $\hat\phi$ be the diffeomorphism related to $\hat g$. Then $\psi:= \hat\phi\circ \phi ^{-1}$ is the diffeomorphism in the lemma above. In the (common) coordinates $z$, they both satisfy $g_{in}=\hat g_{in}=\delta_{in}$; more precisely, 
$(\phi^*g)_{in}=(\hat \phi^*\hat g)_{in}=\delta_{in}$, see, e.g., \cite[sec.~4.1]{Sh-book}. In other words, they are both in the normal gauge. 

The local statement of the lemma immediately implies the
semiglobal statement we made above it, namely the existence of a {\em
  single} diffeomorphism $\psi$ for compact subsets $K$ of
$U_0\subset\pa M$ such that  
$\partial^\alpha g=\partial^\alpha(\psi^* \hat g)$ on $K$ for every  multiindex $\alpha$.

We simply denote the pullback $\psi^*\hat g$ by $\hat g$, i.e.\ we
assume, as we may, that $g$ and $\hat g$ agree to infinite order on $K$. Applying this with an open smooth subdomain 
$U_1\ni p_0$ of $\pa M$ with $ \bar{U}_1\subset U_0$ compact, we can then extend $g$ and $\tilde
g$ to a neighborhood of $M$ in the ambient
manifold without boundary $\tilde M$ so that the extensions are
identical in a neighborhood $O_1$ of $U_1$; from this point on we work in
such a neighborhood of $U_1$.

Recall also that the above linear
results {\em in the normal gauge} required that the metric itself,
whose geodesics we consider, is in
the normal gauge. So for the non-linear problem we proceed as follows. First, we are given a smooth function $\foliation$ with $ d\foliation\not=0$ and strictly concave level sets from the
side of its superlevel sets at least near the 0-level set $H$, assume that the zero level set only intersects $M$ at $p_0\in \pa M$, then $\{\foliation\geq
-c\}\cap M$ is compact for $c>0$ small.  
A unique point of contact with $\pa M$ can be achieved, as in \cite{UV:local}, if we chose the concavity of $H$ to be strictly greater than that of $\pa M$ at $p_0$. Then  $\{\foliation\geq -c\}\cap M$ becomes small when $0<c\ll1$ and converges to $p_0$ as $c\to 0+$.

In fact, only the zero level set of the function $\foliation$ near $p_0$ will be relevant for local boundary rigidity. Thus, the open set $U_0$ above
is a neighborhood of $p_0$ in  $\pa M$, and the open set on which the metric
is recovered will be a neighborhood of $p_0$ in $M$, see also Figure~\ref{fig:local_lens_rigidity_pic1}. 

Namely, using
$H=\{\foliation=0\}$ as the initial hypersurface (rather than $\bo$ as above), we put the metrics $g$,
$\hat g$ into normal coordinate form relative to $H$ in a
neighborhood of $p_0$. In other words,  we pull each one back by a diffeomorphism
fixing $H$, so, dropping the diffeomorphism from the notation (as it
will not be important from now on), they are of the form $g=d\tilde x^2+h(\tilde x,y,dy)$,
$\hat g=d\tilde x^2+\tilde h(\tilde x,y,dy)$, and correspondingly the dual metrics
are of the form $g^{-1}=\pa_{\tilde x}^2+h^{-1}(\tilde x,y,\pa_y)$,
$\hat g^{-1}=\pa_{\tilde x}^2+\tilde h^{-1}(\tilde x,y,\pa_y)$. Note that those diffeomorphisms, constructed by identifying semigeodesic coordinates normal to $H$ map $\bo$ (near $p_0$) to the same hypersurface (pointwise) which we still call $\bo$ since the two metrics are equal outside $M$. 
It is with the so obtained $\tilde x$ that we apply our linear normal
gauge result; note that as $\{\tilde x=0\}=H$,  and $\{\tilde x\geq -c\}\cap M$ is
small when $c\ll1$, we still have the
concavity (as well as the other) assumptions satisfied for the level
sets $\{\tilde x=-c\}$
when $c$ is small. In addition, $g-\hat g$, as well as
$g^{-1}-\hat g^{-1}$, have support whose intersection with $O_1$ is
a subset of $M$.

\subsection{Pseudolinearization}
Our normal gauge result then plugs into the
pseudolinearization formula based on the following identity which appeared in \cite{SU-MRL}, see also \cite{SUV_localrigidity}.  
Let $V$, $\tilde V$ be two vector fields on a  manifold  $M$ which will be replaced later with $T^*M$. Denote by $\X(s,\X\zero)$ the solution of $\dot \X=V(\X)$, $\X(0)=\X\zero$, and we use the same notation for $\tilde V$ with the  
corresponding solution are denoted by $\tilde \X$.

\begin{lemma}\label{SU-identity}
For any $t>0$ and any  initial condition $\X\zero$, if $\tilde \X\!\left(\cdot,\X\zero\right)$ and $\X\!\left(\cdot,\X\zero\right) $ exist on the interval $[0,t]$, then 
\[
\tilde \X\!\left(t,\X\zero\right)-\X\!\left(t,\X\zero\right) = \int_0^t \frac{\partial \tilde \X}{\partial \X\zero}\!\left(t-s,\X(s,\X\zero)\right)\left(\tilde V-V \right)\!\left( \X(s,\X\zero)\right)\,ds.
\]
\end{lemma} 

The proof is based on the application of the Fundamental Theorem of Calculus to the function
\[
F(s) = \tilde \X\!\left(t-s,\X(s,\X\zero)\right), \quad 0\le s\le t. 
\] 
Let $g$, $\hat g$ be two metrics.   The corresponding Hamiltonians and Hamiltonian vector fields are  
\be{HV}
H = \frac12 g^{ij}\xi_i\xi_j , \qquad 
V = \left(g^{-1} \xi, -\frac12\partial_\x |\xi|_{g}^2\right),
\ee
and the same ones related to $\hat g$. Here,  $|\xi|_{g}^2:= g^{ij}\xi_i\xi_j $.

In what follows,  we denote points in the phase space $T^*M$, in a fixed coordinate system, by $z = (\x,\xi)$. We denote the bicharacteristic with initial point $z$ by $Z(t,z) = (\X(t,z),\Xi(t,z))$. 

Then we obtain the identity already used  in \cite{SU-MRL, SUV_localrigidity}:

\be{1}
\tilde Z(t,z) - Z(t,z)= 
\int_0^t \frac{\partial\tilde Z}{\partial z} (t-s,Z 
(s,z))\big(\tilde V - V\big)(Z (s,z))\,ds.
\ee
 We can naturally think of the scattering relation $\cL$ and the travel
 time $\ell$ as functions on the cotangent bundle instead of the
 tangent one, which yields the following.

\begin{proposition}\label{pr1} 
Assume  
\be{5a}
\cL(x_0,\xi^0) = \tilde \cL(x_0,\xi^0), \quad \ell(x_0,\xi^0) =\tilde \ell(x_0,\xi^0)
\ee
for some $z_0= (x_0,\xi^0)\in \partial_-S^*M$. Then 
\be{5v}
\int_0^{ \ell(z_0) }  \frac{\partial\tilde Z}{\partial z} ( \ell(z_0)-s,Z 
(s,z_0) )\big(V -\tilde V\big)(Z (s,z_0))\,ds =0
\ee
with $V$ as in \r{HV}. 
\end{proposition}

Recall from the introduction that the boundary distance function determines the lens data
locally, thus Proposition~\ref{pr1} is the geometric input of
Theorems~\ref{thm:local-impr}-\ref{thm:local-pr} establishing the
connection between the given geometric data and a transform (which
depends on $g$ and $\hat g$) of
$V-\tilde V$, namely \eqref{5v}.

\subsubsection{Linearization near   $g$ Euclidean}   As a simple exercise,
we first consider the special case of the Euclidean metric to develop a 
feel for this identity. So let  $g_{ij}=\delta_{ij}$ and linearize for $\hat g$ near $g$ first under the assumption  $\hat g_{ij}=\delta_{ij}$ outside an open region $\Omega\subset \R^n$. Then
\begin{equation*}  \label{110'}
Z(s,z) = \mat {I_n}{sI_n}0{I_n} z,\quad \frac{\partial Z(s,z)}{\partial z} = 
\mat {I_n}{sI_n}0 {I_n},
\end{equation*}
with $I_n$ being the identity $n\times n$ matrix, and  we get the following formal linearization  of \r{5v} 
\begin{equation} \label{112}
\int_0^t \left(f\xi-\frac12(t-s) \partial_\x f^{ij}\xi_i\xi_j ,\,
-\frac12\partial_{\x} f^{ij}\xi_i\xi_j \right)(\x+s\xi,\xi)\,\d s=0,
\end{equation}
for $t\gg1$ with  
\[
f^{ij}(\x):= \delta^{ij}(\x)-\hat g^{ij}(\x). 
\]
Equation \r{112} is obtained by replacing  $\partial \tilde Z/\partial z$ in \r{1} by $\partial  Z/\partial z$. 
The last $n$ components of  (\ref{112}) imply 
\begin{equation*} \label{112'}
\int  \partial_\x f^{ij}(\x+s\xi)\xi_i\xi_j \,\d s=0.
\end{equation*} 
We integrate over the whole line $s\in\R$ because the integrand vanishes outside the interval $[0, \ell(\x,\xi)]$. 
We can remove the derivative there and get that the X-ray transform $If$ of the tensor field $f$ vanishes. 
Now, assume that this holds for all $(\x,\xi)$. Then  $ f=\dsymm v$ for
some covector field $v$ vanishing at $\bo$.   This is a linearized
version of the statement that $\hat g$ is isometric to $g$ with a
diffeomorphism fixing $\bo$ pointwise. Even in this simple case we see
that we actually obtained at first that $I(\partial_\x f)=0$ rather than $If=0$ and needed to integrate.

\subsubsection{The general case} \label{sec_2.2}
We take the second $n$-dimensional component on \r{1}.
 We get, with $f=g^{-1}-\hat g^{-1}$,
\[
\begin{split}
&\int  \frac{\partial \tilde \Xi }{\partial \x}( \ell(z)-s,Z
(s,z))(f\xi) (Z(s,z))\,\d s\\
&-\frac12  \int  \frac{\partial \tilde \Xi}{\partial \xi}( \ell(z)-s,Z 
(s,z ))(  \partial_\x  f \xi\cdot\xi)   (Z(s,z))\,\d s=0
\end{split}
\]
for any $z\in \partial_-SM$ for which \r{5a} holds. 
As before, we integrate over  $s\in\R$ because the support of the integrand vanishes for $s\not\in [0, \ell(\x,\xi)]$ (for that, we extend the bicharacteristics formally outside so that they do not come back).

Introduce the exit times $\tau(\x,\xi)$ defined as the minimal (and the only) $t> 0$ so that $\X(t,\x,\xi)\in\bo$. They are well defined near $S_p\bo$, if $\bo$ is strictly convex at $p_0$. 
We have
\[
\frac{\partial \tilde Z}{\partial z}(\ell(z)-s,Z(s,z)) = \frac{\partial \tilde Z}{\partial z}(\tau(Z(s,z))). 
\]
Then we get, with $f^{kl}=g^{kl}-\hat g^{kl}$,
\be{3}
\begin{split}
J_i f(\gamma):= &\int \Big( A_i^j(\X(t),\Xi(t))(\partial_{\x^j}f^{kl})(\X(t)) \Xi_k(t) \Xi_l(t) \\
&\qquad + B_i (\X(t),\Xi(t))f^{kl}(\X(t))  \Xi_k(t) \Xi_l(t)  \Big)\d t=0
\end{split}
\ee
for any bicharacteristic  $\gamma = (\X(t),\Xi(t))$ related to the metric  $g$ in our set, where
\be{4}
A_i^j\left(\x,\xi\right) = -\frac12 \frac{\partial\tilde  \Xi_i}{\partial \xi_j}(\tau(\x,\xi),(\x,\xi)) , 
\quad 
B_i\left(\x,\xi\right) = \frac{\partial \tilde\Xi_i}{\partial
  \x^j}(\tau(\x,\xi),(\x,\xi))g^{jk}(\x) \xi_k. 
\ee

The exit time function $\tau(\x,\xi)$ (recall that we assume  strong
convexity) becomes singular at $(\x,\xi)\in T^*\bo$. More precisely,
the normal derivative with respect to $\x$ when $\xi$ is tangent to $\bo$ has a
square root type of singularity. This is yet another reason to extend
the metrics $g$ and $\hat g$ outside $M$, in an identical
manner.

Based on those arguments, we  push the boundary  away a bit, to
$\tilde x=\delta$ with some $\delta>0$. For $(\x,\xi)$ with $\x$ near
$p_0$, redefine $\tau(\x,\xi)$ to be the travel time from $(\x,\xi)$ to
$H_\delta=\{\tilde x=\delta\}$. Let $U_-\subset\partial_-SH_\delta$ be
the set of all points on  $H_\delta$ and incoming unit directions so
that the corresponding geodesic in the metric $g$ is close enough to
one tangent to $\bo$ at
$p_0$. Similarly, let $U_+$ be the set of such pairs with outgoing  directions. 
Redefine the scattering relation $\cL$ locally to act from $U_-$ to $U_+$, and redefine $\ell$ similarly, see Figure~\ref{fig:local_lens_rigidity_pic1}.  
Then under the assumptions of Theorems~\ref{thm:local-impr}-\ref{thm:local-pr},
$\cL=\tilde \cL$ and $\ell=\tilde \ell$ on $U_-$. We can apply the construction above by replacing $\partial_\pm SM$ locally by $U_\pm$.  Equalities \r{3}, \r{4} are preserved then.  The advantage we have now is that on $U_-$, the travel time $\tau$ is non-singular but its derivatives are still large when $\delta\ll1$. To deal with this, we need the following lemmas. 

\begin{figure}[h!] 
  \centering
  \includegraphics[scale=1,page=3]{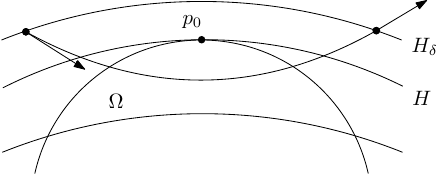}
\caption{The redefined scattering relation.}
  \label{fig:local_lens_rigidity_pic1}
\end{figure}

\begin{lemma}\label{lemma:tau}
For $|\lambda|\le C\delta$, $|x|\le\delta/2$, $y$ bounded, we have, for $0<\delta\ll1$,
\be{eq:tau}
\tau(x,y,\lambda,\omega  )=\sqrt{\delta-x}\,\tilde \tau\Big(\sqrt{\delta-x},y,\frac{\lambda}{\sqrt{\delta-x}},\omega\Big)
\ee
with some smooth function $\tilde \tau$. Moreover, $\tilde \tau$ depends continuously on $g\in C^k$ for $k\ge1$ under small perturbations of $g$. 
\end{lemma}

\begin{proof}
By \eqref{eq:edge-derivs}, 
ignoring the $\Gamma$ terms,   the bicharacteristic meets
$x=\delta$ when $x+\lambda t+\alpha t^2=\delta$, i.e.\ when
\begin{equation}\label{eq:quadratic-roots}
t=\frac{-\lambda\pm\sqrt{\lambda^2+4\alpha(\delta-x)}}{2\alpha};
\end{equation}
for the forward direction one needs to take the $+$ sign.
Note that for $\lambda=0$, this means
$t=\frac{1}{\sqrt{\alpha}}{\sqrt{\delta-x}}$. 
Now, \eqref{eq:quadratic-roots}, and its
$\lambda=0$ case, suggests that we should factor out $\sqrt{\delta-x}$
from the formula for $t$, which then (as $\alpha>0$ is bounded below
by a positive constant) suggests in turn defining
$$
\tilde\lambda= {\lambda}/{\sqrt{\delta-x}},\quad \tilde t= {t}/{\sqrt{\delta-x}},
$$
to get 
\[
\tilde t=\frac{-\tilde\lambda+\sqrt{\tilde\lambda^2+4\alpha}}{2\alpha},
\]
which is a smooth function of $\tilde\lambda$ (for which $|\tilde\lambda|\le C\sqrt\delta$) and
$\alpha=\alpha(x,y,\lambda,\omega)$ for $\tilde\lambda$ small. 
This then immediately suggests how to proceed in the general case, without
ignoring the $\Gamma$ terms. Namely, $x=\delta$ is reached when
$$
x-\delta+\lambda t+\alpha
t^2+t^3\Gamma^{(1)}(x,y,\lambda,\omega,t)
$$
vanishes. 
With  $\tilde\rho=\sqrt{\delta-x}$, this is
$$
\tilde\rho^2\Big(-1+\tilde\lambda\tilde t+\alpha\tilde
t^2+\tilde\rho\,\tilde
t^3\Gamma^{(1)}(x,y,\tilde\rho\,\tilde\lambda,\omega,\tilde\rho\,\tilde
t)\Big),
$$
and the vanishing is equivalent (in the relevant region) to that of
$$
h=-1+\tilde\lambda\tilde t+\alpha\tilde
t^2+\tilde\rho\,\tilde
t^3\Gamma^{(1)}(x,y,\tilde\rho\,\tilde\lambda,\omega,\tilde\rho\,\tilde
t).
$$
But $h$ vanishes when $\tilde\rho=0$, $\tilde\lambda=0$, $\tilde
t=\frac{1}{\sqrt{\alpha}}$, and it is a $\CI$ function of
$\tilde\rho,y,\tilde\lambda,\omega,\tilde t$, with $\pa_{\tilde t} h$
at these points  given by $2\sqrt{\alpha}\neq 0$. Hence the
implicit function theorem applies and shows that, for sufficiently
small $|\tilde\rho|$ and $|\tilde\lambda|$, say both being $<\tilde\delta$, $x=\delta$ is crossed at
$$
\tilde t=\tilde \tau(\tilde\rho,y,\tilde\lambda,\omega),
$$
where $\tilde \tau$ is $\CI$, and hence at $t=\tau$ as in \eqref{eq:tau}.  
Then the smallness requirements for $|\tilde\rho|$ and $|\tilde\lambda|$ are satisfied for $\lambda$, and $x$ as in the lemma as long as $\delta\ll1$. Finally, $\alpha$ and $\Gamma^{(1)}$ depend continuously of $g$ in the sense of the lemma, then so does $\tilde\tau$. 
\end{proof}

\begin{lemma} \label{lemma:A}
For every $k$, 
\be{7}
A_i^j(\x,\xi) =  -\frac12 \delta_i^j +O(\sqrt{\delta}), \quad B_i(\x,\xi)=O(1) 
\quad \text{in $C_{\rm sc}^k$} \ee
as $\delta\ll1$ for  $(\x,\xi)\in T^*M $ near $S^*_{p_0}\bo$ satisfying the smallness assumptions of Lemma~\ref{lemma:tau}. 
 Moreover, $A_i^j$ and $B_i$ with values in   $C_{\rm sc}^k$, depend continuously on $\hat g\in C^k$ for $k\ge1$ under small perturbations of $\hat g$. 
\end{lemma}
Recall that $C_{\rm sc}^k$ was defined after \eqref{eq:edge-derivs}. Estimate \r{7} is not true in general in the conventional $C^k$ norms. 

\begin{proof}
By Lemma~\ref{lemma:tau}, $\tau=O(\sqrt{\delta})$ in $C_{\rm sc}^k$.  Passing to the coordinates $x,y,\lambda, \omega$, we write
\[
A_i^j  =  -\frac12 \delta_i^j + \tau(x,y,\lambda, \omega)\tilde A_i^j(x,y,\lambda, \omega,  \tau(x,y,\lambda, \omega))
\]
with some smooth function $\tilde A_i^j(x,y,\lambda, \omega,  t)$ with derivatives uniformly bounded (and independent of $\delta$) in the region in Lemma~\ref{lemma:tau} and $|t|\ll1$. Then \r{7}  for $A_i^j$ follows by \r{4}. The proof for $B_i$ is similar. 
\end{proof}

\subsection{Local boundary rigidity. Proof of Theorem~\ref{thm:local-pr}} 
The equality of the distance functions $d_g$ and $d_{\hat g}$ for pairs of points on $\pa M$ close to a fixed one implies equality of the lens relations as redefined in the paragraph preceding Lemma~\ref{lemma:tau}, see also Figure~\ref{fig:local_lens_rigidity_pic1}. A priori, minimizing paths  may not be in a small neighborhood of ones tangent to $p_0$ but by shrinking $U$ in Theorem~\ref{thm:local-pr} if needed, we can arrange that they are. Note that the size of $U$ can be chosen uniform under small perturbations of  $g$ and $\hat g$ in $C^k$ with $k\gg1$.

Since in Section~\ref{sec:Fredholm} we analyzed the X-ray transform on
symmetric cotensors with 
weights, it is convenient to 
replace $f$ in \eqref{3} by its cotensor version. Thus, with $
f^{kl}=g^{kl}-\hat g^{kl}$, we have
\begin{equation*}\begin{aligned}
J_i f(\gamma):= &\int \Big(
A_i^j(\X(t),\Xi(t))g_{kr}(\X(t))g_{ls}(\X(t))(\partial_{x^j}f^{rs})(\X(t))\\
&\qquad\qquad\qquad\qquad\qquad\qquad g^{kr'}(\X(t))\Xi_{r'}(t) g^{ls'}(\X(t))\Xi_{ls'}(t) \\
&\qquad + B_i (\X(t),\Xi(t)) g_{kr}(\X(t))g_{ls}(\X(t)) f^{rs}(\X(t)) \\
&\qquad\qquad\qquad\qquad\qquad\qquad g^{kr'}(\X(t))\Xi_{r'}(t) g^{ls'}(\X(t))\Xi_{ls'}(t)  \Big)\d t=0,
\end{aligned}\end{equation*}
where
now $g^{-1}\Xi(t)$ in the arguments of $g_{kr}g_{ls}\partial_{x^j}f^{rs}$
and $g_{kr}g_{ls}f^{rs}$ is the tangent vector of the
geodesic (projected bicharacteristic) at $\X(t)$.
The equality is true for every $z$ for which $\ell(z)=\tilde\ell(z)$ near $S^*_{p_0}\pa M$.

In order to fit into the framework of Section~\ref{sec:Fredholm},
we further want to consider this as a transform on the $n+1$ functions $(f_j)_{ik}=g_{ir}g_{ks}\pa_j(g^{rs}-\hat g^{rs})$,
$(f_0)_{ik}=g_{ir}g_{ks}(g^{rs}-\hat g^{rs})$; thus ultimately the
transform we consider is
\begin{equation}\begin{aligned}\label{eq:tilde-I-defn}
\tilde I_i(\beta)(f_0,f_1,\ldots,f_n) 
&=\int_{\gamma_\beta}A^j_i(\X(t),\Xi(t))
f_j(\X(t))(X'(t),X'(t))\\
&\qquad\qquad+B_i(\X(t),\Xi(t)) f_0(\X(t))(X'(t),X'(t)))\,dt,
\end{aligned}\end{equation}
where $\gamma_\beta$ is the geodesic through $\beta\in S^*X$. Moreover, for every $k$, 
by Lemma~\ref{lemma:A},  $-2A^j_i$ is $O(\delta^{1/2})$ 
close to $\delta^j_i$  in $C_{\rm sc}^k$, if 
 the initial points and directions are $\delta$ close to $T_{p_0}\pa M$. 
Thus, considering the resulting transform $\tilde
N_\digamma$ on the $n$ components $u'=(u_1,\ldots,u_n)$, with
$u=e^{-\digamma/x} f$, we get, as in \cite{SUV:Tensor}, in this case
using Theorem~\ref{thm:X-ray-weights}, that there is $c_0>0$ such that
for $0<c<c_0$,
\begin{equation}\label{eq:X-ray-weight-mod-0}
\|u'\|_{\cX^n}\leq C(\|\tilde N_\digamma u'\|_{}+\|u_0\|_{\cX});
\end{equation}
here $\cX^n$ is the $n$-fold product space based on $\cX$ (i.e.\ each
$u_j\in\cX$, $j=1,\ldots,n$, and is estimated in that space).
We note that here $c_0$ and $C$ can be taken to be independent of $g$ as long
$g$ is $C^k$-close to a background metric (satisfying the assumptions) for suitable $k$.
We also need that

\begin{lemma}\label{lemma:microlocal-small-error}
Suppose $\tilde\delta>0$.
There exists $c_0>0$ such that for $0<c<c_0$, $\|u_0\|_{\cX}\leq
\tilde\delta\|u'\|_{\cX^n}$. Furthermore, $c_0$ can be taken to be
independent of $g$ as long as, for suitable $k$, $g$ is $C^k$-close to
a background metric $g_0$ satisfying our assumptions.
\end{lemma}

\begin{proof}
Recall that $(u_j)_{ik}=e^{-\digamma/x}
g_{ir}g_{ks}\pa_j(g^{rs}-\hat g^{rs})$,
$(u_0)_{jk}=e^{-\digamma/x}g_{ir}g_{ks}(g^{rs}-\tilde
g^{rs})$, i.e.\ $u_j=(g\otimes g)e^{-\digamma/x}\pa_j(g^{-1}-\tilde
g^{-1})$, $u_0=(g\otimes g)e^{-\digamma/x}(g^{-1}-\tilde
g^{-1})$. Thus, $u_j=(g\otimes g)e^{-\digamma/x}\pa_j
e^{\digamma/x}(g^{-1}\otimes g^{-1})u_0$.
Writing the first $n-1$ coordinates as the $y$ variables and
the $n$th as the $x$ variable, the result is proved if we can show
that $\|u_0\|_{\cX}\leq
\tilde\delta\|u_n\|_{\cX}$ when $c$ is suitably small.

Now
$$
-ix^2 u_n =(g\otimes g)e^{-\digamma/x}(x^2D_x)
e^{\digamma/x}(g^{-1}\otimes g^{-1})u_0,
$$
and $(g\otimes g)e^{-\digamma/x}(x^2D_x)
e^{\digamma/x}(g^{-1}\otimes g^{-1})$ has principal symbol
$\xi+i\digamma$ times the identity. By
Proposition~\ref{prop:not-quite-real-princ-scalar} we have
\begin{equation}\label{eq:rad-pt-based-for-Poincare}
\|u_0\|_{s,r}\leq C(\|x^2 u_n\|_{s,r}+\|u_0\|_{-N,-M}),
\end{equation}
with $C$ uniform in $g$ in the sense of the statement of the lemma.

We take $s>0$, $N>0$ as
we may, and note that the error term on the right hand side satisfies
$$
\|u_0\|_{-N,-M}\leq\|u_0\|_{0,-M}=\|x^{r+M} u_0\|_{0,r}\leq c_0^{r+M}\|u_0\|_{0,r}\leq c_0^{r+M}\|u_0\|_{s,r}
$$
if $u$ is supported in $x<c_0$. Substituting into
\eqref{eq:rad-pt-based-for-Poincare}, this can be absorbed into the
left hand side of the same equation for sufficiently small
$c_0$, which is uniform in $g$; for instance $Cc_0^{r+M}<1/2$ suffices.

Therefore, if $u$ is supported in
$x<c_0$, we deduce that $\|u_0\|_{s,r}\leq 2c_0^2C\| u_n\|_{s,r}$.

Now, recall from \eqref{eq:cX-def} that the $\cX$ spaces are just
spaces where similar estimates are made also for $x^2 D_x u$ and $(x^2
D_x)^2 u$ (more precisely, of the microlocal projections of $u$),
so this proves
the lemma.
\end{proof}

Then as in \cite{SUV:Tensor}, for $\tilde\delta>0$ sufficiently small,
one can absorb the $u_0$ term from the right hand side of \eqref{eq:X-ray-weight-mod-0}
in the left hand side. This proves the stable recovery of $u$, thus
$f$, from the transform, and thus local boundary rigidity: restricted
to $\tilde
x\geq -c$, the metrics are the same.

This concludes the proof of Theorem~\ref{thm:local-pr}.

\subsection{Semiglobal and global lens rigidity. Proof of Theorem~\ref{thm:global} and Theorem~\ref{thm:semiglobal}} \label{sec_Sg} 
Our approach also allows us to prove a global rigidity result. The key
point for this is to make the local
boundary rigidity argument uniform in how far from an initial
hypersurface $H$ the metrics $g$ and $\hat g$ can be shown to be identical in
geodesic normal coordinates.

We note that the normal gauge relative to a hypersurface provides a local 
diffeomorphism at a
uniform distance to it if one has a uniform estimate for the second
fundamental form of the hypersurface and of the curvature of the
manifold. We do it by proving differential injectivity first at a uniform distance, i.e.\
giving a lower bound for the flow parameter
for non-zero Jacobi fields to vanish. This follows from
comparison geometry (essentially the Rauch comparison theorem), namely comparing the ODE for Jacobi fields to that of the constant curvature case, when it is explicitly
solvable. To prove that this map is a diffeomorphism to its image,  notice that the geodesic flow from the unit
normal bundle of the hypersurface is globally well defined (if $\tilde
M$ is complete, as one may assume).  The question is if it is injective. 
For points a fixed distance apart, geodesics cannot intersect in
short times and there is a uniform lower bound 
and that bound depends  on the second fundamental form and on the curvature.  Concretely:

\begin{lemma}\label{lemma:normal-coord-tube}\ 

(a) 
Suppose $H$ is an embedded hypersurface in a Riemannian manifold without boundary
$(\tilde M,g)$, the sectional curvature of $g$ is $\leq \mu$, $\mu>0$, and suppose that the second fundamental form $\text{\rm II}$ of $H$
satisfies $|\text{\rm II}|\leq K$. Then the normal geodesic exponential map is a local
diffeomorphism on the $\frac{1}{\sqrt{\mu}}\cot^{-1} K$ (two sided) collar neighborhood of 
$H$, and the there is a uniform bound for the differential of the
local inverse on collars of strictly smaller radii.

(b)  Moreover, if $H$ is a compact subset of $H_\level=\{\foliation=\level\}$ with $d\foliation\not=0$ on $H_\level$, there exists $\delta_0>0$ depending on $\foliation$, $g$ and uniform under small perturbations of $\level$ so that the normal geodesic exponential map is a (global) 
diffeomorphism on the $\ep_0$ collar neighborhood of $H$; and hypersurfaces $\text{\rm dist}(\cdot,H_\level)=s$ are strictly convex for $|s|\le\ep_0$   under a small  perturbation of $\level$ and $g$. 
\end{lemma}

\begin{proof}
We use the result of
\cite[Theorem~4.5.1]{Jost:Riemannian}, which shows that if $J$ is a
Jacobi field, $\mu>0$, and $f_\mu=|J(0)|\cos(\sqrt{\mu}
t)+|J|^\cdot(0)\sin(\sqrt{\mu}t)$ and $f_\mu(t)>0$ for $0<t<\tau$ then
$f_\mu(t)\leq |J(t)|$ for $0\leq t\leq\tau$; here $\cdot$ denotes
derivatives in $t$. In particular, if
$J(0)\neq 0$, the first zero of $J(t)$ cannot happen before the first
zero of $f_\mu$, at which $|\cot(\sqrt{\mu}
t)|=\frac{||J|^\cdot(0)|}{|J(0)|}$, i.e.\
$|t|=\frac{1}{\sqrt{\mu}}\cot^{-1}\frac{||J|^\cdot(0)|}{|J(0)|}$.

Furthermore,
the discussion of \cite[Section~4.6]{Jost:Riemannian}, which is
directly stated for the distance spheres from a point, more generally
applies to geodesic normal coordinates to a submanifold. Thus, using
the computation following Equation~(4.6.12), considering a Jacobi
field arising from varying the initial point in $H$ of the normal
geodesic along a curve in $H$, one has $\dot J(0)=S(J(0),N)$ where $N$
is the unit normal vector to $H$, where $S$ is the second fundamental
form considered as a map $T_p H\times N_p H\to T_p H$, with $N_p H$
denoting the normal bundle.

Now, $(|J|^2)^\cdot=2|J||J|^\cdot$ (where $J\neq 0$), but also
$(|J|^2)^\cdot=2\langle \dot J,J\rangle$, so
$\frac{|J|^\cdot(0)}{|J(0)|}=\frac{1}{|J(0)|^2}\langle \dot
J(0),J(0)\rangle$.
Substituting in the above expression for $\dot J(0)$, we have
$$
\frac{|J|^\cdot(0)}{|J(0)|}=\frac{1}{|J(0)|^2}\langle
S(J(0),N),J(0)\rangle=\frac{1}{|J(0)|^2} \text{\rm II}(J(0),J(0))
$$
since $\text{\rm II}$ is related to $S$ by
$\text{\rm II}(X,Y)=\langle S(X,N),Y\rangle$. Correspondingly, with the assumed
bound on $\text{\rm II}$, we have 
$\frac{||J|^\cdot(0)|}{|J(0)|}\leq K$ and thus, as $\cot$ is
decreasing on $(0,\pi/2]$, so its inverse is such on $[0,\infty)$,
$|t|=\frac{1}{\sqrt{\mu}}\cot^{-1}\frac{||J|^\cdot(0)|}{|J(0)|}\geq
\frac{1}{\sqrt{\mu}}\cot^{-1}K$.

Hence the normal geodesic exponential map is a local diffeomorphism up
to distance $\frac{1}{\sqrt{\mu}}\cot^{-1}K$ from $H$.

One has a uniform bound for the differential of the inverse map if one
obtains a uniform bound for $|J(t)|$; this is provided for by the
explicit bound involving $f_\mu$ above for a strictly smaller collar.

To prove the second statement, notice first that we can find $c_0>0$ so that if $p,q\in H$ with $\text{dist}_H(p,q)<c_0$, then $p$ and $q$ have  distinct images under the normal exponential map $\psi$; and $c_0$ depends on $K$ and $\mu$ only. The complement $\mathcal{K}$ of such pairs  is compact and 
$\text{dist}(p,q)>1/C_0$ there with $C_0>0$ depending on $H$, $g$,  $K$ and $\mu$ but the latter two depend on $\foliation$ and $g$.  Then such $p$ and $q$ would have distinct images under $\psi$ if the latter is limited to $\text{dist}(\cdot,H)\le \ep_0<1/(2C_0)$. 
 Under a small perturbation of $\level$ and $g$, the constant $c_0$ can be chosen uniform, and then by a perturbation argument for $\text{dist}_H(p,q)\ge c_0$, $p$ and $q$  have distinct images if $\ep_0<1/(4C_0)$. The strong convexity statement follows from the fact that we can perturb the strict inequality $\text{\rm II}>0$ on a compact set. 
\end{proof}

\begin{proof}[Proof of Theorem~\ref{thm:global}] 
As before, since we can recover all derivatives of the metric at $\bo$ in boundary normal coordinates \cite{LassasSU, SU-lens}, 
we may assume that $M$ is a domain in $\tilde M$, and $g$ and $\hat g$ are defined on $\tilde M$, identically equal outside $M$.
It is convenient to work with open sets $\cU_0$, $\cU_1$ in $\tilde M$ with $\overline{\cU_0}$ compact and $M\subset\cU_0\subset \overline {\cU_0}\subset \cU_1$ with $\foliation$ smoothly extended to $\cU_1$ so that the
concavity and the condition $\{\foliation\geq 0\}\cap M\subset\pa M$ hold for this extension, and so that all derivatives of $\foliation$
are bounded. Notice that either $g$ or $\hat g$ geodesics cannot reach the complement of $\cU_1$ from $\overline{\cU_0}$ before a
uniformly bounded time, namely the geodesic distance between these two disjoint
sets, one of which is compact, and the other closed.

We prove below that there is a  
diffeomorphism $\psi:M\to M$ (defined on a larger region in
$\tilde M$ as a diffeomorphism), fixing $\pa M$ pointwise so that $g=\psi^* \hat g$. 
 We do it step by step (by ``layer stripping'') by going down along the level sets of $\foliation$. At each step, 
the corresponding foliation surface plays the role of $\pa M$ above, and the advance further, we can take small a bit less 
convex surfaces near each point as we did in Section~\ref{sec_Sg}. 
The proof actually show that $\psi$ is a diffeomorphism from $M$ to its image but the a priori assumption that $M$ is connected easily implies that the $\psi$ is surjective, as well.

We start with preliminary observations.
By Lemma~\ref{lemma:normal-coord-tube} (b), 
applied to $g$,
there is a uniform (independent of $\level$) constant $\ep_0>0$ such
that $g$-geodesic normal
coordinates around $H=H_\level=\{\foliation=-\level\}$ are valid on the
$\ep_0$-collar neighborhood, i.e.\ for an open subset $V$ of $H_\level$ containing 
$\overline{\cU_0}\cap H_\level$, the $g$-normal geodesic exponential map 
$\phi: V\times(-\ep_0,\ep_0)\to \tilde M$ is a diffeomorphism onto
its image. By reducing $\ep_0$ if needed, we may assume that the image
is included in $\cU_1$. 
Similarly, by Lemma~\ref{lemma:normal-coord-tube}~(a), 
there is a uniform (independent of
$\level$ as well as $\psi$) constant $\hat\ep_0>0$ such
that for {\em any} diffeomorphism $\psi$ such that $\psi^*\hat g=g$
on one side of $\hat H_\level = \psi(H_\level)$,  
the $\hat g$-normal exponential map 
$\hat \phi: \hat  V \times(-\hat \ep_0,\hat\ep_0) \to \tilde M$ is a \textit{local} diffeomorphism onto its image included in $\cU_1$. 
  It can be made global, i.e., injective for $\hat \ep_0\ll1$ but a priori, 
we do not know that this $\hat \ep_0$ can be chosen uniform, i.e., independent of $\hat H_\level$ 
to achieve the latter because Lemma~\ref{lemma:normal-coord-tube}~(b) requires control over $\psi^*g$  uniformity (on both sides of $\hat H$), and we do not have such a control yet. A priori, $H_\level$ may have points $p,q$ arbitrary close to each other in $M$ even if $K$ is fixed, and $\text{dist}_{H_\level}(p,q)>1/C$. This would reduce the maximal $\hat\ep_0$ we can choose if we want the $\hat g$-normal exponential map to be a diffeomorphism there. 
  For that reason, we work  in $\hat V\times(-\hat\ep_0,\hat\ep_0)$ as an intermediate manifold 
for now, instead of working on its image under the $\hat g$-normal exponential map, see Figure~\ref{fig:loc-bdy-semiglobal_pic2}.

By shrinking $\ep_0$ or $\hat\ep_0$ if necessary, we can assume that they are equal and will denote it by $\ep$.   We denote the
$g$-signed distance function (corresponding to the normal coordinates
around $H_\level$)
by $\tilde x=\tilde x_\level$ as above.

Note also that for $\tilde\delta>0$
there exists
$\delta_0>0$ such that for all $\level$, $\{0\geq \tilde x_\level\geq
-\tilde\delta\}\cap\overline{\cU_0}$ contains $\{-\level\geq\foliation\geq
-\level-\delta_0\}\cap\overline{\cU_0}$; notice that by the
compactness of $\overline{\cU_0}$, $\foliation$ is bounded on
$\overline{\cU_0}$ so we only need to consider a compact set of $\level$'s. But this is
straightforward, for if this does not hold, then there exists a
sequence $\level_j$ and points $p_j\in \overline{\cU_0}$ such that $\tilde
x_{\level_j}(p_j)< -\tilde\delta$ but $-\level_j\geq\foliation(p_j)\geq
-\level_j-1/j$. We may now extract a subsequence indexed by $j_k$ such
that $\level_{j_k}$ as well as $p_{j_k}$ converge to $\level$,
resp.\ $p$; then $\foliation(p)=-\level$ on the one hand, but $\tilde x_{\level}(p)\leq-\tilde\delta$, so $p\notin
H_\level=\{\foliation=-\level\}$, on the other, giving a
contradiction. Hence the desired $\delta_0>0$ exists.

\begin{figure}[!ht] 
  \centering
  \includegraphics[scale=0.62,page=2]{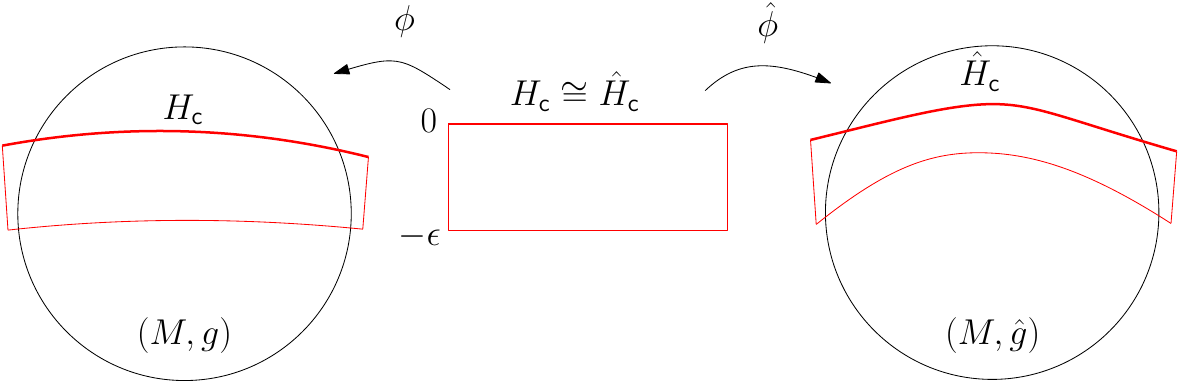}
\caption{The incremental step in the proof of the global rigidity. In
  the middle, $H_\level\times[-\ep,0]$ is shown. A priori,
  $(H_\level,g)$ and  $(\hat H_\level,\hat g)$ are isometric and equal
  outside $M$. The identification between $H_\level$ and $\hat
  H_\level$ is $\psi|_{H_\level}$.}
  \label{fig:loc-bdy-semiglobal_pic2}
\end{figure}

As the next observation, suppose that we have a diffeomorphism
$\psi:U\to\psi(U)\subset\cU_1$, where $U$ is a neighborhood of
$\foliation^{-1}([-\level,\infty))\cap\overline{\cU_0}$, such that
$\psi^*\hat g$ and $g$ agree in
$\{\foliation\geq-\level\}$. This means that if $\phi$ is the
$g$-normal geodesic exponential map around $H=\{\foliation=-\level\}$ (more precisely
around a neighborhood of $H\cap\,\overline{\cU_0}$), then
$\phi^*g=\phi^*\psi^*\hat g$ in $\tilde x=\tilde x_\level\geq 0$, and now both metrics are of the form
$d\tilde x^2+h(\tilde x,y,dy)$ on $V_y\times[0,\ep)_{\tilde x}$, i.e.\ $(\psi\circ\phi\circ( \psi|_{H_\level}^{-1}\times\id))^{-1}$ gives geodesic normal coordinates
for $\hat g$ around $\psi(H\cap\overline{\cU_0})$, at least in
$\tilde x\geq 0$ (here $\psi|_{H_\level}^{-1}$ enters to identify
$\psi(H\cap\overline{\cU_0})$ and $H\cap\overline{\cU_0}$, and in
$\psi|_{H_\level}^{-1}\times\id$, $\id$ is the identity
map on $(-\ep,\ep)$), and thus
is the {\em same} as $\hat\phi^{-1}$ in $\tilde x\geq 0$ (where we use the
notation $\tilde x$ for the first factor variable both for $V\times (-\ep,\ep)$ and
$\hat V\times(-\ep,\ep)$). Since we have a
uniform (independent of $\level$) bound of the collar neighborhood of the geodesic normal coordinates as long
as the second fundamental form, which is diffeomorphism
invariant, is bounded, and is determined from $\tilde x\geq 0$, thus the same as
that of $g$ at $H$, the normal geodesic exponential map gives a uniform
extension of $\psi$, via $\hat\phi\circ
(\psi|_{H_\level}\times\id)\circ\phi^{-1}$, to $\tilde x\geq
-\ep$ (note that by the above remarks the map $\hat\phi\circ
(\psi|_{H_\level}\times\id)\circ\phi^{-1}$ {\em is} $\psi$ in
$\tilde x\geq 0$, so we really have an extension); we continue to denote this by $\psi$. Notice that if $\check\phi$ is the $\psi^* \hat g$-normal
exponential map on $H$ (instead of that of $\hat g$ on $\hat H$, which
is $\hat\phi$), then $\psi\circ\check\phi=\hat\phi\circ(\psi|_H\times\id)$. As explained above, the so extended $\psi$ is
 a local diffeomorphism to its image by construction but a priori, 
we do not know if it is global (i.e.\ if it is injective) due to the appearance of $\hat\phi$ in its definition.
If  $g$ and $\psi^*\hat g$ have the same lens data at $H$, then
$\phi^*g$ and $\phi^* \psi^*\hat
g=(\psi|_{H_\level}\times\id)^*\hat\phi^*\hat g$ have the same data on
$V\times \{0\}$, and are in the normal gauge, i.e.\ are
tangential-tangential tensors plus $d\tilde x^2$. Then the pseudolinearization formula holds, and by \eqref{eq:X-ray-weight-mod-0} and
Lemma~\ref{lemma:microlocal-small-error}, they are the same
within a uniform (independent of $\level$: this uses that in the
semi-product coordinates the metric
depends continuously on $\level$) $\ep$-collar neighborhood around it, or more precisely around
$V$ as above, in respective geodesic normal coordinates, i.e.\
$\phi^*g = \phi^* \psi^* \hat g$   in $V\times (-\ep,\ep)$. 
We show below that $\hat\phi$ is a global (vs.\  just local) diffeomorphism. 
Then  this says exactly that the extension of
$\psi$ which we just gave is indeed an isometry between these two
metrics: $g=\psi^*\hat g$ in $\tilde x\geq -\ep$.

We prove that $\hat\phi$ is a global diffeomorphism from  $\hat V
\times(- \ep,\ep)$  to its image based on two arguments: (1) if it is
not, there should be a hypersurface $S_t: = \hat\phi\big( \hat V
\times \{t\}\big)$  with one piece of it tangent to another one; and
(2) this cannot happen because those pieces are strictly convex and
are touching each other from their concave sides. Below we denote the
variable on $(-\ep,\ep)$ by $t$ (rather than $\tilde x$). Indeed, assume that  there exist pairs of points  $(y_i,t_i)$, $t_i<0$, $y_i\in\hat V$, $i=1,2$ with the same image in $M$ under $\hat\phi$, with $t_1$ and $t_2$ in $[-\ep,0]$. If the set of such pairs is non-empty, we can always restrict $t$ to a slightly smaller closed interval, and $y$ to a compact subset of $\hat V$, and then there, the pairs with the same image would form a compact set. Let $t_0$ be the maximal  value $t_0$ for $\min(t_1,t_2)$. We can assume $t_1=t_0$. 
Then $t_2\ge t_1$ and $t_1$ is the maximal value with  that property. If this inequality is strict, since 
$\hat\phi(y_1,t_1) = \hat\phi(y_2,t_2)$,  we can perturb
$t_1$ and increase it slightly to $t_1'$ {\em and} find a new point
$(y_2',t_2')$ near $(y_2,t_2)$ by the inverse function
theorem (as $\hat\phi$ is a local diffeomorphism) with
$\hat\phi(y'_1,t'_1) = \hat\phi(y_2,t'_2)$ (and $t_1'>t_1=t_0$ still).
This would contradict the maximality property of $t_0$ because $t_1'$ would be a new candidate for it. 
Therefore, $t_1=t_2=t_0$. By the maximality property, $S_{t_0}$ near
$(y_1,t_1)$ (meaning the image $S^{(1)}$ of a neighborhood of $(y_1,t_0)$ under
$\hat\phi$) is tangent to its piece $S^{(2)}$ near $(y_2,t_2)$ (in the same
image sense), which proves (1).  Then $S^{(1)}$ and $S^{(2)}$ have common tangent vectors at $q:= \hat\phi(y_1,t_1) = \hat\phi(y_2,t_2)$, and opposite outer unit normals (along which $t$, say,  decreases, which determines an orientation for each one of them). 
Any geodesic starting from that point in a fixed tangential direction
would stay on the concave side of each piece,  which corresponds to
$\hat\phi\big(\hat V\times (t_0,\ep)\big)$, for a  sufficiently short time. If
$\hat\phi_j$ are the localized $\hat\phi$ near $(y_j,t_j)$, $j=1,2$,
so that they are actually invertible,
then on any such geodesic $\gamma$, the first component of $\hat\phi_j^{-1}$
(which is just the localized signed distance to $\hat V$) will
increase as it leaves $q$. That leads to a contradiction because that means existence
of points (namely $\gamma(s)$ for small $s\neq 0$) with two preimages with $t_j>t_0$.   Therefore, $\hat\phi$ is a global diffeomorphism as stated. Then so is $\hat\phi\circ(\psi|_{H_\level}\times\id)\circ\phi^{-1}$ above and the extended $\psi$ is a diffeomorphism, as well.

Finally, in the step described in the previous paragraph, one cannot encounter the boundary in $(M,\hat g)$ without
encountering it in $(M,g)$, i.e.\ if $\psi(p)\in\pa M$ for some $p\in
M$, with $\psi$ the extended map of the previous paragraph then
$p\in\pa M$, {\em provided} that this property already held for the original map $\psi$ of
that paragraph. Indeed, the lens relations of $(M,g)$ and $(\tilde
M,g)$ being the same plus $\psi$ being a diffeomorphism in a
neighborhood of
$\foliation^{-1}([-\level,\infty))\cap\overline{\cU_0}$, shows that
if for the extended $\psi$ we have $\psi(p)\in\pa M$, then taking in the
normal coordinates a constant-$y$ (normal to $\psi(H_\level)$!)
geodesic segment through $p$, within the range of the $\hat g$-geodesic normal
coordinate map $\tilde\phi$, it will go through a point $q$ in $\psi(H_\level)$. But the
equality of lens relations shows that the $g$-geodesic through $\psi^{-1}(q)$ (again, normal to $H_\level$) will then
also hit $\pa M$ in the range of the $g$-geodesic normal coordinate map $\phi$ since
the two lens relations are the same, and since in
$\foliation^{-1}([-\level,\infty))\cap\overline{\cU_0}$ the metrics
are already the same (thus lens data connecting $H_\level$, resp.\
$\psi(H_\level)$, to $\pa M$, are the same). Correspondingly,
$\psi|_M$ actually maps into $M$, for $\pa M$ separates the interior
of $M$ from $\tilde M\setminus M$. Finally, on the ``illuminated'' part of $\bo$, where $ d\foliation $ makes an acute angle with the outer conormal at $\bo$,  $\psi$ is identity. On the ``un-illuminated'' part of $\bo$ this is still true because the lens relations are the same.

Now we turn to the actual proof.
Let
\begin{equation*}\begin{aligned}
S=\{\level\geq 0:\ &\exists\psi:U\to \psi(U)\subset \cU_1\ \text{diffeo},\
\psi|_{\pa M\cap U}=\id,\\
& U\ \text{neighborhood of}\
\foliation^{-1}([-\level,+\infty))\cap\overline{\cU_0},\\
& \psi^*\hat g|_{\foliation^{-1}([-\level,+\infty))}=g|_{\foliation^{-1}([-\level,+\infty)}\}.
\end{aligned}\end{equation*}
Then $0\in S$ by hypothesis, with $\psi$ the identity map.
By the discussion of the paragraph above,
if $\level\in S$, the $\psi$ that exists by definition of $\level\in
S$ can be extended to a neighborhood of $H_\level\cap\overline{\cU_0}$
so that
$\psi^*\hat g$ and $g$ agree near $H=H_\level$, namely
in $\tilde
x>-c$, $c>0$. Taking into account the observations above, this
means that $\psi$ is defined in $\foliation>-\level-\delta_0$ for some
$\delta_0>0$. Thus, the set $S$ is
open, as $[0,\level+\delta_0)\subset S$.
Finally $S$ is also closed since by the discussion of the paragraph
above, if $\level\in S$, the $\psi$ that exists by definition of $\level\in
S$ can be extended to a {\em uniform} ($\level$-independent) neighborhood of $H_\level\cap\overline{\cU_0}$
so that
$\psi^*\hat g$ and $g$ agree near $H=H_\level$, namely
in $\tilde
x>-c$, $c>0$. The observation above shows then that $g$ and $\psi^*\hat g$ are the
same in $\foliation\geq -\level-\delta_0$, with $\delta_0>0$
independent of $\level$, proving that $S$ is closed (if $\level\notin
S$, $\level_j\in S$, $\level_j\to\level$, then take $j$ such that
$\level_j>\level-\delta_0$ to obtain a contradiction), and thus the theorem.
\end{proof}

Note that the function $\foliation$ need not satisfy the properties
globally on $M$;
in this case a completely analogous argument implies
that if in $\foliation>-T$ the assumptions of the theorem hold, then the
conclusions hold on $\foliation\geq -t$, $t<T$. Moreover, the 0 level set
condition may be replaced by an arbitrary level set (if needed, shift
$\foliation$ by a constant).

Thus, for instance, if $\foliation$ is the
distance function from a point in $M^\circ$, this gives that under the
hypotheses of the theorem, which hold if $g$ has no focal
points, for any $\ep>0$, in $\foliation\geq\ep$, $\hat g$ is the pullback of
$g$ by a diffeomorphism. In particular, this proves Theorem~\ref{thm:semiglobal}.

\section{The foliation condition and corollaries} \label{AppendixB}
The assumption of an existence of a strictly convex function appears also in some works on Carleman estimates, see, e.g., \cite{TriggianiY} and the references there. Existence of such a function is also assumed in the recent work \cite{PSUZ} on integral geometry. We will connect such functions with our foliation condition below. 

A $C^2$ function $f$ on $M$ is called strictly convex on some set, if $\textrm{Hess}\, f>0$ as a  form on that set, where $\textrm{Hess}$ is the Riemannian Hessian defined through covariant derivatives. Such a function can have at most one critical point which is a local minimum. It was shown in \cite{PSUZ} that if the foliation condition   holds with $\{\foliation=0\}=\bo$, then there exists a strictly convex function $f$ in $M$. We will show that the converse is true, which is actually an easier statement to prove. 

\begin{lemma}
Let $f$ be a strictly convex function on $(M,g)$ near a non-critical point $\x=\x_0$. Then the level hypersurfaces $f(\x)=c$ are strictly convex near $x_0$ when viewed from $f>c$. 
\end{lemma}

\begin{proof}
We have
\be{B1}
\frac{d^2}{dt^2} f(\gamma(t)) = \text{Hess}\, (f)(\dot\gamma,\dot\gamma)) \ge c_0>0 
\ee
for any geodesic $\gamma$ as long as $\gamma(t)$ is close to $\x_0$ where $f$ is strictly convex. We can always assume $f(\x_0)=0$; we will prove strict convexity of $S:= \{f=0\}$ near $\x_0$.   Take $\gamma(0)=\x_0$, with $\dot\gamma(0)$ tangent to the level set $f(\x)=0$. Then 
\begin{equation}\label{fg}
f(\gamma(t))\ge  (c_0/4)t^2 \quad \text{for $|t|\ll1$} 
\end{equation}
for any such tangent geodesic through $\x_0$. 
Since $f$ is a defining function of $S$, (\ref{fg}) implies strict convexity of the latter. Indeed, \r{B1} when $(d/dt)f(\gamma(t))=0$ at $t=0$ is preserved for any other defining function $\tilde f$ of $S$ preserving the orientation, which is easy to check, since $\tilde f=fh$ with $h>0$ on $S$. If we take $\tilde f$ to be the signed distance to $S$, positive on $\{f>c\}$, \r{B1} becomes just the second fundamental form of $S$, up to a positive multiplier. 
\end{proof}

Therefore,  existence of a strictly convex function implies our foliation condition away from the possibly unique  critical point (when $M$ is connected).  In particular, if the sectional curvature is positive or negative,  our foliation condition is satisfied on $M\setminus\{x_0\}$ by  \cite[section~2]{PSUZ} , where $x_0$ is the critical point, if exists; otherwise, on $M$. 

We show next that existence of a critical point of $f$ still allows us to prove global lens rigidity. 

\begin{thm}\label{thm:convex-critical-pt}
Let $(M,g)$ be a compact $n$-dimensional Riemannian manifold, $n\ge3$, with a strictly convex boundary so that there exists a strictly convex function $f$ on $M$ with $\{f=0\}=\bo$. Let $\hat  g$ be another Riemannian metric on $g$, and assume that $\bo$ is strictly convex w.r.t. $\hat g$ as well. If $g$ and $\hat g$ have the same lens relations, then there exists a diffeomorphism $\psi$ on $M$ fixing $\bo$ pointwise such that $g=\psi^*\hat g$. 
\end{thm}

\begin{proof}
The interesting case we have not covered  so far is when $f$ can have a critical point, $x_0$, in (the interior of) $M$, which is also the minimum of $f$ in $M$. For $0<\epsilon\ll1$, let $M_0 = \{\x|\; f(\x) \le f(\x_0)+\epsilon\}$. If $\ep\ll1$, then $M_0$ can be covered by a single chart and it is diffeomorphic to a closed ball. By the semiglobal Theorem~\ref{thm:semiglobal},  $(\overline{M\setminus M_0},g)$ is isometric to $(\overline{M\setminus \hat M_0}, \hat g)$, with some compact connected $\hat M_0$ with smooth boundary in the interior of $M$, 
 and the diffeomorphism realizing the isometry fixes $\bo$ pointwise. If $\epsilon\ll1$, then $M_0$ is simple and it can  be foliated by strictly convex surfaces without a critical point in its closure, for example by the Euclidean spheres centered at a point a bit away from its boundary. Then by our global Theorem~\ref{thm:global}, $(M_0,g)$ and $(\hat  M_0,\hat  g)$ are isometric. Since one can perturb $\epsilon$ a bit, the diffeomorphism from outside can be extended a bit inside. On the other hand,   if two metrics are isometric near the boundary, with a diffeomorphism fixing the latter, that diffeomorphism is determined uniquely near  the boundary by identifying boundary normal coordinates. Therefore, the two diffeomorphisms coincide in the overlapping region. 
\end{proof}

This result implies Corollary~\ref{cor_B} of the introduction:

\begin{proof}[Proof of Corollary~\ref{cor_B}]
The proof follows directly from \cite{PSUZ}, where it is shown that under either of those conditions, there exists a smooth strictly convex function $\foliation$ with $\{\foliation=0\}=\bo$. 
\end{proof}

Finally, we give some sufficient conditions for the foliation condition to hold. As shown in \cite{SUV_localrigidity, SUV_elastic}, for metrics $c^{-2}d x^2$ in a domain in $\R^n$, the generalized  Herglotz \cite{Herglotz} and Wiechert and Zoeppritz \cite{WZ}  condition $\partial_r(r/c(r\omega))>0$, where $r,\omega$ are polar coordinates (compare to \r{Her}), is equivalent to the requirement that the Euclidean spheres $|x|=C$ are strictly convex in the metric $c^{-2}d x^2$. If $M$ is given locally by $x^n>0$, if $\partial_{x^n}c>0$, then the hyperplanes $x^n=C\ge0$ form a strictly convex foliation.  Then our results prove rigidity  for such metrics in the class of all metrics, not necessarily conformal to the Euclidean.

\appendix

\section{An improvement of a lemma from \cite{SUV:Tensor}.}

We need a new version of Lemma~4.13 of \cite{SUV:Tensor} which is
lossless in terms of decay in order to apply the perturbation
argument above in Section~\ref{sec:Fredholm} culminating in the proof
of Theorem~\ref{thm:X-ray-weights}, namely that the X-ray transform of $g$ with  
weights (as opposed to the standard X-ray transform) is invertible, in
the sense of a left inverse on $\Omega$,
when the weight is close to the identity. Recall that this lemma gives
an estimate of $u$ in terms of $\dsymmw u$ on
$\Omega_1\setminus\Omega$ for $u$ vanishing at $\pa_\inter\Omega_1$,
but not necessarily at $\pa_\inter\Omega$ (i.e.\ $\pa
M\cap\Omega$). The loss of the lemma is in the decay at $\pa X$, which
we now fix.
In order to obtain this improved version, we first prove a similar lemma for the
symmetric gradient of a scattering metric.

\subsection{A lossless estimate for scattering metrics}
Thus, we consider
scattering metrics of the form $g_{\scl}=\frac{dx^2}{x^4}+\frac{h}{x^2}$ with respect to a product
decomposition of a neighborhood of the boundary $x=0$, where $h$ is a
metric on the boundary: $h=h(y,dy)$, and let $\dsymmsc$ be the
symmetric gradient of $g_{\scl}$, and let
$$
\dsymmscw=e^{-\digamma/x}\dsymmsc
e^{\digamma/x}:\Hsc^{s,r}(X;\Tsc^*X)\to\Hsc^{s-1,r}(X;\Sym^2\Tsc^*X).
$$
Then we have a lossless estimate for expressing $u$ in terms of
$\dsymmscw u$:

\begin{lemma}\label{lemma:local-ds-inverse-forms-improved-sc}
Let $\Hscd^{1,0}(\Omega_1\setminus\Omega)$ be as in
Lemma~4.12 of \cite{SUV:Tensor}, but with values in one-forms, and let
$\rho_{\Omega_1\setminus\Omega}$ be a defining function of
$\pa_\inter\Omega$ as a boundary of $\Omega_1\setminus\Omega$, i.e.\
it is positive in the latter set. 
Suppose that
$\pa_x\rho_{\Omega_1\setminus\Omega}> 0$ at $\pa_\inter\Omega$ (with
$\pa_x$ understood with respect to the product decomposition); note
that this is independent of the choice of
$\rho_{\Omega_1\setminus\Omega}$ satisfying the previous criteria (so this is a statement on $x$ being
increasing as one leaves $\Omega$ at $\pa_\inter\Omega$). Then there
exists $\digamma_0>0$ such that for $\digamma\geq\digamma_0$, on one-forms
the map
$$
\dsymmscw:\Hscd^{1,r}(\Omega_1\setminus\Omega)\to \Hsc^{0,r}(\Omega_1\setminus\Omega)
$$
is injective, with a continuous left inverse
$P_{\scl,\Omega_1\setminus\Omega}:\Hsc^{0,r}(\Omega_1\setminus\Omega)\to
\Hscd^{1,r}(\Omega_1\setminus\Omega)$.

Moreover, for $\digamma\geq \digamma_0$, the norms of
$\digamma P_{\scl,\Omega_1\setminus\Omega}:\Hsc^{0,r}(\Omega_1\setminus\Omega)\to
\Hsc^{0,r}(\Omega_1\setminus\Omega)$, $P_{\scl,\Omega_1\setminus\Omega}:\Hsc^{0,r}(\Omega_1\setminus\Omega)\to
\Hsc^{1,r}(\Omega_1\setminus\Omega)$ are uniformly bounded.
\end{lemma}

\begin{proof}
In \cite[Proof of Lemma~4.13]{SUV:Tensor} the following formula from
\cite[Chapter~3.3]{Sh-book} played a key role:
\begin{equation}\label{eq:int-along-curves}
\sum_i [v(\gamma(s))]_i\dot\gamma^i(s)=\int_0^s \sum_{ij}[\dsymmsc v(\gamma(t))]_{ij}\dot\gamma^i(t)\dot\gamma^j(t)\,dt,
\end{equation}
where $\gamma$ is a unit speed geodesic of the metric
whose symmetric gradient we are considering (so the scattering metric
$g_\scl$ in the present case) with
$\gamma(0)\in\pa_\inter\Omega_1$ (so $v(\gamma(0))$
vanishes) and $\gamma(\tau)\in\pa_\inter\Omega\cup\pa X$, with
$\gamma|_{(0,\tau)}$ in $\Omega_1\setminus\overline{\Omega}$. 
The identity (\ref{eq:int-along-curves}) is just an application of the Fundamental Theorem of Calculus with the $s$-derivative of the l.h.s.\ computed using the rules of covariant differentiation. 
In this formula we use
$[\dsymmsc v(\gamma(t))]_{ij}$ for the components in the symmetric
2-cotensors corresponding to the standard cotangent
bundle, and similarly for $[v(\gamma(s))]_i$. Notice that this formula
gives an explicit left inverse for $\dsymmscw$.

Here we use the differential version of this (i.e.\ prior to an
application of the fundamental theorem of calculus):
\begin{equation*}\label{eq:diff-along-curves}
\frac{d}{ds}\sum_i [v(\gamma(s))]_i\dot\gamma^i(s)=\sum_{ij}[\dsymmsc v(\gamma(s))]_{ij}\dot\gamma^i(s)\dot\gamma^j(s),
\end{equation*}
and note that the left hand side is simply
$$
\dot\gamma(.)\Big(\sum_i [v(\gamma(.))]_i\dot\gamma^i(.)\Big)\Big|_s,
$$
with the first $\dot\gamma$ considered as a vector field
differentiating the function to which it is applied. Thus, taking any
smooth family of such geodesics emanating from
$\pa_\inter\Omega_1$, parameterized by $\pa_\inter\Omega_1$, and
letting their tangent vectors define a vector field $\sX$ on
$\Omega_1$, we have on $\Omega_1\setminus\Omega$:
$$
\sX\iota_\sX v=(\dsymmsc v)(\sX,\sX),
$$
which we consider a PDE for $\tilde u=\iota_\sX v$.
We can then proceed 
as in Lemma~4.12 of \cite{SUV:Tensor}. 

We first need to discuss the geometry.
For a general scattering metric, see \cite[Lemma~2]{RBMZw}, the limiting geodesics on $\pa X$
(which make sense directly as projected integral curves of the
rescaled Hamilton vector field $\scH_{g_\scl}$ on $\Tsc^*X$) are geodesics on $\pa X$
connecting distance $\pi$ points (i.e.\ they have length $\pi$). More
precisely, the projection of these integral curves in $\Tsc^*_{\pa
  X}X$ is either a single point, or a length $\pi$ $h$-geodesic. (Note
that in the case of Euclidean space this is simply the statement that
geodesics at infinity tend to antipodal points on the sphere at
infinity, and this remains true if the geodesics move uniformly to
infinity.) Since
our metric $g_{\scl}$ is homogeneous of degree $-2$ under dilations in $x$, the
analogous statement remains true for {\em all} geodesics, i.e.\ they
are either radial, so that $y$ is fixed along them, or their
projection to the $\pa X$ factor of $[0,\delta)_x\times\pa X_y$ is a
length $\pi$ geodesic of $h$, with appropriate behavior in $x$.

Now, recalling the basic b- and sc- objects (vector fields, bundles,
etc.) from Section~\ref{sec:normal-gauge}, the
tangent vector of a reparameterized (corresponding to the
renormalization of the Hamilton vector field, by a factor $x^{-1}$,
used to define $\scH_{g_\scl}$) geodesic, considered as a point in
$\Tb_{\gamma(s)} X$, is the pushforward of the rescaled
Hamilton vector field $\scH_{g_\scl}$ (which is a vector field on
$\Tsc^*X$ tangent to its boundary) under the bundle projection
$\Tsc^*X\to X$. The actual tangent vector to the geodesic is an
element of $\Tsc_{\gamma(s)} X$ (corresponding to reinserting the
$x$-factor). If coordinates on $\Tsc^*X$ are written as $(x,y,\xi,\eta)$,
corresponding to 1-forms being written as
$\xi\frac{dx}{x^2}+\sum_j\eta_j\frac{dy_j}{x}$, then the explicit formula for this pushed forward vector field is $\xi (x^2\pa_x)+\sum h^{ij}(y)\eta_i(x\pa_{y_j})$  modulo terms that push forward to $0$, 
see \cite[Equation~(8.17)]{RBMSpec}.  The second term is coming from
the Hamilton vector field of the dual boundary metric $h^{-1}$, and
$\xi^2+|\eta|^2_{h_y}=1$ by virtue of the geodesic flow being the
Hamilton flow on the unit cosphere bundle (a factor of 2 has been removed from the
vector field to make the
geodesics unit speed).

\begin{figure}[ht]
\begin{center}
\includegraphics[scale=0.9]{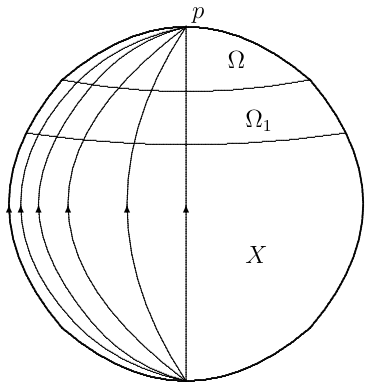}
\end{center}
\caption{Geodesics of $g_{\scl}$ tending towards the point $p$, 
  including the limiting boundary geodesic.}
\label{fig:sc-geod-1}
\end{figure}

For instance, as an illustration (we use a different family below for
the actual proof) take 
geodesics tending to a fixed point $p\in\pa X\cap\Omega^\circ$
(corresponding to a family of parallel lines in Euclidean space). They give
a family of geodesics we could consider below in many cases, e.g.\ if
we are working in a suitable small neighborhood of a point on $\pa M$ (see Figure~\ref{fig:sc-geod-1}). Then $-\xi$ (thus the $x^2\pa_x$
component of the tangent vector) is cosine of the distance from
$\gamma(s)$ to $p$ within (i.e.\ for the projection to) the $\pa X$ factor, while
$-\eta$ is the tangent vector of the $h$-geodesic given by the $\pa X$
projection times $(1-\xi^2)^{1/2}$ (i.e.\ sine of the distance
within $\pa X$); see again \cite[Lemma~2]{RBMZw}. We consider cases when $\pa
X$ is large metrically but $\pa
X\cap\overline{\Omega_1}$ is small, so all points in $\pa
X\cap\overline{\Omega_1}$ are distance $<\tilde\ep<\pi/2$ distance from each
other; this is relevant because of the length $\pi$-behavior of the
projected geodesics and the appearance of sine and cosine above. In
this case, varying $p$, taking finitely many appropriate nearby choices
gives rise to geodesics whose tangent vectors span $\Tsc_qX$ for each
$q$ as is immediate from the above discussion.  For instance if $h$ is
the flat metric, the $\eta$ component is simply the unit vector (up to
sign) from the projection of $q$ to $\pa X$ to $p$ times the sine of
the distance, and the $-\xi$ component is, as always, the cosine of the
distance, so it is straightforward to arrange finitely many choices of $p$'s with spanning
geodesic tangent vectors.  In general for $\tilde\ep>0$ small, a similar
conclusion holds.

\begin{figure}[ht]
\begin{center}
\includegraphics[scale=0.9]{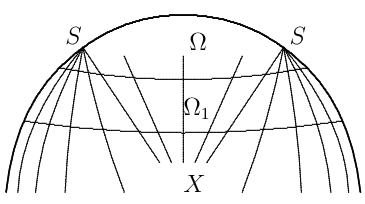}
\end{center}
\caption{Geodesics of $g_{\scl}$ tending towards the submanifold $S$
  (here shown as 2 points), with the family extended by radial 
  geodesics to cover $\pa_\inter\Omega^\circ$. For $n\geq 3$ (as is
  the case here), for a better illustration, the picture should be imagined rotationally symmetric
around the vertical axis through the middle of the figure, so the
indicated two points on $S$ are in the same rotation orbit.}
\label{fig:sc-geod-2}
\end{figure}

In fact, for the general considerations below (as opposed to certain
special cases), it is best to
take a codimension $1$ submanifold $S$ in $\pa X\cap\Omega$ near
$\pa_\inter\Omega$, namely a slight inward perturbation of
$\pa_\inter\Omega$, e.g.\ a short time flow by the $h$-normal
geodesics on $\pa X$ from $\pa_\inter\Omega$, and use a 1-dimensional family of
geodesics tending to each of the points on it locally near $\pa X$ (for a
total $(n-1)$-dimensional family). For example, one can  pick a vector field on $S$
close to the $h$-normal vector field of $S$, and use geodesics whose $\pa
X$-projection is a length $\pi$ $h$-geodesic with this given tangent
vector at the end point in $S$; see Figure~\ref{fig:sc-geod-2}. These form a one parameter family since the
normal to $\pa X$ component of the tangent vector is arbitrary (but we
will take it relatively small). Then the geodesics all intersect
$\pa_\inter\Omega$ close to their limiting point on $S$ (close e.g.\
in the sense that the affine parameter in the projection to $\pa X$,
when considered as a unit speed $h$-geodesic, is close to that on $S$,
i.e.\ the $h$-geodesic segment is short) and in
particular near $\pa X$.  Thus they  do so transversally, so the
derivative of $\rho_{\Omega_1\setminus\Omega}$ along the tangent
vector of the geodesics (when rescaled by $x^{-1}$) has a definite (negative) sign at $\pa_\inter\Omega$.  (The actual tangent
vector of the $g_{\scl}$-geodesic will give a derivative $\leq -Cx$,
$C>0$, corresponding to the $x\pa_y$-component of the pushforward of $\scH_{g_{\scl}}$.) One
can then smoothly combine this with geodesics crossing
$\pa_\inter\Omega$ farther away from $\pa X$ (e.g.\ specifying their
tangent vectors at $\pa_\inter\Omega$ smoothly extending the already specified
tangent vectors near $\pa X\cap\pa_\inter\Omega$) to obtain the full
$(n-1)$-dimensional family of geodesics in such a manner that,  when rescaled by $x^{-1}$,  the
derivative of $\rho_{\Omega_1\setminus\Omega}$ along the family has a
negative definite sign at
$\pa_\inter\Omega$. For instance, one can use radial geodesics or
their small perturbations (changing the direction at
$\pa_\inter\Omega^\circ$ slightly)  in the
extension, i.e.\ ones in which the $\pa X$ component is constant;
these behave as desired due to the assumption on $\pa_x\rho_{\Omega_1\setminus\Omega}$. We then eventually
take finitely many such families of geodesics as discussed above to
span the scattering tangent space (starting by varying the vector
field specified on $S$). Note that the latter is just the
standard tangent space away from $\pa X$, hence the usual
considerations apply there. On the other hand,  near $\pa X$ our previous discussion
applied to geodesics close to the initial point (now on $S$)  applies, with only the $h$-distance along the $\pa
X$-projections of {\em these} geodesics from the initial point to
$\pa_\inter\Omega_1$ required to be small  (so for any $h$, if
$\Omega_1$ is chosen so that $\Omega_1\setminus\Omega$ is small, the
construction works). (This contrasts with the discussion of the
previous paragraph, where
geodesics tending to a single fixed point $p$ were used, in which the
$h$-diameter of $\Omega_1$ had to be small.)

Now, to use these observations, first notice that as we consider geodesics of a scattering metric, $\sX\in\Diffsc^1$.
Thus, let $V=\frac{1}{i}\sX$, $P=e^{-\digamma/x}Ve^{\digamma/x}\in\Diffsc^1$ and
consider $\|P u\|^2$ again keeping in mind that we
need to be careful at $\pa_\inter\Omega$ since $u$ does not vanish
there (though it {\em does} vanish at $\pa_\inter\Omega_1$). Thus, there is an integration by parts boundary term, which we
express in terms of the characteristic function
$\chi_{\Omega_1\setminus\Omega}$:
\begin{equation*}\begin{aligned}
\|Pu\|^2_{L^2(\Omega_1\setminus\Omega)}&=\langle\chi_{\Omega_1\setminus\Omega}
Pu,Pu\rangle_{L^2(\Omega_1)}
=\langle P^*\chi_{\Omega_1\setminus\Omega}
Pu,u\rangle_{L^2(\Omega_1)}\\
&=\langle P^*Pu,u\rangle_{L^2(\Omega_1\setminus\Omega)}+\langle[P^*,\chi_{\Omega_1\setminus\Omega}] Pu,u\rangle_{L^2(\Omega_1)}.
\end{aligned}\end{equation*}
Writing $P=P_R+iP_I$ (as in Lemma~4.2 of \cite{SUV:Tensor}), $P_R=\frac{P+P^*}{2}$,
\begin{equation*}\begin{aligned}
\|P_R u\|^2_{L^2(\Omega_1\setminus\Omega)}=\langle P_R^*P_Ru,u\rangle_{L^2(\Omega_1\setminus\Omega)}+\langle[P_R^*,\chi_{\Omega_1\setminus\Omega}] P_Ru,u\rangle_{L^2(\Omega_1)}.
\end{aligned}\end{equation*}
On the other hand, with $P_I=\frac{P-P^*}{2i}$ being $0$-th order, the commutator term
vanishes for it. Correspondingly,
\begin{equation*}\begin{aligned}
\|Pu\|^2_{L^2(\Omega_1\setminus\Omega)}&=\langle
P^*Pu,u\rangle_{L^2(\Omega_1\setminus\Omega)}+\langle[P^*,\chi_{\Omega_1\setminus\Omega}]
Pu,u\rangle_{L^2(\Omega_1)}\\
&=\langle
P_R^*P_Ru,u\rangle_{L^2(\Omega_1\setminus\Omega)}+\langle
P_I^*P_Iu,u\rangle_{L^2(\Omega_1\setminus\Omega)}+\langle
i[P_R,P_I]u,u\rangle_{L^2(\Omega_1\setminus\Omega)}\\
&\qquad\qquad+\langle[P^*,\chi_{\Omega_1\setminus\Omega}]
Pu,u\rangle_{L^2(\Omega_1)}\\
&=\|P_Ru\|^2_{L^2(\Omega_1\setminus\Omega)}+\|P_Iu\|^2_{L^2(\Omega_1\setminus\Omega)}+\langle
i[P_R,P_I]u,u\rangle_{L^2(\Omega_1\setminus\Omega)}\\
&\qquad\qquad+\langle[P^*,\chi_{\Omega_1\setminus\Omega}]
Pu,u\rangle_{L^2(\Omega_1)}-\langle[P_R^*,\chi_{\Omega_1\setminus\Omega}]
P_Ru,u\rangle_{L^2(\Omega_1)}.
\end{aligned}\end{equation*}
Now, as $P-P_R$ is $0$-th order,
$[P^*,\chi_{\Omega_1\setminus\Omega}]=[P_R^*,\chi_{\Omega_1\setminus\Omega}]$,
so the last two terms on the right hand side give
\begin{equation}\label{eq:d-inj-annulus}
\langle[P^*,\chi_{\Omega_1\setminus\Omega}]
i P_Iu,u\rangle_{L^2(\Omega_1)}=\langle (\sX\chi_{\Omega_1\setminus\Omega})P_Iu,u\rangle_{L^2(\Omega_1)}.
\end{equation}
Now, $P=V-\digamma x^{-2}Vx$ with $V-V^*\in x\Diffsc^0$ (since it has
real principal symbol in the full scattering sense), and hence
$P_I=\digamma x^{-2}\sX x+a$, $a\in x\CI$. Thus, \eqref{eq:d-inj-annulus}
is non-negative, at least if $x$ is sufficiently small (or
$\digamma$ large) on $\pa_\inter\Omega$ since
$\chi_{\Omega_1\setminus\Omega}$ is
$\chi_{(0,\infty)}\circ\rho_{\Omega_1\setminus\Omega}$ times a similar
composite function of the defining function of $\pa_\inter\Omega_1$
(which however plays no role as $u$ vanishes there by assumption),
$\sX\rho_{\Omega_1\setminus\Omega}$ and
$\sX x$ can be arranged to be negative (i.e.\ $x$ decreasing
along the geodesics being considered) in the strong $\leq -Cx^{2}$
sense (with $C>0$). Correspondingly, this term
can be dropped. In addition, $[P_R,P_I]\in x\CI$, so the corresponding
term can be absorbed into the $\|P_I u\|^2$ terms, and one obtains
\begin{equation}\label{eq:ds-basic-L2-est-sc}
\|u\|_{L^2(\Omega_1\setminus\Omega)}\leq C\|Pu\|_{L^2(\Omega_1\setminus\Omega)},
\end{equation}
at
least if $x$ is small on $\Omega_1$ just as in the proof of
\cite[Lemma~4.2]{SUV:Tensor}.  (In fact, $\digamma$ large also works as
$[P_R,P_I]=O(\digamma)$, while $\|P_I u\|^2$ gives an upper bound for
$c^2\digamma^2\|u\|^2$ if $\digamma\geq\digamma_0$, $\digamma_0>0$
sufficiently large, see below for more detail.) This in turn gives with
$u=e^{-\digamma/x} \tilde u$,
$$
\|e^{-\digamma/x}\tilde u\|_{L^2(\Omega_1\setminus\Omega)}\leq C\|P
e^{-\digamma/x}\tilde
u\|_{L^2(\Omega_1\setminus\Omega)}=C\|e^{-\digamma/x}\sX\tilde u\|_{L^2(\Omega_1\setminus\Omega)}
$$
i.e.\ with $\tilde u=\iota_\sX v$, using $\sX\iota_\sX v=(\dsymmsc v)(\sX,\sX)$,
\begin{equation*}\begin{aligned}
&\|\iota_\sX(e^{-\digamma/x}
v)\|_{L^2(\Omega_1\setminus\Omega)}=\|e^{-\digamma/x}\iota_\sX
v\|_{L^2(\Omega_1\setminus\Omega)}\\
&\leq C\|e^{-\digamma/x}(\dsymmsc v)(\sX,\sX)\|_{L^2(\Omega_1\setminus\Omega)}=C\|\dsymmscw (e^{-\digamma/x} v)\|_{L^2(\Omega_1\setminus\Omega)}
\end{aligned}\end{equation*}
in this case.
The case of $x$ not
necessarily small on $\Omega_1$ (though small on $\Omega$) follows
exactly as in \cite[Lemma~4.13]{SUV:Tensor} discussed above, using the
standard Poincar\'e inequality, and even the case where $x$ is not
small on $\Omega$ can be handled similarly since one now has an extra
term at $\pa_\inter\Omega$, away from $x=0$, which one can control
using the standard Poincar\'e inequality. (Again, one can instead
simply take $\digamma$ sufficiently large.)

Taking a finite number of families of geodesics with tangent vectors
spanning $\Tsc^*X$ then gives, with $\tilde v=e^{-\digamma/x}
v$,
\begin{equation}\label{eq:local-Poincare-tensor}
\|\tilde v\|_{L^2(\Omega_1\setminus\Omega)}\leq C\|\dsymmscw \tilde v\|_{L^2(\Omega_1\setminus\Omega)}.
\end{equation}
To obtain the $H^1$ estimate, we use Lemma~4.5 of \cite{SUV:Tensor}. It is stated there for $\dsymmw$ (symmetric gradient with
respect to a standard metric) but it works equally well for
$\dsymmscw$ since it treats the $0$-th order term, by which
these symmetric gradients differ from that of a flat metric, as an
error term, which in both cases is a $0$-th order scattering differential
operator between the appropriate bundles; see below for more detail in
the large parameter discussion. This
gives, for $\tilde v\in\Hscb^{1,r}(\Omega_1\setminus\Omega)$,
$$
\|\tilde v\|^2_{\Hscb^{1,r}(\Omega_1\setminus\Omega)}\leq
C(\|\dsymmscw \tilde
v\|^2_{\Hsc^{0,r}(\Omega_1\setminus\Omega)}+\|\tilde v\|^2_{\Hsc^{0,r}(\Omega_1\setminus\Omega)}),
$$
which combined with \eqref{eq:local-Poincare-tensor} proves
$$
\|\tilde v\|_{\Hscd^{1,r}(\Omega_1\setminus\Omega)}\leq C\|\dsymmscw
\tilde v\|_{\Hsc^{0,r} (\Omega_1\setminus\Omega)},\qquad \tilde v\in\Hscd^{1,r}(\Omega_1\setminus\Omega),
$$
where recall that our notation is that membership of
$\Hscd^{1,r}(\Omega_1\setminus\Omega)$ only implies vanishing at
$\pa_\inter\Omega_1$, not at $\pa_\inter\Omega$.
In particular, this
shows the claimed injectivity of $\dsymmscw$. Further, this gives a continuous
inverse from the range of $\dsymmscw$, which is closed in
$L^2(\Omega_1\setminus\Omega)$; one can use an orthogonal projection
to this space to define the left inverse
$P_{\Omega_1\setminus\Omega}$, completing the proof when $k=0$.

For general $k$, one can proceed as in
\cite[Lemma~4.4]{SUV:Tensor}, conjugating
$\dsymmscw$ by $x^k$, which changes it by $x$ times a smooth
one form; this changes $P$ by an element of $x\CI(X)$,
with the only effect of modifying the $a$ term in
\eqref{eq:d-inj-annulus}, which does not affect the proof.

To see the final claim, observe that $\|P_I u\|^2\geq
c^2\digamma^2\|u\|^2_{L^2(\Omega_1\setminus\Omega)}$ for some $c>0$, when $\digamma\geq\digamma_0$,
and thus the estimate \eqref{eq:ds-basic-L2-est-sc} actually holds
with $c\digamma\|u\|_{L^2(\Omega_1\setminus\Omega)}$ on the
left hand side, which in turn gives
the estimate
\begin{equation*}\label{eq:local-Poincare-tensor-sc}
\digamma\|\tilde v\|_{L^2(\Omega_1\setminus\Omega)}\leq
C\|\dsymmscw \tilde v\|_{L^2(\Omega_1\setminus\Omega)}.
\end{equation*}
Finally, the proof of the modified Korn's inequality, Lemma~4.5 of
\cite{SUV:Tensor}, gives the estimate, for $u\in \Hscb^{1,r}(\Omega_1\setminus\Omega)$,
$$
\|u\|_{\Hscb^{1,r}(\Omega_1\setminus\Omega)}\leq C(\|\dsymmscw u\|_{\Hsc^{0,r}(\Omega_1\setminus\Omega)}+\digamma\|u\|_{\Hsc^{0,r}(\Omega_1\setminus\Omega)}). 
$$
Indeed the proof there has a direct estimate for the symmetric
gradient of the flat metric and then regards the $0$-th order terms, by
which a general symmetric gradient differs from this flat symmetric gradient as error terms to
be absorbed into the second term on the right hand side.  In our case
these $0$-th order terms have $C\digamma$ bounds (corresponding
to the exponential conjugation), so the conclusion follows, proving
the claim. Applying it in our setting we have
\begin{equation*}\label{eq:local-Poincare-tensor-sc-1}
\|\tilde v\|_{\Hsc^{1,0}(\Omega_1\setminus\Omega)}\leq
C\|\dsymmscw \tilde v\|_{L^2(\Omega_1\setminus\Omega)}.
\end{equation*}
Again, adding polynomial weights proceeds without difficulties.
\end{proof}

\subsection{The extension of the results to `standard' metrics}
Now, a straightforward calculation of the Christoffel symbols shows
that they do not contribute to
the full principal symbol of the gradient relative to $g_{\scl}$, in
$\Diffsc^1(X;\Tsc^*X;\Tsc^*X\otimes\Tsc^*X)$, and thus this
principal symbol
is, as a map from one-forms to 2-tensors (which we
write in the four block form as before) is
\begin{equation*}\label{eq:g-grad-symbol-sc}
\begin{pmatrix} \xi&0\\\eta\otimes&0\\0&\xi\\0&\eta\otimes\end{pmatrix},
\end{equation*}
and thus
that of $\dsymmsc$ in $\Diffsc^1(X;\Tsc^*X;\Sym^2\Tsc^*X)$ (with symmetric 2-tensors considered as a subspace of
2-tensors) is
$$
\begin{pmatrix} \xi&0\\\frac{1}{2}\eta\otimes&\frac{1}{2}\xi\\\frac{1}{2}\eta\otimes&\frac{1}{2}\xi\\0&\eta\otimes_s\end{pmatrix}.
$$
Thus the symbol of $\dsymmscw=e^{-\digamma/x}\dsymmsc e^{\digamma/x}$,
which conjugation effectively replaces $\xi$ by $\xi+i\digamma$ (as
$e^{-\digamma/x} x^2D_x e^{\digamma/x}=x^2D_x+i\digamma$), is
$$
\begin{pmatrix} \xi+i\digamma&0\\\frac{1}{2}\eta\otimes&\frac{1}{2}(\xi+i\digamma)\\\frac{1}{2}\eta\otimes&\frac{1}{2}(\xi+i\digamma)\\0&\eta\otimes_s\end{pmatrix}.
$$

It is useful to consider this as a semiclassical operator with
Planck's constant $h=\digamma^{-1}$, i.e.\ to analyze what happens
when $h$ is small, i.e.\ $\digamma$ is large. Thus, consider the
semiclassical operator $h\dsymmsc=\digamma^{-1}\dsymmsc$; its full
(i.e.\ at $h=0$, fiber infinity and base infinity all included)
semiclassical principal symbol (since it only depends on $\digamma$
via this explicit prefactor) is, writing $\xi_h=h\xi=\xi/\digamma$
and $\eta_h=h\eta=\eta/\digamma$ as the semiclassical variables
$$
\begin{pmatrix} \xi_h&0\\\frac{1}{2}\eta_h\otimes&\frac{1}{2}\xi_h\\\frac{1}{2}\eta_h\otimes&\frac{1}{2}\xi_h\\0&\eta_h\otimes_s\end{pmatrix}.
$$
Correspondingly, the full
(i.e.\ at $h=0$, fiber infinity and base infinity all included)
semiclassical principal symbol of $h\dsymmscw$ is
$$
\begin{pmatrix} \xi_h+i&0\\\frac{1}{2}\eta_h\otimes&\frac{1}{2}(\xi_h+i)\\\frac{1}{2}\eta_h\otimes&\frac{1}{2}(\xi_h+i)\\0&\eta_h\otimes_s\end{pmatrix}.
$$

On the other hand, the proof of Lemma~3.2 of \cite{SUV:Tensor} shows
that the full principal symbol of $\dsymm$, relative to a standard metric
$g$, in $\Diffsc^1(X;\Tsc^*X;\Sym^2\Tsc^*X)$ is
$$
\begin{pmatrix} \xi&0\\\frac{1}{2}\eta\otimes&\frac{1}{2}\xi\\\frac{1}{2}\eta\otimes&\frac{1}{2}\xi\\a&\eta\otimes_s\end{pmatrix},
$$
with $a$ a symmetric 2-tensor,
so the full semiclassical principal symbol of $h\dsymm=\digamma^{-1} \dsymm$ is
$$
\begin{pmatrix} \xi_h&0\\\frac{1}{2}\eta_h\otimes&\frac{1}{2}\xi_h\\\frac{1}{2}\eta_h\otimes&\frac{1}{2}\xi_h\\ha&\eta_h\otimes_s\end{pmatrix}.
$$
and thus that
of $h\dsymmw=e^{-\digamma/x} h\dsymm e^{\digamma/x}$ is
$$
\begin{pmatrix} \xi_h+i&0\\\frac{1}{2}\eta_h\otimes&\frac{1}{2}(\xi_h+i)\\\frac{1}{2}\eta_h\otimes&\frac{1}{2}(\xi_h+i)\\ha&\eta_h\otimes_s\end{pmatrix}.
$$

This proves that, with the subscript $h$ on $\Diffsc$ denoting
semiclassical operators,
\begin{equation}\label{eq:symm-gr-diff-def}
R=h\dsymmw-h\dsymmscw\in h\Diff_{\scl,h}^0(X;\Tsc^*X;\Sym^2\Tsc^*X).
\end{equation}

This allows us to prove the following sharp form of Lemma~4.13 of \cite{SUV:Tensor}:

\begin{lemma}\label{lemma:local-ds-inverse-forms-improved}
Let $\Hscd^{1,0}(\Omega_1\setminus\Omega)$ be as in
Lemma~4.12 of \cite{SUV:Tensor}, i.e.\ with dot implying vanishing at
$\pa_\inter\Omega_1$ only, but with values in one-forms, and let
$\rho_{\Omega_1\setminus\Omega}$ be a defining function of
$\pa_\inter\Omega$ as a boundary of $\Omega_1\setminus\Omega$, i.e.\
it is positive in the latter set. 
Suppose that
$\pa_x\rho_{\Omega_1\setminus\Omega}> 0$ at $\pa_\inter\Omega$ (with
$\pa_x$ defined relative to the product decomposition reflecting the
warped product structure of $g_\scl$); note
that this is independent of the choice of
$\rho_{\Omega_1\setminus\Omega}$ satisfying the previous criteria (so this is a statement on $x$ being
increasing as one leaves $\Omega$ at $\pa_\inter\Omega$). Then there
exists $\digamma_0>0$, such that for $\digamma\geq\digamma_0$,
the map
$$
\dsymmw:\Hscd^{1,r}(\Omega_1\setminus\Omega)\to \Hsc^{0,r}(\Omega_1\setminus\Omega)
$$
is injective, with a continuous left inverse $P_{\Omega_1\setminus\Omega}:\Hsc^{0,r}(\Omega_1\setminus\Omega)\to \Hscd^{1,r}(\Omega_1\setminus\Omega)$.
\end{lemma}

\begin{proof}
We let $P_{\scl;\Omega_1\setminus\Omega}$ be the
left inverse given in
Lemma~\ref{lemma:local-ds-inverse-forms-improved-sc}; then with $R$ as
in \eqref{eq:symm-gr-diff-def},
$$
P_{\scl;\Omega_1\setminus\Omega}\dsymmw=P_{\scl;\Omega_1\setminus\Omega}\dsymmscw+P_{\scl;\Omega_1\setminus\Omega}h^{-1}R=\Id+P_{\scl;\Omega_1\setminus\Omega}h^{-1}R.
$$
Now
$$
h^{-1}P_{\scl;\Omega_1\setminus\Omega}=\digamma
P_{\scl;\Omega_1\setminus\Omega}:
\Hsc^{0,r}(\Omega_1\setminus\Omega)\to
\Hsc^{0,r}(\Omega_1\setminus\Omega)
$$
and
$$
P_{\scl;\Omega_1\setminus\Omega}:
\Hsc^{0,r}(\Omega_1\setminus\Omega)\to
\Hscd^{1,r}(\Omega_1\setminus\Omega)
$$
are uniformly bounded in $\digamma\geq\digamma_0$ by
Lemma~\ref{lemma:local-ds-inverse-forms-improved-sc}, which means in
terms of semiclassical Sobolev spaces (recall that $\Hsc^{0,r}(\Omega_1\setminus\Omega)=H_{\scl,h}^{0,r}(\Omega_1\setminus\Omega)$) that
$$
h^{-1}P_{\scl;\Omega_1\setminus\Omega}=\digamma
P_{\scl;\Omega_1\setminus\Omega}:
\Hsc^{0,r}(\Omega_1\setminus\Omega)\to
\dot H_{\scl,h}^{1,r}(\Omega_1\setminus\Omega).
$$
On the other hand, $R\in
h\Diff_{\scl,h}^0(X;\Tsc^*X;\Sym^2\Tsc^*X)$ shows that $h^{-1}R$
bounded $\dot H_{\scl,h}^{1,r}(\Omega_1\setminus\Omega)\to
\dot H_{\scl,h}^{1,r}(\Omega_1\setminus\Omega)$. In combination,
$P_{\scl;\Omega_1\setminus\Omega}h^{-1}R=h(h^{-1}P_{\scl;\Omega_1\setminus\Omega})(h^{-1}R)$ is bounded by $Ch$ as a map $\dot H_{\scl,h}^{1,r} (\Omega_1\setminus\Omega)\to
\dot H_{\scl,h}^{1,r} (\Omega_1\setminus\Omega)$, and thus
$\Id+P_{\scl;\Omega_1\setminus\Omega}h^{-1}R$ is invertible for $h>0$
sufficiently small. Then
$$
P_{\Omega_1\setminus\Omega}=(\Id+P_{\scl;\Omega_1\setminus\Omega}h^{-1}R)^{-1}P_{\scl;\Omega_1\setminus\Omega}
$$
gives the desired left inverse for $\dsymmw$ with the bound
$$
h^{-1}P_{\Omega_1\setminus\Omega}: \Hsc^{0,r}(\Omega_1\setminus\Omega)\to
\dot H_{\scl,h}^{1,r}(\Omega_1\setminus\Omega),
$$
which in particular means for finite (sufficiently large) $\digamma$
that
$$
P_{\Omega_1\setminus\Omega}: \Hsc^{0,r}(\Omega_1\setminus\Omega)\to
\Hscd^{1,r}(\Omega_1\setminus\Omega)
$$
is bounded, proving the lemma.
\end{proof}


\begin{thebibliography}{88}



\bibitem{BCG}{G. Besson, G. Courtois, and S. Gallot}. 
Entropies et rigidit\'es des espaces localement sym\'etriques
de courbure strictment n\'egative.
\newblock
{\em Geom. Funct. Anal.}, \textbf{5} (1995), 731--799.

\bibitem{BI} D. Burago and S. Ivanov, Boundary rigidity and filling volume minimality of metrics close to a flat one. 
\newblock{\em Ann. Math.}, {\bf 171} (2010), 1183--1211.

\bibitem{Creager}
K.~C. Creager.
\newblock Anisotropy of the inner core from differential travel times of the
  phases {PKP} and {PKIPK}.
\newblock {\em Nature}, 356:309--414, 1992.

\bibitem{Croke90}
C.~B. Croke.
\newblock Rigidity for surfaces of nonpositive curvature.
\newblock {\em Comment. Math. Helv.}, 65(1):150--169, 1990.

\bibitem{Croke04}
C.~B. Croke.
\newblock Rigidity theorems in {R}iemannian geometry.
\newblock In {\em Geometric methods in inverse problems and PDE control},
  volume 137 of {\em IMA Vol. Math. Appl.}, pages 47--72. Springer, New York,
  2004.


\bibitem{Croke_scatteringrigidity}
C.~Croke.
\newblock Scattering rigidity with trapped geodesics.
\newblock {\em Ergodic Theory Dynam. Systems}, 34(3):826--836, 2014.

\bibitem{CDS}
C.~B. Croke, N.~S. Dairbekov, and V.~A. Sharafutdinov.
\newblock Local boundary rigidity of a compact {R}iemannian manifold with
  curvature bounded above.
\newblock {\em Trans. Amer. Math. Soc.}, 352(9):3937--3956, 2000.


\bibitem{Gr} M. Gromov. Filling Riemannian manifolds. \newblock {\em J. Diff. Geometry} \textbf{18} (1983),  1--148. 

\bibitem{Herglotz}
G.~Herglotz.
\newblock {\"U}ber die {E}lastizitaet der {E}rde bei {B}eruecksichtigung ihrer
  variablen {D}ichte.
\newblock {\em Zeitschr. f\"ur Math. Phys.}, 52:275--299, 1905.

\bibitem{Hor}
L.~H\"ormander.
\newblock {\em The Analysis of Linear Partial Differential Operators, {\rm vol.
  1-4}}.
\newblock Springer-Verlag, 1983.


\bibitem{Colin14}
C.~Guillarmou.
\newblock Lens rigidity for manifolds with hyperbolic trapped set.
\newblock {\em J. Amer. Math. Soc.}, 30:561--599, 2017.


\bibitem{Ikawa}
M.~Ikawa, editor.
\newblock {\em Spectral and scattering theory}.
\newblock Marcel Dekker, 1994.

\bibitem{I} S. Ivanov. Volume comparison via boundary distances. {\em Proceedings of the International Congress of Mathematicians, vol. II}, 769--784, New Delhi, 2010.

\bibitem{Jost:Riemannian}
\newblock J.~Jost.
\newblock {\em Riemannian Geometry and Geometric Analysis}.
\newblock Springer-Verlag, Berlin, 1998.


\bibitem{LassasSU}
M.~Lassas, V.~Sharafutdinov, and G.~Uhlmann.
\newblock Semiglobal boundary rigidity for {R}iemannian metrics.
\newblock {\em Math. Ann.}, 325(4):767--793, 2003.

\bibitem{Mazzeo:Edge}
R.~Mazzeo.
\newblock Elliptic theory of differential edge operators. {I}.
\newblock {\em Comm. Partial Differential Equations},
16(10):1615--1664 (1991).
 
\bibitem{Mu2}
R.~G. Muhometov.
\newblock On a problem of reconstructing {R}iemannian metrics.
\newblock {\em Sibirsk. Mat. Zh.}, 22(3):119--135, 237, 1981.



\bibitem{MuRo}
R.~G. Muhometov and V.~G. Romanov.
\newblock On the problem of finding an isotropic {R}iemannian metric in an
  {$n$}-dimensional space.
\newblock {\em Dokl. Akad. Nauk SSSR}, 243(1):41--44, 1978.


\bibitem{RBMSpec}
R.~B. Melrose.
\newblock {\em Spectral and scattering theory for the Laplacian on
  asymptotically Euclidian spaces}.
\newblock In Ikawa \cite{Ikawa}, 1994.

\bibitem{RBMZw}
R.~B. Melrose and M.~Zworski.
\newblock Scattering metrics and geodesic flow at infinity.
\newblock {\em Inventiones Mathematicae}, 124:389--436, 1996.

\bibitem{Michel}
R.~Michel.
\newblock Sur la rigidit\'e impos\'ee par la longueur des g\'eod\'esiques.
\newblock {\em Invent. Math.}, 65(1):71--83, 1981/82.

\bibitem{Otal90}
J-P.~Otal,
\newblock Sur les longueurs des g\'{e}od\'{e}siques d'une m\'{e}trique \`a courbure
              n\'{e}gative dans le disque.
\newblock  {\em Comment. Math. Helv.}, 65(2):334--347, 1990.

\bibitem{PSUZ} G.~Paternain, M.~Salo, G.~Uhlmann and H.~Zhou.
\newblock The geodesic X-ray transform with matrix weights. 
\newblock {\em Amer. J. Math.}, 141:1707--1750, 2019.

\bibitem{Pa}
{C. Parenti}, {Operatori pseudo-differenziali in {$R^{n}$} e applicazioni},
 \textit{Ann. Mat. Pura Appl.} (4), \textbf{93} (1972), 359--389.


\bibitem{PestovU}
L.~Pestov and G.~Uhlmann.
\newblock Two dimensional compact simple {R}iemannian manifolds are boundary
  distance rigid.
\newblock {\em Ann. of Math. (2)}, 161(2):1093--1110, 2005.

\bibitem{Sh-book}
V.~A. Sharafutdinov.
\newblock {\em Integral geometry of tensor fields}.
\newblock Inverse and Ill-posed Problems Series. VSP, Utrecht, 1994.


\bibitem{Sh}
{M. A. Shubin},
\textit{Pseudodifferential Operators in {$R^{n}$}},
{Dokl. Akad. Nauk SSSR},
\textbf{196} (1971), 316--319.


\bibitem{S-Serdica}
P.~Stefanov.
\newblock Microlocal approach to tensor tomography and boundary and lens
  rigidity.
\newblock {\em Serdica Math. J.}, 34(1):67--112, 2008.


\bibitem{SU-Duke}
P.~Stefanov and G.~Uhlmann.
\newblock Stability estimates for the {X}-ray transform of tensor fields and
  boundary rigidity.
\newblock {\em Duke Math. J.}, 123(3):445--467, 2004.


\bibitem{SU-JAMS}
P.~Stefanov and G.~Uhlmann.
\newblock Boundary rigidity and stability for generic simple metrics.
\newblock {\em J. Amer. Math. Soc.}, 18(4):975--1003, 2005.


\bibitem{SU-Kawai}
P.~Stefanov and G.~Uhlmann.
\newblock Boundary and lens rigidity, tensor tomography and analytic microlocal
  analysis.
\newblock In {\em Algebraic Analysis of Differential Equations}. Springer,
  2008.


\bibitem{SU-MRL}
P.~Stefanov and G.~Uhlmann.
\newblock Rigidity for metrics with the same lengths of geodesics.
\newblock {\em Math. Res. Lett.}, 5(1-2):83--96, 1998.



\bibitem{SU-lens}
P.~Stefanov and G.~Uhlmann.
\newblock Local lens rigidity with incomplete data for a class of non-simple
  {R}iemannian manifolds.
\newblock {\em J. Differential Geom.}, 82(2):383--409, 2009.


\bibitem{SUV_localrigidity}
P.~Stefanov, G.~Uhlmann, and A.~Vasy.
\newblock Boundary rigidity with partial data.
\newblock {\em J. Amer. Math. Soc.,} 29:299-332 (2016), {\em arxiv.1306.2995}.

\bibitem{SUV_elastic}
P.~Stefanov, G.~Uhlmann, and A.~Vasy.
\newblock Local recovery of the compressional and shear speeds from the hyperbolic DN map.
\newblock {\em Inverse Problems}, 34 014003, 2018.

\bibitem{SUV:Tensor}
P.~Stefanov, G.~Uhlmann, and A.~Vasy.
\newblock Inverting the local geodesic X-ray transform on tensors.
\newblock {\em Journal d'Analyse Math\'ematique}, 136: 151--208, 2018.

\bibitem{TriggianiY}
R.~Triggiani and P.~F. Yao.
\newblock Carleman estimates with no lower-order terms for general {R}iemann
  wave equations. {G}lobal uniqueness and observability in one shot.
\newblock {\em Appl. Math. Optim.}, 46(2-3):331--375, 2002.
\newblock Special issue dedicated to the memory of Jacques-Louis Lions.

\bibitem{UV:local}
G.~Uhlmann and A.~Vasy.
\newblock The inverse problem for the local geodesic ray transform.
\newblock {\em Inventiones Math}, 205:83-120 (2016). 

\bibitem{V} J. Vargo. A proof of lens rigidity in the category of analytic metrics, {\it Math. Research Letters}, {\bf 16} (2009), 1057--1069.

\bibitem{Vasy:Minicourse}
A.~Vasy.
\newblock A minicourse on microlocal analysis for wave propagation.
\newblock {\em Asymptotic Analysis in General Relativity}.
\newblock London Mathematical Society Lecture Note Series, Cambridge
University Press, 443:219--374, 2018.

\bibitem{Vasy-Dyatlov:Microlocal-Kerr}
A.~Vasy.
\newblock Microlocal analysis of asymptotically hyperbolic and {K}err-de
  {S}itter spaces.
\newblock {\em Inventiones Math.}, 194:381--513, 2013.
\newblock With an appendix by S. Dyatlov.

\bibitem{Wen_2015}
H.~Wen.
\newblock Simple {R}iemannian surfaces are scattering rigid.
\newblock {\em Geom. Topol.}, 19(4):2329--2357, 2015.


\bibitem{WZ}
E.~Wiechert and K.~Zoeppritz.
\newblock {\"U}ber {E}rdbebenwellen.
\newblock {\em Nachr. Koenigl. Geselschaft Wiss. G\"ottingen}, 4:415--549,
  1907.

\end{thebibliography}
\end{document}